\documentclass[11pt,twoside]{amsbook}
\usepackage[all]{xy}   
        \usepackage {amssymb,latexsym,amsthm,amsmath,xcolor}
        
        \numberwithin{equation}{section}
        \numberwithin{section}{chapter}

\long\def\symbolfootnote[#1]#2{\begingroup%
\def\thefootnote{\fnsymbol{footnote}}\footnote[#1]{#2}\endgroup}

\newcommand{\Q}{\ensuremath{\mathbb{Q}}}

\newcommand{\osum}{\ensuremath{\bigcirc\!\!\!\!\!\!\!\perp}}
\newcommand{\tr}{\ensuremath{{}^t\!}}
\newcommand{\tra}{\ensuremath{{}^t}}
\newcommand{\C}{\mathbb C}
\newcommand{\B}{\mathfrak B}

\newcommand{\GL}{\textup{GL}}
\newcommand{\SL}{\textup{SL}}
\newcommand{\Hom}{\textup{Hom}}
\newcommand{\End}{\textup{End}}
\newcommand{\diag}{\textup{diag}}
\newcommand{\gal}{\textup{Gal}}

\newcommand{\tensor}{\otimes}
\makeatletter
\def\imod#1{\allowbreak\mkern10mu({\operator@font mod}\,\,#1)}
\makeatother

\newtheorem{theorem}{\bf Theorem}[chapter]
\newtheorem{lemma}[theorem]{\bf Lemma}
\newtheorem{corollary}[theorem]{\bf Corollary}
\newtheorem{proposition}[theorem]{\bf Proposition}
\newtheorem{definition}[theorem]{\bf Definition}
\newtheorem*{theorem*}{Theorem}
\theoremstyle{definition}
\newtheorem{remark}[theorem]{\bf Remark}
\newtheorem{exercise}[theorem]{\bf Exercise}
\newtheorem{example}[theorem]{\bf Example}
\numberwithin{equation}{section}

\newcommand{\ignore}[1]{}

\newcommand{\mynote}[1]{}

\begin{document}
\setcounter{section}{0}
\title{\Huge{\color{blue} Representation Theory of Finite Groups}}
\author{\Large Dr. Anupam Singh \\ Professor (Mathematics, IISER Pune)\\ \vskip4cm {\large {\color{teal}Indian Institute of Science Education and Research (IISER) 
Pune, \\  Dr Homi Bhabha Road, Pashan, Pune 411008 INDIA \\ 
email : anupam@iiserpune.ac.in}\\\vskip4cm
Revised November 2022\\
First Draft December 2009}}
\date{}

\maketitle
\tableofcontents

\chapter{Introduction}
\section{Written in December 2022}

Since the first version of these notes, the author has received several feedback from students as well as teachers (two places, IIT Kanpur and IMSc Chennai, are worth acknowledging). Over the years the author got opportunities to teach a part of the material in various workshops and a full course at IISER Pune. The author expresses his gratitude to all students who were part of these exciting courses. The present version is revised with some more material added. The most important part on Induced Representation which was missing earlier is added now along with several more examples. Please let me know your feedback as and when you refer to these notes.

\section{Written in December 2009}
This is a class note for the course on the Representation Theory of Finite Groups taught by the author at IISER Pune to undergraduate students.  
We study the character theory of finite groups and illustrate how to get more information about groups. 
Burnside's theorem is one of the very good applications. 
It states that every group of order $p^aq^b$, where $p,q$ are distinct primes, is solvable.
We will always consider finite groups unless stated otherwise.
All vector spaces will be considered over general fields in the beginning but for the purpose of character theory, we 
assume the field is that of complex numbers.

We assume knowledge of the basic group theory and linear algebra. The point of view I projected to the students in the class is that
we have studied linear algebra hence we are familiar with the groups $\GL(V)$, the general linear 
group, or $\GL_n(k)$ in the matrix
notation. The idea of representation theory is to compare (via homomorphisms) finite (abstract) groups with these linear groups (somewhat concrete) 
and hope to a gain better understanding of them.

The students were asked to read about ``linear groups'' from the book by Alperin and Bell (mentioned in the bibliography) from
the chapter with the same title. We also revised, side-by-side in the class, Sylow's Theorem, Solvable groups and motivated ourselves for Burnside's $pq$-theorem.

The aim to start with an arbitrary field was to give the feeling that the theory is dependent on the base field and it gets considerably
complicated if we move away from characteristic $0$ algebraically closed field.
This we illustrate by giving an example of a higher dimensional irreducible representation of cyclic group over $\mathbb Q$ 
while all its irreducible representations are one dimensional over $\mathbb C$.
This puts things in perspective as to why we are doing the theory over $\mathbb C$ and motivates us 
to develop the ``Character Theory''.

\chapter*{}
\vskip10cm
\begin{center}
{\bf \Huge PART -- I}
\end{center}
\vskip5cm
This part contains basic definitions and some important foundational Theorems.
\chapter{Representation of a Group}
Let $G$ be a finite group.
Let $k$ be a field. We will assume that characteristic of $k$ is $0$, e.g., $k=\mathbb C$, $\mathbb 
R$ or $\mathbb Q$. Although often it would be enough to assume $char(k)\nmid |G|$. 
{\color{teal} 
\begin{definition}[Representation]
A {\bf representation of $G$ over $k$} is a homomorphism $\rho \colon G\rightarrow \GL(V)$ where 
$V$ 
is a vector space of finite dimension over the field $k$. 
The vector space $V$ is called {\bf representation space} of $G$, and its dimension is the {\bf 
dimension of representation}. 
\end{definition}}
\noindent Strictly speaking, the pair $(\rho, V) $ is called a representation of $G$ over the field 
$k$. 
However, if there is no confusion we would simply call either $\rho$ or $V$, a representation of 
$G$. 

Let us fix a basis $\{v_1,v_2,\ldots,v_n\}$ of $V$. Then, each $\rho(g)$ can be written in a matrix 
form with respect to this basis. This would give a map $\tilde\rho\colon G\rightarrow \GL_n(k)$ 
which is a group homomorphism. Thus, a representation of $G$ can also be thought of as such group 
homomorphism to $\GL_n(k)$. This idea is often useful in computing things as is the case in Linear 
Algebra.
{\color{teal} 
\begin{definition}[Invariant Subspace]
Let $\rho$ be a representation of $G$, and $W\subset V$ be a subspace. 
The space $W$ is said to be a {\bf $G$-invariant (or $G$-stable)} subspace if $\rho(g)(w)\in W$ 
$\forall w\in W$ and $\forall g\in G$.
\end{definition}}
\noindent Notice that once we have a $G$-invariant subspace $W$ we can restrict the representation 
to this subspace and define another representation $\rho_W \colon G\rightarrow \GL(W)$ where 
$\rho_W(g)=\rho(g)|_{W}$. Hence, $W$ is also called a {\bf subrepresentation}.

\begin{example}[Trivial Representation] Let $G$ be a group, and $k$ a field. Let $V$ be a vector 
space over $k$. 
Then $\rho(g)=1$ for all $g\in G$ is a representation. 
This is called trivial representation. In this case, every subspace of $V$ is an invariant subspace.
\end{example}

\begin{example}
Let $G=\mathbb Z/m\mathbb Z$ and $k=\mathbb C$. Let $V$ be a vector space of dimension $n$.
\begin{enumerate}
\item Suppose $\dim(V)=1$. Define $\rho_r\colon \mathbb Z/m\mathbb Z \rightarrow \mathbb C^*$ by 
$1\mapsto e^{\frac{2\pi i r}{m}}$ for $1\leq r \leq m-1$.
\item Define $\rho \colon \mathbb Z/m\mathbb Z \rightarrow \GL(V)$ by $1\mapsto T$ where $T\in 
\GL(V)$ satisfying $T^m=1$. For example, in the case $\dim(V)=2$ we can take $T= 
\diag\{e^{\frac{2\pi ir_1}{m}}, e^{\frac{2\pi ir_2}{m}}\}$. There is a general theorem
in Linear Algebra which says that any such matrix, over $\mathbb C$, is diagonalizable. 
\end{enumerate}
\end{example}

\begin{example}
Let $G=\mathbb Z/m\mathbb Z$ and $k=\mathbb R$. 
Let $V=\mathbb R^2$ with basis $\{e_1,e_2\}$. Then, we have representations of $\mathbb Z/m\mathbb 
Z$:
$$\rho_r \colon 1 \mapsto \left[\begin{matrix} \cos{\frac{2\pi r}{m}} & -\sin{\frac{2\pi r}{m}} \\ 
\sin{\frac{2\pi r}{m}} & \cos{\frac{2\pi r}{m}}\end{matrix}\right]$$
where $0\leq r\leq m-1$. Notice that we get $m$ distinct representations.
\end{example}
\begin{example}
Let $\phi \colon G\rightarrow H$ be  a group homomorphism. Let $\rho$ be a representation of $H$. 
Then $\rho\circ\phi$ is a representation of $G$.
\end{example}

\begin{example}\label{dihedralrep}
Let $G=D_m=\langle a,b \mid a^m=1=b^2, ab=ba^{m-1} \rangle$ be the dihedral group with $2m$ 
elements.
We have representations $\rho_r$ defined by:
$$a \mapsto \left[\begin{matrix} \cos{\frac{2\pi r}{m}} & -\sin{\frac{2\pi r}{m}} \\ 
\sin{\frac{2\pi r}{m}} & \cos{\frac{2\pi r}{m}}\end{matrix}\right], 
b \mapsto \left[\begin{matrix} 0 & 1 \\ 
1 & 0\end{matrix}\right].$$
Verify that these are group homomorphisms. Notice that the earlier representation of $\mathbb 
Z/m\mathbb Z$ is a restriction of this.
\end{example}

\begin{example}[Permutation Representation of $S_n$]\label{perm-rep} 
Let $S_n$ be the symmetric group on $n$ symbols, and $k$ be a field. 
Let $V=k^n$ with the standard basis $\{e_1,\ldots, e_n\}$. We define a representation $\rho$ of 
$S_n$ on $V$ by its action on the basis vectors as follows: 
$\rho(\sigma)(e_i)=e_{\sigma(i)}$ for $\sigma \in S_n$. Notice that while defining this 
representation we don't need to specify any field. That is if we write the matrix of this 
representation the entries are from $\mathbb Z$. 
\end{example}

\begin{example}[Group Action]
Let $G$ be a group and $k$ a field. Suppose $G$ is acting on a finite set $X$, i.e., we have a 
group action $G\times X \rightarrow X$. We denote by $k[X]=\{ f \mid f\colon X\rightarrow k \}$, the 
set of all maps. Clearly, $k[X]$ is a vector space of dimension $|X|$. The elements $e_x\colon X 
\rightarrow k$ defined by 
$e_x(x)=1$ and $e_x(y)=0$ if $x\neq y$, form a basis of $k[X]$. This action gives rise to a 
representation of $G$ on the space $k[X]$ as follows: $\rho \colon G\rightarrow \GL(k[X])$ given by 
$(\rho(g)(f))(x)=f(g^{-1}x)$ for $x\in X$.

If we take $G=S_n$ and $X=\{1, 2, \ldots, n\}$ we get back the Example~\ref{perm-rep}.
\end{example}

\begin{example}[Regular Representation] Let $G$ be a group of order $n$, and $k$ a field. Let 
$V=k[G]$ be the $n$-dimensional vector space with basis as elements of the group itself. We define 
$L\colon G\rightarrow \GL(k[G])$ by $L(g)(f)(h)=f(g^{-1}h)$, called 
the left regular representation. Similarly, $R(g)(f)(h)=f(hg)$, defines right regular 
representation $R$ of $G$. Prove that these representations are injective. Further, these 
representations are obtained by 
the action of $G$ on the
set $X=G$ by left multiplication or right multiplication. In fact, one can make $k[G]$ an algebra by 
defining the following multiplication:
$$(f*f')(t)=\sum_{x\in G} f(x)f'(x^{-1}t).$$
Note that this is the convolution multiplication (not the usual point-wise multiplication). This 
algebra $k[G]$ is called the {\bf group algebra} of $G$ over $k$. 
\end{example}

\begin{example}
Let $G=Q_8=\{\pm 1, \pm i, \pm j, \pm k\}$, and $k=\mathbb C$. We define a $2$-dimensional 
representation of $Q_8$ by:
$$i\mapsto \left[\begin{matrix} 0 & 1 \\ -1 & 0\end{matrix}\right], j\mapsto  \left[\begin{matrix} 0 
& i \\ i & 0\end{matrix}\right].$$
\end{example}

\begin{example}[Galois Theory]
Let $K=\mathbb Q(\theta)$ be a finite extension of $\mathbb Q$. Let $G=\gal(K/Q)$. We take $V=K$, a 
finite-dimensional vector space over $\mathbb Q$. We have a natural representation of $G$ as 
follows: $\rho \colon G \rightarrow \GL(K)$ defined by 
$\rho(g)(x)=g(x)$. Take $\theta=\zeta$, some $n$th root of unity, and show that the cyclic groups 
$\mathbb Z/m\mathbb Z$ have representations over the field $\mathbb Q$ of possibly dimensions more 
than $2$. This could be thought of as a reinterpretation of the statement of the Kronecker-Weber 
theorem.
\end{example}
{\color{teal}
\begin{definition}[Equivalence of Representations]
Let  $(\rho,V)$ and $(\rho',V')$ be two representations of $G$ over a field $k$. 
The representations  $(\rho,V)$ and $(\rho',V')$ are said to be {\bf $G$-equivalent (or 
equivalent)} if there exists a linear isomorphism $T\colon V\rightarrow V'$ such that 
$\rho'(g)=T\rho(g)T^{-1}$ for all $g\in G$. The following diagram helps in understanding this requirement: 
\[\xymatrix{V \ar[d]_{\rho(g)} \ar[r]^T & V' \ar[d]^{\rho'(g)} \\
V \ar[r]^T & V'
}\]
\end{definition}}
\noindent Let $\rho$ be a representation. Fix a basis of $V$, say $\{e_1,\ldots, e_n\}$. Then 
$\rho$ gives rise to a map $G\rightarrow
\GL_n(k)$ which is a group homomorphism. Notice that, if we change the basis of $V$ then we may get a different map for the same $\rho$. However, these would be equivalent as representations, i.e. 
these would differ by conjugation with respect to a fixed matrix (namely the base change matrix).

We have seen that a representation can have possibly a subrepresentation. This motivates us to define:
{\color{teal}
\begin{definition}[Irreducible Representation]
A representation $(\rho, V)$ of $G$ is called {\bf irreducible} if it has no proper $G$-invariant 
subspace, i.e., the only $G$-invariant subspaces are $0$ and $V$.
\end{definition}}

\begin{example} 
The trivial representation is irreducible if and only if it is one-dimensional.
\end{example}
\begin{example}\label{permutation} 
In the case of Permutation representation the subspace $W=<(1,1,\ldots,1)>$ and $W'=\{(x_1,\ldots 
,x_n) \mid \sum x_i = 0\}$ are two irreducible $S_n$ invariant subspaces. In fact, this 
representation is a direct sum of these two, and hence completely reducible.
\end{example}
\begin{example} 
One-dimensional representation is always irreducible. If $|G|\geq 2$ then the regular representation is not irreducible.
\end{example}

\begin{exercise} Let $G$ be a finite group. In the definition of a representation, let us not assume 
that the vector space $V$ is finite-dimensional. Prove that there exists a finite-dimensional 
$G$-invariant subspace of $V$. 
\end{exercise}
Hint: Fix $v\in V$ non-zero and take $W$ the subspace generated by $\rho(g)(v)~\forall g\in G$. 

\begin{exercise} A representation of dimension $1$ is a map $\rho \colon G\rightarrow k^*$. 
There are exactly two one-dimensional representations of $S_n$ over $\mathbb C$. There are exactly 
$n$ one-dimensional representations of the cyclic group $\mathbb Z/n\mathbb Z$ over $\mathbb C$.
\end{exercise}

\begin{exercise}
Prove that every finite group can be embedded inside symmetric group $S_n$, for some $n$ as well as 
a linear group $\GL_m$ for some $m$.
\end{exercise}
Hint: Make use of the regular representation. This representation is also called ``God-given'' 
representation. Later in the course, we will see why it's so.
\begin{exercise}
Is the above exercise true if we replace $S_n$ by $A_n$ and $\GL_m$ by $\SL_m$? 
\end{exercise}

\begin{exercise}
Prove that the cyclic group $\mathbb Z/p\mathbb Z$ has a representation of dimension $p-1$ over $\mathbb Q$.
\end{exercise}
Hint: Make use of the cyclotomic field extension $\mathbb Q(\zeta_p)$ and consider the map left 
multiplication by $\zeta_p$. This exercise shows that representation theory is deeply connected to 
the Galois Theory of field extensions.

\section{Commutator Subgroup and One Dimensional Representations}
Let $G$ be a finite group. Consider the set of elements $\{xyx^{-1}y^{-1}\mid x,y\in G\}$ and $G'$  the subgroup generated by this subset. This subgroup is called the {\bf commutator subgroup of $G$}.  We list some of the properties of this subgroup as an exercise here.
\begin{exercise}
\begin{enumerate}
\item $G'$ is a normal subgroup.
\item $G/G'$ is Abelian.
\item $G'$ is smallest subgroup of $G$ such that $G/G'$ is Abelian.
\item $G'=1$ if and only if $G$ is Abelian.
\item For $G=S_n$, $G'=A_n$; $G=D_n=<r,s \mid r^n=1=s^2, srs=r^{-1}>$ we have $G'=<r> \cong \mathbb Z/n\mathbb Z$ and $Q_8'=\mathcal Z(Q_8)$.
\end{enumerate}
\end{exercise}
Let $\widehat G$ be the set of all one-dimensional representations of $G$ over $\mathbb C$, i.e., the set of all group homomorphisms from $G$ to $\mathbb C^*$.
For $\chi_1,\chi_2 \in \widehat G$ we define multiplication by:
$$(\chi_1\chi_2)(g)=\chi_1(g)\chi_2(g).$$
\begin{exercise}
Prove that $\widehat G$ is an Abelian group.
\end{exercise}
We observe that for a $\chi\in \widehat G$ we have $G'\subset ker(\chi)$. Hence we can  prove,
\begin{exercise} 
Show that $\widehat{G} \cong G/G'$. 
\end{exercise}
\begin{exercise}
Calculate directly $\widehat G$ for $G=\mathbb Z/n\mathbb Z, S_n$ and $D_n$.
\end{exercise}

Let $G$ be a group. The group $G$ is called {\bf simple} if $G$ has no proper normal subgroup.
\begin{exercise}
\begin{enumerate}
\item Let $G$ be an Abelian simple group. Prove that $G$ is isomorphic to $\mathbb Z/p\mathbb Z$ where $p$ is a prime.
\item Let $G$ be a simple non-Abelian group. Then  $G=G'$.
\end{enumerate}
\end{exercise}

\chapter{Maschke's Theorem}
 
From the previous chapter, we recall,
{\color{teal}
\begin{definition}[Irreducible Representation]
A representation $(\rho,V)$ of $G$ is called {\bf irreducible} if it has no proper $G$-invariant 
subspace, i.e., only $G$-invariant subspaces are $0$ and $V$.
\end{definition}}

Let $(\rho,V)$ and $(\rho',V')$ be two representations of $G$ over a field $k$. We can define {\bf 
direct sum} of these two representations $(\rho\oplus\rho', V\oplus V')$ as follows: 
$\rho\oplus\rho' \colon G\rightarrow \GL(V \oplus V')$ given by 
$(\rho\oplus\rho')(g)(v,v') = (\rho(g)(v), \rho'(g)(v'))$ where $v\in V$ and $v'\in V'$. In the 
matrix notation, if we have two representations $\rho \colon G \rightarrow \GL_n(k)$ and 
$\rho'\colon 
G\rightarrow \GL_m(k)$, then $\rho\oplus\rho'$ is given (with respect to an appropriate basis) by 
$$g\mapsto \left[\begin{matrix} \rho(g) & 0 \\ 0 & \rho'(g)\end{matrix}\right].$$
This motivates us to look at those ``nice'' representations which can be obtained by taking the direct 
sum of irreducible ones.
{\color{teal}
\begin{definition}[Completely Reducible]
A representation $(\rho, V)$ of $G$ is called {\bf completely reducible} if it is a direct sum of 
irreducible ones. Equivalently, if $V = W_1 \oplus \ldots \oplus W_r$, where each $W_i$ is 
$G$-invariant irreducible representation.
\end{definition}}
\noindent This brings us to the following questions:
\begin{enumerate}
\item Is it true that every representation is a direct sum of irreducible ones?
\item How many irreducible representations are there for a given $G$ over $k$?
\end{enumerate}
The answer to the first question is affirmative for finite groups when $char(k)\nmid |G|$. 
This is Maschke's theorem proved below (it is also true for Compact Groups where it is called 
the Peter-Weyl Theorem). The other case when $char(k) \mid |G|$ comes under the subject of ``Modular 
Representation Theory'', and the answer to the first question is `No'. We will answer the second 
question over the field of complex numbers (and possibly over $\mathbb R$) in this course which is 
the subject of ``Character Theory''. For the theory over $\mathbb Q$, the subject is called 
`rationality questions' (refer to the book by Serre). 

 {\color{red}
\begin{theorem}[Maschke's Theorem]
Let $k$ be a field and $G$ be a finite group. Suppose $char(k)\nmid |G|$, i.e. $|G|$ is invertible in 
the field $k$. Let $(\rho, V)$ be a finite-dimensional representation of $G$. Let $W$ be a 
$G$-invariant subspace of $V$. Then, there exists $W'$, a $G$-invariant subspace, such that 
$V=W\oplus W'$.

Conversely, if $char(k)\mid |G|$ then there exists a representation, namely the regular 
representation, and a proper $G$-invariant subspace which does not have a $G$-invariant complement.
\end{theorem}}
\noindent We can re-state the above theorem as follows: Let $k$ be a field and $G$ be a finite group. 
Then, every finite-dimensional representation of $G$ over $k$ has the property that whenever it has 
a $G$-invariant subspace it has a $G$-invariant complement if and only if $char(k)\nmid |G|$. This 
gives us:

\begin{proposition}[Complete Reducibility]
Let $k$ be a field and $G$ a finite group with $char(k)\nmid |G|$. Then, every finite-dimensional 
representation of $G$ is completely reducible.
\end{proposition}

Now we are going to prove the above results. We need to recall the notion of projection from 
`Linear Algebra'. Let $V$ be a finite-dimensional vector space over the field $k$.
{\color{teal}
\begin{definition}
An endomorphisms $\pi\colon V\rightarrow V$ is called a {\bf projection} if $\pi^2=\pi$.
\end{definition}}
\noindent Some examples of projection are given as exercises below.

Let $W\subset V$ be a subspace. A subspace $W'$ is called a {\bf complement} of $W$ if $V = W\oplus 
W'$. Sometimes, we also say that $W$ has a {\bf direct summand}. It is a simple exercise in `Linear 
Algebra' to show that such a complement always exists (see exercise below), and there could be many 
of them.

\begin{lemma}
Let $\pi$ be an endomorphism. Then, $\pi$ is a projection if and only if there exists a 
decomposition $V = W \oplus W'$ such that $\pi(W)=0$, $\pi(W')=W'$, and $\pi$ restricted to $W'$ 
is identity.
\end{lemma}
\begin{proof}
Let $\pi\colon V\rightarrow V$ be such that $\pi(w, w') = w'$. Then, clearly $\pi^2=\pi$.

Now suppose $\pi$ is a projection. We claim that $V=ker(\pi)\oplus Im(\pi)$. Let $x\in ker(\pi)\cap Im(\pi)$. Then
there exists $y\in V$ such that $\pi(y)=x$ and $x=\pi(y)=\pi^2(y)=\pi(\pi(x))=\pi(0)=0$. Hence $ker(\pi)\cap Im(\pi)=0$.
Now, let $v\in V$. Then $v=(v-\pi(v))+\pi(v)$, and we see that $\pi(v)\in Im(\pi)$ and $v-\pi(v)\in 
ker(\pi)$ since 
$\pi(v-\pi(v))=\pi(v)-\pi^2(v)=0$. Now, let $x\in Im(\pi)$, say $x=\pi(y)$. Then 
$\pi(x)=\pi(\pi(y))=\pi(y)=x$. This shows
that $\pi$ restricted to $Im(\pi)$ is the identity map. 
\end{proof}

\begin{remark}
The reader familiar with `Jordan/Rational Canonical Form Theory' would recognise the following. The 
minimal polynomial of $\pi$ is $X(X-1)$ which is a product of distinct linear factors. Hence $\pi$ 
is a diagonalizable linear transformation with eigenvalues $0$ and $1$. Hence, there exists a basis 
so that the matrix of $\pi$ is $\diag\{0, \ldots, 0, 1, \ldots, 1\}$. This will give another proof 
of the above Lemma. 
\end{remark}
\begin{proof}[{\bf Proof of the Maschke's Theorem}]
Let $\rho \colon G\rightarrow \GL(V)$ be a representation. 
Let $W$ be a $G$-invariant subspace of $V$. Let $W_0$ be a (vector space) complement, i.e., 
$V=W_0\oplus W$. We have to produce a 
complement which is $G$-invariant. Let $\pi$ be a projection corresponding to this decomposition, 
i.e., $\pi(W_0)=0$, and $\pi(w)=w$ for all $w\in W$. We define an endomorphism $\pi'\colon 
V\rightarrow V$ by `averaging technique' as follows:
$$\pi'=\frac{1}{|G|}\sum_{t\in G}\rho(t)^{-1}\pi\rho(t).$$
We claim that $\pi'$ is a projection. We note that $\pi'(V)\subset W$ since $\pi\rho(t)(V)\subset W$ and $W$ is $G$-invariant.
In fact, $\pi'(w)=w$ for all $w\in W$ since $\pi'(w)= \frac{1}{|G|}\sum_{t\in G}\rho(t)^{-1}\pi\rho(t)(w)=
\frac{1}{|G|}\sum_{t\in G}\rho(t)^{-1}\pi(\rho(t)(w))=\frac{1}{|G|}\sum_{t\in G}\rho(t)^{-1}(\rho(t)(w))=w$ (note that 
$\rho(t)(w)\in W$ and $\pi$ takes it to itself). Let $v\in V$. Then $\pi'(v)\in W$. Hence $\pi'^2(v)=\pi'(\pi'(v))=\pi'(v)$ as 
we have $\pi'(v)\in W$ and $\pi'$ takes any element of $W$ to itself. Hence $\pi'^2=\pi'$.

Now we write decomposition of $V$ with respect to $\pi'$, say $V=W'\oplus W$ where $W'=ker(\pi')$ and $Im(\pi')=W$. 
We claim that $W'$ is $G$-invariant which will prove the theorem.
For this we observe that $\pi'$ is a $G$-invariant homomorphism, i.e., 
$\pi'(\rho(g)(v))=\rho(g)(\pi'(v))$ for all $g\in G$ and $v\in V$.
\begin{eqnarray*}
\pi'(\rho(g)(v)) &=& \frac{1}{|G|}\sum_{t\in G}\rho(t)^{-1}\pi\rho(t)(\rho(g)(v))\\
&=&\frac{1}{|G|}\sum_{t\in G}\rho(g)\rho(g)^{-1}\rho(t)^{-1}\pi\rho(t)\rho(g)(v)\\
&=& \rho(g)\frac{1}{|G|}\sum_{t\in G}\rho(tg)^{-1}\pi\rho(tg)(v)\\
&=&\rho(g)(\pi'(v)).
\end{eqnarray*}
This helps us to verify that $W'$ is $G$-invariant. Let $w'\in W'$. To show that $\rho(g)(w')\in W'$. 
For this we note that $\pi'(\rho(g)(w'))=\rho(g)(\pi'(w))=\rho(g)(0)=0$. This way we have produced 
$G$-invariant complement of $W$.

For the converse let $char(k)\mid |G|$. We take the regular representation $V=k[G]$. Consider 
$W=\left\{\sum_{g\in G}\alpha_gg\mid \sum\alpha_g=0\right\}$. We claim that $W$ is $G$-invariant but 
it has no $G$-invariant complement. This requires a bit of effort and we leave it to an interested 
reader to either work out herself or look up some textbook.
\end{proof}
\begin{remark}
In the proof of Maschke's theorem, one can start with a symmetric bilinear form and apply the trick 
of averaging to it. In that case, the complement will be the orthogonal subspace. 
Conceptually, I 
like that proof better however it requires familiarity with the bilinear form to be able to appreciate 
that proof. Later we will do that in some other context. An enthusiastic reader can work out the 
exercise below. 
\end{remark}

\begin{proof}[{\bf Proof of the Proposition (Complete Reducibility)}]
Let $\rho\colon G\rightarrow \GL(V)$ be a representation. 
We use induction on the dimension of $V$ to prove this result. Let $\dim(V)=1$. It is easy to verify that one-dimensional 
representation is always irreducible. Let $V$ be of dimension $n\geq 2$. If $V$ is irreducible we have nothing to prove.
So we may assume $V$ has a $G$-invariant proper subspace, say $W$ with $1\leq \dim(W)\leq n-1$. 
By Maschke's Theorem, we can write $V=W\oplus W'$ where $W'$ is also $G$-invariant. But now $\dim(W)$ and $\dim(W')$ both are 
less than $n$. By the induction hypothesis, they can be written as a direct sum of irreducible representations. This proves
the proposition.
\end{proof}

\begin{exercise}
Let $V$ be a finite-dimensional vector space.
Let $W\subset V$ be a subspace. Show that there exists a subspace $W'$ such that $V=W\oplus W'$. 
\end{exercise}
Hint: Start with a basis of $W$ and extend it to a basis of $V$.
\begin{exercise}
Show that when $V=W\oplus W'$ the map $\pi(w+w')=w$ is a projection map. Verify that $ker(\pi)=W'$ and $Im(\pi)=W$.
The map $\pi$ is called a projection on $W$. Notice that this map depends on the chosen complement $W'$.
\end{exercise}
\begin{exercise}
\begin{enumerate}
\item Complement of a subspace is not unique. Let us consider $V=\mathbb R^2$. Take a line $L$ passing through the origin. It
is a one-dimensional subspace. Prove that any other line is a complement.
\item Let $W$ be the one-dimensional subspace x-axis. Choose the complement space as y-axis and write down the projection map.
What if we chose the complement as the line $x=y$?
\end{enumerate}
\end{exercise}

The exercises below show that Maschke's theorem may not be true if we don't have a finite group.
\begin{exercise}
Let $V=\{(\ldots, a_{-1}, a_0, a_1, a_2,\ldots) \mid a_i\in \mathbb R\}$ be the space (a vector space of infinite dimension) and $G=\mathbb Z$. 
Define $\rho(1)$ to be the shift operator and $\rho(n)$ is obtained by composing $\rho(1)$ $n$-times. Show that this is a representation of $\mathbb Z$. It has an invariant subspace $<(\ldots,1,1,1,\ldots)>$. Is it completely reducible?
\end{exercise}

\begin{exercise} Consider a two dimensional representation of $\mathbb R$ as follows:
 
$$ a \mapsto  \left (\begin{array}{cc}
1 & a \\ 0 & 1
\end{array}  \right) .
$$
It leaves a one-dimensional subspace fixed generated by $(1,0)$ but it has no invariant complementary 
subspace. Hence this representation is not completely reducible.
\end{exercise}
\begin{exercise} Let $k=\mathbb Z/p\mathbb Z$. Consider a two-dimensional representation of the cyclic group $G=\mathbb Z/p \mathbb Z$
of order $p$ over $k$ of characteristic $p$ defined as in the previous example. 
Find a subspace to show that Maschke's theorem does not hold.
\end{exercise}

\begin{exercise}
Let $V=\mathbb C^n$ be the $n$-dimensional complex vector space. Let $G$ be a finite group. 
\begin{enumerate}
 \item Show that $\langle , \rangle \colon V \times V \rightarrow \mathbb C$ given by $\langle 
(z_1, \ldots, z_n), (w_1, \ldots, w_n) \rangle = \sum_{i=1}^n z_i \bar w_i$ is a non-degenerate 
inner product. (Non-degenerate amounts to showing $V^{\perp} =0$ or equivalently the matrix associated with the form is invertible). 
\item Let $\rho \colon G \rightarrow \GL(\mathbb C^n)$ be a representation. Show that $\mathcal H 
\colon V \times V \rightarrow \mathbb C$ defined by 
$$\mathcal H(z,w) = \frac{1}{|G|} \sum_{g\in G} \langle \rho(g)z, \rho(g)w\rangle$$
is an inner product.
\item Show that $\mathcal H$ is $G$-invariant, i.e., for each $x\in G$ we have $\mathcal H(z,w) = 
\mathcal H(\rho(x)z, \rho(x)w)$.
\item Show that if $W$ is a $G$-invariant subspace of $V$ then $W^{\perp}$ is also $G$-invariant, 
and $V=W\osum W^{\perp}$. This gives another proof of Maschke's theorem.
\end{enumerate}

\end{exercise}

\chapter{Schur's Lemma}

The broad question we would like to deal with is that can we classify all representations, up to 
some kind of equivalence (defined below). Further, in the wake of Maschke's Theorem, considering irreducible representations will be enough. 
{\color{teal}
\begin{definition}[$G$-map]
Let $(\rho,V)$ and $(\rho',V')$ be two representations of $G$ over field $k$. A linear map $T\colon V\rightarrow V'$ is called a {\bf $G$-map (between two representations)} if it satisfies the following:
$$\rho'(t)T=T\rho(t) \ \forall t\in G.$$
The following diagram helps in understanding this condition: 
\[\xymatrix{V \ar[d]_{\rho(t)} \ar[r]^T & V' \ar[d]^{\rho'(t)} \\
V \ar[r]^T & V'
}\]
\end{definition}}
\noindent The $G$-maps are also called {\bf intertwiners}.
\begin{exercise}
Prove that two representations of $G$ are equivalent if and only if there exists an invertible $G$-map.
\end{exercise}
\noindent In the case representations are irreducible the $G$-maps are easy to decide as follows: 
\begin{proposition}[Schur's Lemma]\label{schurlemma}
Let $(\rho,V)$ and $(\rho',V')$ be two irreducible representations of $G$ (of dimension $\geq 1$). 
Let $T\colon V\rightarrow V'$ be a $G$-map. Then, either $T=0$ or $T$ is an isomorphism. Moreover, 
if $T$ is non-zero then $T$ is an isomorphism if and only if the two representations are equivalent.
\end{proposition}
\begin{proof}
Let us consider the subspace $ker(T)$. 
We claim that it is a $G$-invariant subspace of $V$. 
For this let us take $v\in ker(T)$. 
Then $T\rho(t)(v)=\rho'(t)T(v)=0$ implies $\rho(t)(v)\in ker(T)$ for all $t\in G$. 
Since $V$ is irreducible we get, either $ker(T)=0$ or $ker(T)=V$. In the case $ker(T)=V$ the map $T=0$.

Hence we may assume $ker(T)=0$, i.e., $T$ is injective. 
Now we consider the subspace $Im(T)\subset V'$. 
We claim that it is also $G$-invariant. 
For this let $y=T(x)\in Im(T)$. 
Then $\rho'(t)(y)=\rho'(t)T(x)=T\rho(t)(x)\in Im(T)$ for any $t\in G$. 
Hence $Im(T)$ is $G$-invariant. 
Since $V'$ is irreducible we get either, $Im(T)=0$ or $Im(T)=V'$. Since $T$ is injective $Im(T)\neq 0$ and hence $Im(T)=V'$. This proves that in this case $T$ is an isomorphism.
\end{proof}
\begin{exercise}
Let $V$ be a vector space over $\mathbb C$, and $T\in\End(V)$ be a linear transformation. Show that 
there exists a one-dimensional subspace of $V$ left invariant by $T$. Show by example that this need 
not be true if the field is $\mathbb R$ instead of $\mathbb C$.
\end{exercise}
Hint: Show that $T$ has an eigenvalue then the corresponding eigenvector will do the job.
\begin{corollary}\label{corschur}
Let $(\rho,V)$ be an irreducible representation of $G$ over $\mathbb C$. Let $T\colon V\rightarrow 
V$ be a $G$-map. Then, $T=\lambda.Id$ for some $\lambda\in \mathbb C$ and $Id$ is the identity map 
on $V$. 
\end{corollary}
\begin{proof}
Let $\lambda$ be an eigen-value of $T$ corresponding to the eigen-vector $v\in V$, i.e., $T(v)=\lambda v$. Consider the subspace $W=ker(T-\lambda.Id)$. We claim that $W$ is a $G$-invariant subspace. Since $T$ and scalar multiplications are $G$-maps so is $T-\lambda$. Hence the kernel is $G$-invariant (as we verified in the proof of Schur's Lemma). One can do this directly also see the exercise below.

Since $W\neq 0$ and is $G$-invariant subspace of irreducible representation $V$, we get $W=V$. This gives $T=\lambda. Id$.
\end{proof}
Thus, if $(\rho, V)$ is an irreducible representation. Then, the set of all $G$-linear maps 
$\End_G(V)$ contains scalars $k\subset \End_G(V) \subset \End(V)$, and is a (finite-dimensional) 
division ring. In the case $k=\mathbb C$, $\End_G(V)$ is $\mathbb C$ the set of scalars only. 
\begin{exercise}
Let $T$ and $S$ be two $G$-maps. Show that $ker(T+S)$ is a $G$-invariant subspace.
\end{exercise}

\chapter{Representation Theory of Finite Abelian Groups over $\mathbb C$}\label{rep-abelian}
Throughout this chapter, $G$ denotes a finite Abelian group.
\begin{proposition}\label{abelian1dimension} 
Let $k=\mathbb C$ and $G$ be a finite Abelian group.
Let $(\rho,V)$ be an irreducible representation of $G$. Then, $\dim(V)=1$. 
\end{proposition}
\begin{proof}
The proof is a simple application of Schur's Lemma. We will break it down in the step-by-step exercises below.
\end{proof}
\begin{exercise}
With notation as in the proposition,
\begin{enumerate}
\item for $g\in G$ consider $\rho(g)\colon V \rightarrow V$. 
Prove that $\rho(g)$ is a $G$-map. (Hint: $ \rho(g)(\rho(h)(v))=\rho(gh)(v)=\rho(hg)(v)=\rho(h)(\rho(g)(v)). $)
\item Prove that there exists $\lambda$ (depending on $g$) in $\mathbb C$ such that  $\rho(g)=\lambda.Id$. (Hint: Use the corollary of Schur's Lemma.)
\item Prove that the map $\rho\colon G\rightarrow \GL(V)$ maps every element $g$ to a scalar map, 
i.e., it is given by $\rho(g)=\lambda_g.Id$ where $\lambda_g\in\mathbb C$.
\item Prove that the dimension of $V$ is $1$. (Hint: Take any one-dimensional subspace of $V$. It is $G$-invariant. Use Maschke's theorem on it as $V$ is irreducible.) 
\end{enumerate}
\end{exercise}
\begin{proposition}
Let $k=\mathbb C$ and $G$ be a finite Abelian group. Let $\rho\colon G\rightarrow \GL(V)$ be a 
representation of dimension $n$. Prove that we can choose a basis of $V$ such that $\rho(G)$ is 
contained in diagonal matrices.
\end{proposition} 
\begin{proof}
Since $V$ is a representation of a finite group we can use Maschke's theorem to write it as a direct sum of $G$-invariant irreducible ones, say $V=W_1\oplus\ldots\oplus W_r$.  Now using Schur's lemma we conclude that $\dim(W_i)=1$ for all $i$ and hence in turn we get $r=n$. By choosing a vector in each $W_i$ we get the required result. 
\end{proof}
\begin{corollary}\label{diagonal}
Let $G$ be a finite group (possibly non-commutative). Let $\rho\colon G\rightarrow \GL(V)$ be a 
representation. 
Let $g\in G$. Then there exists a basis of $V$ such that the matrix of $\rho(g)$ is diagonal.
\end{corollary}
\begin{proof}
Consider $H=<g>\subset G$ and $\rho\colon H\rightarrow \GL(V)$ the restriction map. Since $H$ is 
Abelian, using the above proposition,
we can simultaneously diagonalise elements of $H$. This proves the required result.
\end{proof}
\begin{remark}
In `Linear Algebra' we prove the following result: A commuting set of diagonalizable matrices over $\mathbb C$ can be simultaneously diagonalised. The proposition above is a version of the same result.
We also give a warning about the corollary above that if we have a finite subgroup $G$ of 
$\GL_n(\mathbb C)$ 
then we can take a conjugate of $G$ in such a way that a particular element becomes diagonal. 
\end{remark}

Now this leaves us the question to determine all irreducible representations of an Abelian group $G$. For this, we need to determine all group homomorphisms $\rho\colon G\rightarrow \mathbb C^*$.  
\begin{exercise}
Let $G$ be a finite group (not necessarily Abelian). Let $\chi\colon G\rightarrow \mathbb C^*$ be a group homomorphism. Prove that $|\chi(g)|=1$ and hence $\chi(g)$ is a root of unity. 
\end{exercise}

Let $\widehat G$ be the set of all group homomorphisms from $G$ to the multiplicative group $\mathbb C^*$. Let us also denote $\widehat {\widehat G}$ for the group homomorphisms from $\widehat G$ to $\mathbb C^*$. 
\begin{exercise}
With the notation as above,
\begin{enumerate}
\item Prove that for $G=G_1\times G_2$ we have $\widehat G \cong \widehat G_1\times \widehat G_2$.
\item Let $G=\mathbb Z/n\mathbb Z$. Prove that $\widehat G=\{\chi_k\mid 0\leq k\leq n-1\}$ is a group generated by $\chi_1$ of order $n$ 
where $\chi_1(r)=e^{\frac{2\pi i r}{n}}$ and $\chi_k=\chi_1^k$. Hence $\widehat{\mathbb Z/n\mathbb Z}\cong \mathbb Z/n\mathbb Z$.
\item Use the structure theorem of finite Abelian groups to prove  that $G\cong \widehat G$.
\item Prove that $G$ is naturally isomorphic to $\widehat{\widehat G}$ given by $g\mapsto e_g$ where $e_g(\chi)=\chi(g)$ for all $\chi\in G$.
\end{enumerate}
\end{exercise} 

\begin{exercise}[Fourier Transform]
For $f \in \mathbb C [\mathbb Z/n\mathbb Z]=\{f\mid f\colon \mathbb Z/n\mathbb Z \rightarrow \mathbb C\}$ we define $\hat{f} \in \mathbb C[\mathbb Z/n\mathbb Z]$ by,
$$ \hat{f} (q)= \frac{1}{n} \sum_{k=0}^{n-1} f(k) e(-kq) = \frac{1}{n} \sum_{k=0}^{n-1} f(k) \chi_q(-k).$$
Show that $ f(k)=\sum_{q=0}^{n-1} \hat{f} (q) e(kq) = \sum_{q=0}^{n-1} \hat{f} (q) \chi_q(k)$ and
$ \frac{1}{n}~\sum_{k=0}^{n-1} |f(k)|^2 = \sum_{q=0}^{n-1} | \hat{f} (q)|^2.$
\end{exercise}

\begin{exercise}
On $\mathbb C [\mathbb Z/n\mathbb Z]$ let us define an inner product by $\langle f, f'\rangle = \frac{1}{n}\sum_{j=0}^{n-1}
f(j)\bar f'(j)$ where bar denotes complex conjugation. Prove that $\{\chi_k \mid 0\leq k\leq n-1\}$ form an orthonormal basis
of $\mathbb C [\mathbb Z/n\mathbb Z]$. Let 
$$f=\sum_{\chi\in \widehat{ \mathbb Z/n\mathbb Z} }c_{\chi}\chi.$$ 
Calculate
the coefficients using the inner product and compare this with the previous exercise.
\end{exercise}

Now we show that the converse of the Proposition~\ref{abelian1dimension} is also true.
{\color{red}
\begin{theorem}\label{Abelian}
Let $G$ be a finite group. Every irreducible representation of $G$ over $\mathbb C$ is $1$ dimensional if and only if $G$ is
an Abelian group.
\end{theorem}}
\begin{proof}
Let all irreducible representations of $G$ over $\mathbb C$ be of dimension $1$. Consider the regular representation 
$\rho \colon G\rightarrow \GL(V)$ where $V=\mathbb C[G])$. We know that if $|G|\geq 2$ this 
representation is reducible 
and is an injective
map (also called faithful representation). Using Maschke's theorem we can write $V$ as a direct sum of irreducible ones and 
they are given to be of dimension $1$. Hence there exists a basis (check why?) $\{v_1,\ldots,v_n\}$ of $V$ such that the subspace
generated by each basis vector are invariant. Hence $\rho(G)$ consists of diagonal matrices with respect to this 
basis which is an Abelian group. Hence $G\cong \rho(G)$ is an Abelian group.
\end{proof}

\section{Example of Representation over $\mathbb Q$}

\subsection{An Irreducible Representation of $\mathbb Z/p\mathbb Z$}
Consider $G=\mathbb Z/p\mathbb Z$ where $p$ is an odd prime. Let $K=\mathbb Q(\zeta)$ where $\zeta$ is a primitive $p$th root of 
unity. Let us consider the left multiplication map $l_{\zeta}\colon K\rightarrow K$ given by $x\mapsto \zeta x$.  
Consider the basis $\{1,\zeta,\zeta^2,\ldots,\zeta^{p-2}\}$ of $K$. Then $l_{\zeta}(1)=\zeta, l_{\zeta}(\zeta^i)=\zeta^{i+1}$ for 
$1\leq i\leq p-1$ and $l_{\zeta}(\zeta^{p-2})=\zeta^{p-1}=-(1+\zeta+\zeta^2+\ldots+\zeta^{p-2})$ and the matrix of $l_\zeta$ is:
\[\left[\begin{matrix}
0&0&0&\cdots&0&-1\\
1&0&0&\cdots&0&-1\\
0&1&0&\cdots&0&-1\\
\vdots&&&\ddots&\vdots&\vdots\\
0&0&0&\cdots&1&-1
\end{matrix}
\right]\]
The map $1\mapsto l_{\zeta}$ defines a representation $\rho \colon G\rightarrow \GL_{p-1}(\mathbb 
Q)$. It is an irreducible 
representation of $G$.

\subsection{An Irreducible Representation of the Dihedral Group $D_{2p}$}

Notice that the Galois group $\gal(K/\mathbb Q)\cong \mathbb Z/(p-1)\mathbb Z$ comes with a natural representation on $K$. 
Let $\sigma \in \gal(K/\mathbb Q)$.
Then $\sigma$ is a $\mathbb Q$-linear map which gives representation $\gal(K/Q)\cong \mathbb 
Z/(p-1)\mathbb Z \rightarrow \GL_{p-1}(\mathbb Q)$.
If we consider slightly different basis of $K$, namely, $\{\zeta,\zeta^2,\ldots,\zeta^{p-1}\}$ then the matrix of each $\sigma$
is a permutation matrix. In fact this way $\gal(K/\mathbb Q) \hookrightarrow S_{p-1}$, the symmetric group. However, this representation
is not irreducible (the element $\zeta+\zeta^2+\cdots+\zeta^{p-1}$ is invariant and gives decomposition).

Notice that the Galois automomrphism $\sigma \colon K\rightarrow K$ given by $\zeta\mapsto \zeta^{-1}$ is an order $2$ element.
We claim that $\sigma$ normalizes $l_{\zeta}$, i.e., $\sigma l_{\zeta} \sigma = l_{\zeta^{-1}}$. 
Since $\sigma l_{\zeta}\sigma (\zeta^i)=\sigma l_{\zeta}(\zeta^{-i})=\sigma(\zeta^{1-i})=\zeta^{i-1}=l_{\zeta^{-1}}(\zeta^i)$.
Let us denote the subgroup generated by $\sigma$ as $H_1$ and the subgroup generated by $l_{\zeta}$ by $H_2$.
Then $H_1H_2$ is a group of order $2p$ where $H_2$ is a normal subgroup of order $p$. Hence $H_1H_2\cong D_{2p}$.
This gives representation of order $p-1$ of $D_{2p}=\langle r,s \mid r_p=1=s^2, srs=r^{-1} \rangle$ 
given by $D_{2p} \rightarrow \GL(K)$ such that
$r\mapsto l_{\zeta}$ and $s\mapsto \sigma$.

\begin{exercise}
Prove that the representation constructed above is irreducible.
\end{exercise}

\begin{exercise}
Write down the above representation concretely for $D_6$ and $D_{10}$.
\end{exercise}


\chapter{The Group Algebra $k[G]$}

Let $R$ be a ring (possibly non-commutative) with $1$. 
{\color{teal}
\begin{definition}
A (left) {\bf module} $M$ over a ring $R$ is an Abelian group $(M,+)$ with a map (called scalar multiplication) $R\times M \rightarrow M$ 
satisfying the following:
\begin{enumerate}
\item $(r_1+r_2)m=r_1m+r_2m$ for all $r_1,r_2\in R$ and $m\in M$.
\item $r(m_1+m_2)=rm_1+rm_2$ for all $r\in R$ and $m_1,m_2\in M$.
\item $r_1(r_2m)=(r_1r_2)m$ for all $r_1,r_2\in R$ and $m\in M$.
\item $1.m=m$ for all $m\in M$.
\end{enumerate}
\end{definition}}
Notice that this definition is the same as the definition of a vector space over a field. Analogous to definitions there we 
can define submodules and module homomorphisms.
\begin{example}
If $R=k$ (a field) or $D$ (a division ring) then the modules are nothing but vector spaces over $R$.
\end{example}
\begin{example}
Let $R$ be a PID (a commutative ring such as $\mathbb Z$ or polynomial ring $k[X]$ etc). Then $R\times R\ldots \times R$ and
$R/I$ for an ideal $I$ are modules over $R$. The structure theory of modules over PID states that any module is a direct 
sum of these kinds. However, over a non-PID things could be more complicated.
\end{example}
\begin{example}
In the non-commutative situation, the simple/semisimple rings are studied.
\end{example}
A module $M$ over a ring $R$ is called {\bf simple} if it has no proper submodules. And  a module $M$ is called 
{\bf semisimple} if every
submodule of $M$ has a direct complement. It is also equivalent to saying that $M$ is a direct sum of simple modules.

A ring $R$ is called {\bf semisimple} if every module over it is semisimple. And a ring $R$ is called {\bf simple} if it has no proper
two-sided ideal.

\begin{exercise}
\begin{enumerate}
\item Is $\mathbb Z$ a semisimple ring or simple ring?
\item When $\mathbb Z/n\mathbb Z$ a semisimple or simple ring?
\item Prove that the ring $M_n(D)$ where $D$ is a division ring is a simple ring and the module $D^n$ thought as one of the
columns of this ring is a module over this ring. 
\end{enumerate}
\end{exercise}

All of the representation theory definitions can be very neatly interpreted in module theory language. 
Given a field $k$ and a group $G$, we form the ring 
$$k[G]=\left\{\sum_{g\in G}\alpha_gg \mid \alpha_g\in k \right\}$$
called the {\bf group ring of $G$}. We can also define $k[G]=\{f\mid f\colon G\rightarrow k\}$.
We define following operations on $k[G]$:
$\left(\sum_{g\in G}\alpha_g g \right) + \left(\sum_{g\in G}\beta_g g \right) = \sum_{g\in G}\left(\alpha_g+\beta_g\right)g $,
$\lambda \left(\sum_{g\in G}\alpha_gg\right)  = \sum_{g\in G}(\lambda \alpha_g)g $
and the multiplication (recall the convolution definition) by 
$$ \left(\sum_{g\in G}\alpha_gg\right).\left(\sum_{g\in G}\beta_gg\right) =  \sum_{g\in G}\left(\sum_{t\in G}\alpha_t\beta_{t^{-1}g}\right)g   
= \sum_{g\in G}\left(\sum_{ts=g\in G}\alpha_t\beta_{s}\right)g. $$
With the above operations, $k[G]$ is an algebra called the {\bf group algebra} of $G$ (clearly it's a ring).
A representation $(\rho,V)$ for $G$ is equivalent to taking a $k[G]$-module $V$ (see the exercises below).

\begin{exercise}
Prove that $k[G]$ is a ring as well as a vector space of dimension $|G|$. In fact, it is a $k$-algebra.
\end{exercise}
\begin{exercise}
Let $k$ be a field. Let $G$ be a group. Then,
\begin{enumerate}
\item $(\rho, V)$ is a representation of $G$ if and only if $V$ is a $k[G]$-module.
\item $W$ is a $G$-invariant subspace of $V$ if and only if $W$ is a $k[G]$-submodule of $V$.
\item The representations $V$ and $V'$ are equivalent if and only if $V$ is isomorphic to $V'$ as $k[G]$-module. 
\item $V$ is irreducible if and only if $V$ is a simple $k[G]$-module.
\item $V$ is completely reducible if and only if $V$ is a semisimple module.
\end{enumerate}
\end{exercise}

We can rewrite Maschke's Theorem and Schur's Lemma in modules language:
{\color{red}
\begin{theorem}[Maschke's Theorem]
Let $G$ be a finite group and $k$ a field.
The ring $k[G]$ is semisimple if and only if $char(k)\nmid |G|$.
\end{theorem}}

\begin{proposition}[Schur's Lemma]
Let $M,M'$ be two non-isomorphic simple $R$ modules. 
Then $\Hom_R(M,M')=\{0\}$. Moreover, $\Hom_R(M,M)$ is a division ring.
\end{proposition}
\begin{proof}
The proof is left as an exercise.
\end{proof}
\begin{exercise}
Let $D$ be a finite-dimensional division algebra over $\mathbb C$ then $D=\mathbb C$.
\end{exercise}
\begin{exercise}
Let $R$ be a finite-dimensional algebra over $\mathbb C$ and $M$  a simple module over $R$. Suppose $M$ is a finite-dimensional 
over $\mathbb C$. Then $\Hom_R(M,M)\cong \mathbb C$.
\end{exercise}

\chapter{Constructing New Representations}

Here we will see how we can get new representations out of the known ones.
All of the representations are considered over any field $k$. Let $(\rho,V)$ be a representation of $G$. The {\bf character} of (corresponding to) representation $\rho$ is a map $\chi \colon 
G\rightarrow k$ is defined by $\chi(t) = tr(\rho(t))$ where $tr$ is the trace of the corresponding 
matrix. We will deal with ``character theory'' in the following chapters.

\section{Subrepresentation of a Representation}
Suppose we have a representation $(\rho, V)$ of $G$, i.e., we have $\rho\colon G\rightarrow 
\GL(V)$. 
Let $W$ be a $G$-invariant subspace of $V$. Then, we can define a subrepresentation $(\tilde\rho, 
W)$ of $G$ i.e., $\tilde\rho \colon G \rightarrow \GL(W)$ by $\tilde\rho(t)(w)=\rho(t)(w)$. In 
matrix 
notation, if we choose a basis of $W$ first and extend it to a basis of $V$ then 
$\rho(g)=\begin{pmatrix} A(g) & B(g) \\ & C(g)\end{pmatrix}$ and $\tilde\rho(g) = A(g)$.

\section{Sum of Representations}
Let $(\rho,V)$ and $(\rho',V')$ be two representations of the group $G$ of dimension $n$ and $n'$ respectively. 
We define $(\rho\oplus\rho', V\oplus V')$ by $(\rho\oplus\rho')(t)(v,v')=(t(v),t(v'))$.  
This is called the sum of the two representations. 
In matrix notation if the representations are $\rho\colon G\rightarrow \GL_n(k)$ and $\rho' \colon 
G\rightarrow \GL_{n'}(k)$ then $(\rho\oplus \rho')(t) = \begin{pmatrix} \rho(t)& \\ 
&\rho'(t)\end{pmatrix}$. 
\begin{exercise}
If $\chi$ and $ \chi'$ are characters of $\rho$ and $ \rho'$ respectively then the character of $\rho\oplus \rho'$ is $\chi+\chi'$. 
\end{exercise}
 
If $\rho$ is a representation and $V=W\oplus W'$ is a $G$-invariant decomposition then $\rho$ is a 
sum of its two subrepresentations. Clearly, the idea of the sum of two representations defined here can 
be extended to a finite sum.

\section{Adjoint Representation}
Let $V$ be a vector space over $k$ of dimension $n$. Let us recall some basic linear algebra. A 
linear map $f\colon V\rightarrow k$ is called a {\bf linear functional}. 
We denote $V^*=\Hom_k(V,k)$, the set of all linear functionals. We define operations on $V^*$ by $(f_1+f_2)(v)=f_1(v)+f_2(v)$ and $(\lambda f)(v)=
\lambda f(v)$ and it becomes a vector space. The vector space $V^*$ is called the dual space of $V$.

\begin{exercise}
Let $V$ be a vector space over $k$ with a basis $\{e_1, \ldots, e_n\}$. With the notation as above,
\begin{enumerate}
\item Check that $V^*$ is a vector space with basis $e_i^*$ which are defined by $e_i^*(e_j)=\delta_{ij}$. 
Hence it has the same dimension as $V$. After fixing a basis of $V$ we can obtain a basis of $V^*$ 
this way which is called a {\bf dual basis} with respect to the given one.
\item Show that $V$ is ``naturally'' isomorphic (by a map which is defined without the requirement of 
a basis and is an isomorphism) to $V^{**}$.
\end{enumerate}
\end{exercise}
\noindent Let $T\colon V\rightarrow V$ be a linear map. We define a map $T^*\colon V^*\rightarrow 
V^*$ by $T^*(f)(v)=f(T(v))$ represented in diagram as follows:
\[\xymatrix{V \ar[rd]_{T^*(f)} \ar[r]^T & V \ar[d]^{f} \\
 & k.
}\]
The map $T^*$ is called {\bf dual or adjoint} of $T$. 
\begin{exercise}
Fix a basis of $V$ and consider the dual basis of $V^*$ with respect to that (as in the previous 
exercise). Let $A=(a_{ij})$ be the matrix of $T$. Show that the matrix of $T^*$ with respect to the 
dual basis is $\tr A =(a_{ji})$, the transpose matrix.
\end{exercise}

Now, with the knowledge of the dual/adjoint, we can define adjoint representation. Let $\rho\colon 
G\rightarrow \GL(V)$ be a representation. We define the {\bf adjoint representation} $(\rho^*, 
V^*)$ 
as follows: $\rho^* \colon G\rightarrow \GL(V^*)$ where $\rho^*(g)=\rho(g^{-1})^*$ (the inverse in 
this definition is forced to make $\rho^*$ a group homomorphism).
In the matrix form if we have a representation $\tau \colon G\rightarrow \GL_n(k)$ then 
$\tau^*\colon G\rightarrow \GL_n(k)$ is given by $\tau^*(g)=\tra \tau(g)^{-1}$. 
\begin{exercise}
With the notation as above,
\begin{enumerate}
\item Show that $\rho^*$ is a representation of $G$ of the same dimension as $\rho$.
\item Fix a basis and suppose $\tau$ is the matrix form of $\rho$ then show that the matrix form of $\rho^*$ is $\tau^*$.
\item Prove that if $\rho$ is irreducible then so is $\rho^*$. 
\end{enumerate} 
\end{exercise}
The last exercise will become easier in the case of complex representations once we define characters as we will have a simpler criterion to test when a representation is irreducible. 
\begin{exercise}
When $k=\mathbb C$, we get that the character of $\rho^*$, $\chi^*=\bar \chi$. Further, $\chi$ is 
irreducible if and only if $\chi^*$ is irreducible (Over $\mathbb C$, this would be easier to prove 
using orthogonality relations).
\end{exercise}

\section{Restriction of a Representation}
Let $(\rho,V)$ be a representation of the group $G$. 
Let $H$ be a subgroup.  
Then $(\rho,V)$ is a representation of $H$ also obtained by composing $H\hookrightarrow G 
\rightarrow \GL(V)$ denoted as $\rho_H$ or $\rho|_H$ or $Res(\rho)$.

\section{Lift of a Representation}
Let $N$ be a normal subgroup of $G$. Thus, we have the quotient map $\pi\colon G\rightarrow G/N$. 
Then any representation of $G/N$ gives rise to a representation of $G$: 
$$G\stackrel{\pi} {\rightarrow} G/N \stackrel{\rho}{\rightarrow} \GL(V).$$ 
Moreover, if the representation of $G/N$ is irreducible then the representation of $G$ remains 
irreducible. We use this technique to obtain $1$-dimensional representations by taking $N=[G,G]$.

\section{Tensor Product of two Representations}\label{tensor}

There is a general method to construct the tensor product of two $R$ modules $M$ and $M'$, namely $M\otimes_R M'$ which is again an $R$-module. It can be done even when $R$ is not commutative. The tensor product is an $R$-module denoted as $M\otimes_R M'$ which is defined by the following universal property:
The module $(T, \phi)$ where $\phi\colon M\times M' \rightarrow T$ is a bilinear map, 
is said to be the tensor product (to be denoted as $M\otimes M'$) if for any $R$-module $P$ and 
any bilinear map $\beta\colon M\times M' \rightarrow P$ there exists a unique linear map $\tilde 
\beta$ such that the following diagram commutes: 
\[\xymatrix{M\times M' \ar[r]^\beta \ar[rd]_{\phi} & P  \\
& M\otimes M'\ar[u]_{\tilde\beta}
}\]
i.e., $\beta = \tilde\beta\phi$. At the moment we need this only for vector spaces which is easier 
to deal with.

Let $V$ and $V'$ be two vector spaces over $k$. First, we define the tensor product of two vector spaces as follows. Consider $\mathcal V = \displaystyle\bigoplus_{(v,v')\in V\oplus V'}k{(v,v')}$ (where $(v,v')$ serve as a basis) and the subspace $\mathcal W$ spanned by elements $(v_1+v_2, v') - (v_1, v')-(v_2,v')$,  $(v, v_1'+ v_2') - (v, v_1')-(v,v_2')$, $a(v,v')-(av,v')$ and $a(v,v')-(v, av')$ for all $v, v_1, v_2\in V$, $v', v_1', v_2'\in V'$ and $a\in k$. Then the tensor product denoted as $V\otimes V'$ is given by the quotient space $\mathcal V/\mathcal W$. Most of the time we don't need to worry about the construction but rather get used to the idea of using it. Thus a more practical way is as follows.

Tensor product of $V$ and $V'$ is a vector space $V\otimes V'=\left\{\sum_{i=1}^r v_i\otimes v_i'\mid v_i\in V, v_i'\in V'\right\}$ with the following properties:
\begin{itemize}
\item $\left(\sum_{i=1}^r v_i\otimes v_i'\right) + \left(\sum_{i=1}^s w_i\otimes w_i'\right)=  v_1\otimes v_1'+\cdots + v_r\otimes v_r' + 
w_1\otimes w_1'+ \cdots + w_s\otimes w_s'.$
\item $(v_1+v_2)\otimes v'=v_1\otimes v' + v_2\otimes v'$ and $v\otimes(v_1'+v_2')=v\otimes v_1' + v\otimes v_2'$.
\item $\lambda \left(\sum_{i=1}^r v_i\otimes v_i'\right) = \sum_{i=1}^r \lambda v_i\otimes v_i' = \sum_{i=1}^r v_i\otimes \lambda v_i'.$
\end{itemize}
Let $\{e_1,\ldots,e_n\}$ be a basis of $V$ and $\{e_1',\ldots,e_m'\}$ be that of $V'$. Then $\{e_i\otimes e_j' \mid 1\leq i\leq n, 1\leq j\leq m\}$
is a basis of the vector space $V\otimes V'$ hence the dimension of $V\tensor V'$ is $nm$.

\noindent{\bf Warning: } Elements of $V\otimes V'$ are not exactly of kind $v\otimes v'$ but they are supposed to be a finite sum of these ones!

\begin{exercise}
For any vector space $W$ and given a bilinear map $\phi \colon V\times V' \rightarrow W$ there exists a linear map $\tilde \phi \colon V\otimes V' \rightarrow W$ such that $\phi(v,v') = \tilde\phi(v\otimes v')$. Thus, $Bil_k(V\times V', ?) = End_k(V\otimes V', ?)$. 
\end{exercise}
Given linear maps $T\in End(V)$ and $S\in End(V')$ we can define a linear map $T\otimes S\in End(V\otimes V')$ as follows: $(T\otimes S)(\sum v\otimes v') = \sum T(v)\otimes S(v')$.
\begin{exercise}
Given the matrix of $T$ and $S$ compute the matrix of $T\otimes S$. If the matrix of $T$ with respect to the basis $\{e_i\}$ is $A=(a_{ij})$ and the matrix of $S$ with respect to the basis $\{e_i'\}$ is $B=(b_{ij})$ then the matrix of $T\otimes S$ with respect to the basis $\{e_r\otimes e_s'\}$ will turn out to be $(a_{ij}B)$. 
\end{exercise}
\begin{exercise}
Show that $Trace(T\otimes S)= Trace(T)Trace(S)$.
\end{exercise}

With the definition of tensor product in hand, we can define new representations as follows. Let $(\rho,V)$ and $(\rho',V')$ be two representations of the group $G$ of dimension $n$ and $n'$ respectively. 
We define a representation $(\rho\tensor\rho', V\tensor V')$ of $G$ as follows: 
$$(\rho\tensor\rho')(t)(\sum v\tensor v')=\sum(\rho(t)(v)\tensor \rho'(t)(v')).$$
\begin{exercise} 
Choose a basis of $V$ and $V'$. Let $A=(a_{ij})$ be the matrix of $\rho(t)$ and $B=(b_{lm})$ be that of $\rho'(t)$.
What is the matrix of $(\rho\tensor\rho')(t)$? 
\end{exercise}
\begin{exercise}
If $\chi$ and $ \chi'$ are characters of $\rho$ and $ \rho'$ respectively then the character of $\rho\otimes \rho'$ is $\chi\chi'$. 
\end{exercise}

If  $(\rho,V)$ is a representation of $G$ then $V^{\tensor d}=V\otimes\cdots\otimes V, Sym^d{V}, \wedge^{d}(V)$ are also representations of $G$. 
This way starting from one representation we can get many representations. 
Though, even if you start from an irreducible representation the above-constructed representations need not be irreducible (for example $V\otimes V$ given above) but often they contain other  irreducible representations.
Writing down the direct sum decomposition of tensor representations is an important topic of study. 
Often it happens (for semisimple Lie algebras) that we need a much smaller number of representations (called {\bf fundamental representations}) of which tensor products contain all irreducible representations (called highest weight theory in the case of semisimple Lie algebras).

\begin{exercise}
Let $(\rho,V)$ be a representation of $G$ and $\chi$ be its character. Show that the character of 
$\rho^{\otimes m}$ is $\chi^m$.  
\end{exercise}

\begin{exercise}
Can we decompose $V^{\otimes n}$ in terms of $Sym^i(V)$ and $\wedge^j(V)$? 
\end{exercise}

\begin{exercise}
Consider the group $\SL_2(\mathbb R)$. This group has a natural representation on $2$-dimensional 
real vector space $V_2$. We write this, say with basis $\{x,y\}$, as follows: for $g\in 
\SL_2(\mathbb 
R)$ write $g=\begin{pmatrix} a& b \\c&d\end{pmatrix}$ then, $$g.x=ax+cy \ \ {\rm and}\ \ 
g.y=bx+dy.$$ 
Now consider a real vector space $V_{d+1}$ for $d\geq 1$ with basis 
$$\{x^d, x^{d-1}y, x^{d-2}y^2,\ldots, xy^{d-1}, y^d\}.$$
We define a representation $\rho_{d+1}\colon \SL_2(\mathbb R) \rightarrow \GL(V_{d+1})$ as follows: 
$$\rho_{d+1}(g)(x^iy^j):=(g.x)^i(g.y)^j$$
where the right-hand side is expanded as polynomials. Show that $\rho_{d+1}$ is an irreducible 
representation. Further, $\rho_{d+1}$ is nothing by symmetric power representation $sym^{d}(V_2)$. 
Thus, $V_2$ is an example of fundamental representation in this case. 
\end{exercise}


\section{Decomposition of the Representation $V\otimes V$}\label{sym-rep}

Let $G$ be a finite group. Suppose that $(\rho,V)$ and $(\rho',V')$ are representations of $G$.
Then we can get a new representation of $G$ from these  representations defined as follows (recall 
from Section~\ref{tensor}):
\begin{eqnarray*}
  \rho\otimes \rho' \colon G &\rightarrow & \GL(V \otimes V') \\
  (\rho\otimes \rho')(g)(v \otimes v') &=& \rho(g) (v) \otimes \rho'(g) (v')
\end{eqnarray*}
Now suppose $\rho$ and $\rho'$ are representations over $\mathbb C$ and $\chi$ and $\chi'$ are the 
corresponding characters, then the character of $\rho \otimes \rho'$ is $\chi\chi'$ given by 
$(\chi\chi')(g)=\chi(g)\chi'(g)$. However, even if $\rho$ and $\rho'$ are irreducible $\rho \otimes 
\rho'$ need not be irreducible.

Let $(\rho,V)$ be a representation of $G$. We consider $(\rho \otimes \rho, V\otimes V)$. 
As defined above it is a representation of $G$ with character $\chi^2$ where $\chi^2(g)=\chi(g)^2$. 
Recall from section~\ref{tensor} it can be decomposed (as a vector space) as   
$$V \otimes V = Sym^2(V) \oplus \wedge^2(V).$$
We prove below that both $Sym^2(V)$ and $\wedge^2(V)$ are $G$-spaces.
{\color{red}
\begin{theorem}
Let $(\rho,V)$ be a $\mathbb C$-representation of $G$. Then, $V \otimes V = Sym^2(V) \oplus \wedge^2(V)$ where each of the subspaces $Sym^2(V) $ and $\wedge^2(V)$ are $G$-invariant.
  \end{theorem}}
 \begin{proof}
  Let $\{v_1,\ldots,v_n\}$ be a basis of $V$. Then $\{v_i \otimes v_j \mid 1\leq i,j \leq n\}$ is a 
basis of $V \otimes V$ with dimension $n^2$. Consider the linear map $\theta$ defined on the basis of $V \otimes V$ by   $\theta(v_i \otimes v_j) = v_j \otimes v_i $ and extended linearly.
 We observe that $\theta$ can also be defined without the help of any basis by $\theta(v\otimes 
w)=w\otimes v$ since if $v = \sum a_i v_i$ and $w =\sum b_j v_j$, then
 $   \theta(v \otimes w) = \theta\left(\sum a_i v_i \otimes \sum b_j v_j\right) =
 \sum a_i b_j \theta(v_i \otimes v_j) = \sum a_i b_j (v_j \otimes v_i)=
 \sum (b_j v_j \otimes a_i v_i)  = (w \otimes v)$.
Observe that $\theta^2=1$ and we take the subspaces of $V \otimes V$ corresponding to eigenvalues 
$1$ and $-1$:
\begin{eqnarray*}
  Sym^2(V) &=& \{x \in V \otimes V \mid  \theta(x) = x\} \\
  \wedge^2(V) &=& \{x \in V \otimes V \mid \theta(x) = -x\}
\end{eqnarray*}

\noindent Note that  $\{(v_i \otimes v_j + v_j \otimes v_i) \mid 1\leq i \leq j \leq n\}  $   is a 
basis for $Sym^2(V)$ and hence its dimension is $ \frac{n(n+1)}{2}$ and  $\{(v_i \otimes v_j - v_j 
\otimes v_i) \mid 1\leq i < j \leq n\}$ is a basis for $\wedge^2(V)$ and hence dimension is $ 
\frac{n(n-1)}{2}$.

 We claim that $Sym^2(V)$ and $\wedge^2(V)$ both are $G$-invariant. 
 Suppose that $v \otimes w \in Sym^2(V)$ and $g \in G$ then we have
 \begin{eqnarray*}
  \theta (\rho(g) (v \otimes w))  &=& \theta\left(\rho(g)\left(\sum \lambda_{ij}(v_i \otimes v_j + 
v_j \otimes v_i)\right)\right) \\
    &=& \theta\left (\sum \lambda_{ij}(\rho(g) v_i \otimes \rho(g) v_j +  \rho(g) v_j \otimes 
\rho(g) v_i)\right) \\
    &=&  \sum \lambda_{ij}\left[\theta(\rho(g) v_i \otimes \rho(g) v_j) + \theta( \rho(g) v_j 
\otimes \rho(g) v_i)\right]\\
    &=& \sum \lambda_{ij}[\rho(g) v_j \otimes \rho(g) v_i +  \rho(g) v_i \otimes \rho(g) v_j)]\\
    &=& \rho(g)\left(\sum \lambda_{ij}(v_j \otimes v_i + v_i \otimes v_j)\right)\\
    &=& \rho(g) (v \otimes w).
 \end{eqnarray*}
We also see that  $Sym^2(V) \cap \wedge^2(V) = \{0\}$. Also, their individual dimensions add up to 
$\frac{n(n+1)}{2} + \frac{n(n-1)}{2} = n^2 = \dim(V \otimes V)$, hence we have  $V \otimes V = 
Sym^2(V) \oplus \wedge^2(V)$ as $G$-spaces.
  \end{proof}
{\color{red}  
\begin{theorem} The characters of $Sym^2(V)$ and $\wedge^2(V)$ are $\chi_S$ and $\chi_A$ 
respectively given by
  \begin{eqnarray*}
    \chi_S (g) &=& \frac{1}{2}\left (\chi^2 (g)+\chi(g^2)\right) \\
    \chi_A (g) &=& \frac{1}{2}\left (\chi^2 (g)-\chi(g^2)\right).
  \end{eqnarray*}
  \end{theorem}}
  \begin{proof}
Suppose that $|G|=d$. Then for any $g \in G$, $(\rho(g))^d = I$. Thus, $m(X)$, the minimal 
polynomial of $\rho(g)$, divides the polynomial $p(X) = X^d - 1$. Since $p(X)$ has distinct roots so 
will $m(X)$ and hence $\rho(g)$ is diagonalisable.

Let $\{e_1,\cdots e_n\}$ be an eigenbasis for $V$ and let $\{\lambda_1,\cdots,\lambda_n\}$ be the corresponding eigenvalues. Then, from the proof of the previous theorem, it follows that $\{(e_i \otimes 
e_j - e_j \otimes e_i) \mid i<j\}$ is an eigenbasis for $\wedge^2(V)$ with corresponding eigenvalues $\{\lambda_i \lambda_j \mid i<j\}$.
 We now have
$$    \chi_A (g) = Tr(\rho \otimes \rho)(g) = \sum_{i<j}\lambda_i \lambda_j = 
\frac{1}{2}\left(\left(\sum \lambda_i\right)^2-\sum \left(\lambda_i ^2\right)\right)   = \frac{1}{2} 
\left(\chi^2 (g)-\chi(g^2)\right).$$
Now, we can calculate $\chi_S$ as $\chi^2=\chi_S +\chi_A$.
  \end{proof}

\section{Induced Representation}
See the chapter~\ref{ind-rep}.

\chapter*{}
\vskip10cm
\begin{center}
{\bf \Huge PART -- II}
\end{center}
\vskip5cm
This part deals with the Character Theory where we mostly work with representations over $\mathbb C$.

\chapter{Matrix Elements}
Now onwards we assume that the field $k=\C$. 
We also denote $\mathbb S_1=\{\alpha\in \mathbb C \mid |\alpha|=1\}$.
Let $G$ be a finite group. Then, $\mathbb C[G] = \{f\mid f\colon G\rightarrow \mathbb C\}$ is a 
vector space of dimension $|G|$.  
Let $f_1,f_2$ be two functions from $G$ to $\mathbb C$, i.e., $f_1,f_2\in \mathbb C[G]$. 
We define a map $(\ ,\ )\colon \mathbb C[G] \times \mathbb C[G] \rightarrow \mathbb C$ as follows, 
$$
(f_1,f_2)=\frac{1}{|G|}\sum_{t\in G}f_1(t)f_2(t^{-1}).
$$
Note that $(\ ,\ )$ is a symmetric bilinear form.

Let $\rho \colon G\rightarrow \GL(V)$ be a representation. We can choose a basis and get a map in 
matrix form $\rho \colon G\rightarrow \GL_n(\mathbb C)$ where $n$ is the dimension of the 
representation.
This means we have,
$$
g\mapsto \left[\begin{matrix} a_{11}(g) & a_{12}(g)&\cdots& a_{1n}(g)\\ a_{21}(g) & a_{22}(g)&\cdots& a_{2n}(g)\\
\vdots & \vdots &\ddots &\vdots\\
a_{n1}(g) & a_{n2}(g)&\cdots& a_{nn}(g) 
\end{matrix}\right]
$$
where the matrix entries are $a_{ij}\colon G\rightarrow \mathbb C$, i.e, $a_{ij}\in \mathbb C[G]$.
The maps $a_{ij}$'s are called {\bf matrix elements} of $\rho$.
Thus, to a representation $\rho$ we can associate a subspace $\mathcal W$ of $\mathbb C[G]$ spanned 
by $a_{ij}$. It can be shown that this subspace does not depend on the chosen basis.
In what follows, we will explore relations between these subspaces $\mathcal W$ associated to 
irreducible representations of a finite group $G$.
Let $\rho_1, \ldots, \rho_r, \cdots$ be irreducible representations of $G$ of dimension $n_1,\cdots, 
n_r, \cdots$ respectively.
We don't know yet whether there are finitely many irreducible representations which we will prove later.
Let $\mathcal W_1,\cdots, \mathcal W_r, \cdots$ be associated subspaces of $\mathbb C[G]$ to the irreducible representations.
{\color{red}
\begin{theorem}\label{matrixelements}
Let $(\rho, V)$ and $(\rho', V')$ be two non-equivalent irreducible representations of $G$ of dimension $n$ and $n'$ respectively. 
Let $a_{ij}$ and $b_{ij}$ be the corresponding matrix elements with respect to some fixed basis of 
$V$ and $V'$ respectively. Then,
\begin{enumerate}
\item $(a_{il},b_{mj})=0$ for all $i,j,l,m$.
\item $(a_{il},a_{mj}) = \frac{1}{n}\delta_{ij}\delta_{lm} = 
\begin{cases} \frac{1}{n} & \mbox{if $i=j$ and $l=m$}\\
0 &\mbox{otherwise}. 
\end{cases}$
\end{enumerate}
\end{theorem}}
\begin{proof}
Let $T\colon V\rightarrow V'$ be a linear map. 
Define 
$$
T^0=\frac{1}{|G|}\sum_{t\in G}\rho(t)T (\rho'(t))^{-1}.
$$ 
Then, $T^0$ is a $G$-linear map. Using Schur's Lemma~\ref{schurlemma}, we get $T^0=0$.
Let us denote the matrix of $T$ by $x_{lm}$.
Then, $ij$th entry of $T^0$ is zero for all $i$ and $j$, i.e.,
$$ \frac{1}{|G|} \sum_{t\in G} \sum_{l,m} a_{il}(t)x_{lm}b_{mj}(t^{-1}) = 0.$$
Since $T$ is arbitrary linear transformation the entries $x_{lm}$ are arbitrary complex number hence 
can be treated as indeterminate.
Hence, coefficients of $x_{lm}$ are $0$.
This gives $\frac{1}{|G|} \sum_{t\in G} a_{il}(t)b_{mj}(t^{-1}) = 0$ hence $(a_{il}, b_{mj})=0$ for all $i,j,l,m$.

Now, let us consider $T \colon V\rightarrow V$, a linear map.
Again define $ T^0=\frac{1}{|G|}\sum_{t\in G}\rho(t)T (\rho(t))^{-1}$ which is a $G$-map. From 
Corollary~\ref{corschur}, we get that
$T^0=\lambda. Id$ where $\lambda = \frac{1}{n} tr(T)=\frac{1}{n}\sum_{l}x_{ll}= \frac{1}{n}\sum_{l,m}x_{lm}\delta_{lm}$  since $n.\lambda = tr(T^0) = \frac{1}{|G|}\sum_{t\in G} tr(T)=tr(T)$.
Now using matrix elements we can write $ij$th term of $T^0$:
$$\frac{1}{|G|} \sum_{t\in G}\sum_{l,m} a_{il}(t)x_{lm}a_{mj}(t^{-1}) = \lambda \delta_{ij} = 
\frac{1}{n}\sum_{l,m}x_{lm}\delta_{lm}\delta_{ij}.$$
Again $T$ is an arbitrary linear map so its matrix elements $x_{lm}$ can be treated as indeterminate. Comparing coefficients
of $x_{lm}$ we get:
$$
\frac{1}{|G|} \sum_{t\in G} a_{il}(t)a_{mj}(t^{-1}) =  \frac{1}{n}\delta_{lm}\delta_{ij}.
$$
Which gives $(a_{il},a_{mj})= \frac{1}{n}\delta_{lm}\delta_{ij}$.
\end{proof}

\begin{corollary}
If $f\in \mathcal W_i$ and $f'\in\mathcal W_j$ with $i\neq j$ then $(f,f')=0$, i.e., $(\mathcal 
W_i,\mathcal W_j)=0$.
\end{corollary}

\begin{exercise}
Prove that, in both cases, $T^0$ is a $G$-map.
\end{exercise}

\begin{exercise}
Let $a_{ij}$ and $a_{ij}'$ be the set of two different matrix elements of the representation 
$\rho$ with respect to a different basis. Then, there exists a liner transformation $P$, namely the 
base change, such that $[a_{ij}(g)] = P[a_{ij}'(g)]P^{-1}$. Using this we can show that the 
subspace $<a_{ij}> = <a_{ij}'> \subset \mathbb C[G]$.   
\end{exercise}

\chapter{Character Theory}\label{chartheory}

We have $\mathbb C[G]$, space of all complex-valued functions on $G$ which is a vector space of dimension $|G|$. 
We define an {\bf inner product} $\langle\ ,\ \rangle \colon \mathbb C[G] \times \mathbb C[G] 
\rightarrow \mathbb C$ by 
$$
\langle f_1,f_2 \rangle = \frac{1}{|G|}\sum_{t\in G}f_1(t)\overline{f_2(t)}.
$$
\begin{exercise}
With the notation above,
\begin{enumerate}
\item Prove that $\langle\ ,\ \rangle$ is an inner product (non-degenerate) on $\mathbb C[G]$.
\item If $f_1$ and $f_2$ take value in $\mathbb S_1\subset \mathbb C$ then $\langle f_1,f_2\rangle = 
(f_1, f_2\circ \iota)$ where $\iota \colon g\mapsto g^{-1}$.
\end{enumerate}
\end{exercise}

{\color{teal}
\begin{definition}[Character of a Representation]
Let $(\rho,V)$ be a representation of $G$. 
The {\bf character} of (corresponding to) a representation $\rho$ is a map $\chi \colon 
G\rightarrow \mathbb C$ defined by $\chi(t) = tr(\rho(t))$ where $tr$ is the trace of corresponding 
matrix. The character of an irreducible representation is said to be {\bf irreducible character}.
\end{definition}}
\noindent Strictly speaking, it is $\chi_{\rho}$ but for the simplicity of notation we write $\chi$ 
only when it is clear which representation it corresponds to.
\begin{exercise}
Let $A,B\in \GL_n(k)$. Prove the following:
\begin{enumerate}
\item $tr(AB)=tr(BA)$.
\item $tr(A) =tr(BAB^{-1})$.
\end{enumerate}
\end{exercise}

\begin{exercise}
Usually, we define the trace of a matrix. Show that in the above definition of character it is a well-defined function. That is, prove that $\chi(t)$ doesn't change if we choose a different basis and 
calculate the trace of $\rho(t)$. This is another way to say that trace is an invariant of the 
conjugacy classes of $\GL_n(k)$.
\end{exercise}

\begin{exercise}
If $\rho$ and $\rho'$ are two isomorphic representations, i.e., they are $G$-equivalent, then the 
corresponding characters are the same. The converse of this statement is also true which we will prove 
later.
\end{exercise}

\begin{proposition}\label{character}
If $\rho$ is a representation of dimension $n$, and $\chi$ is the corresponding character then,
\begin{enumerate}
\item $\chi(1)=n$, the dimension of the representation.
\item $\chi(t^{-1})=\overline{\chi(t)}$ for all $t\in G$ where $\ \bar{}\ $ denotes the complex 
conjugation.
\item $\chi(tst^{-1})=\chi(s)$ for all $t, s\in G$, i.e., character is constant on the 
conjugacy classes of $G$.
\item For $f\in \mathbb C[G]$ we have $(f,\chi)=\langle f,\chi \rangle$.
\end{enumerate}
\end{proposition}
\begin{proof}
For the proof of part two, we use Corollary~\ref{diagonal} to calculate $\chi(t)$. From that corollary every $\rho(t)$ 
 can be diagonalised (at a time not simultaneously which is enough for our purposes), say $\diag\{\omega_1,\cdots,\omega_n\}$.
Since $t$ is of finite order, say $d$, we have $\rho(t)^d=1$. That is each $\omega_j^d=1$ means the diagonal elements
are $d$th root of unity. We know that roots of unity satisfy $\omega_j^{-1}=\overline{\omega_j}$. Hence 
$\chi(t^{-1})=tr(\rho(t)^{-1})=\omega_1^{-1}+\cdots+\omega_n^{-1}= \overline{\omega_1}+\cdots+\overline{\omega_n}=\overline{\chi(t)}$.

We can also prove this result by the upper triangulation theorem, i.e., every matrix $\rho(t)$ can be 
conjugated to an upper triangular matrix. Combining this with the fact that the trace is the same for the 
conjugate matrices, gives the result.
\end{proof}

{\color{teal}
\begin{definition}[Class Function]
A function $f \colon G\rightarrow \mathbb C$ is called a {\bf class function} if $f$ is constant on 
the conjugacy classes of $G$. We denote the set of class functions on $G$ by $\mathcal H$. 
\end{definition}}
\begin{exercise}

With the notation as above,
\begin{enumerate}
\item Prove that $\mathcal H$ is a subspace of  $\mathbb C[G]$.
\item The dimension of $\mathcal H$ is the number of conjugacy classes of $G$.
\item Let $c_g=\sum_{x\in G} xgx^{-1} \in \mathbb C[G]$. The centre of the group algebra $\mathbb C[G]$ is spanned by $c_g$.
\item Let $\chi$ be a character corresponding to some representation of $G$. Then, $\chi\in \mathcal H$.
\end{enumerate}
\end{exercise}

\begin{proposition}
Let $(\rho,V)$ and $(\rho',V')$ be two representations of the group $G$ and $\chi, \chi'$ be the corresponding characters. Then,
\begin{enumerate}
\item The character of the sum of two representations is equal to the sum of characters, i.e., $\chi_{\rho\oplus \rho'}=\chi+\chi'$.
\item The character of the tensor product of two representations is the product of two characters, i.e., $\chi_{\rho\otimes\rho'}=\chi\chi'$.
\end{enumerate}
\end{proposition}
\begin{proof}
The proof is a simple exercise involving matrices.
\end{proof}
This way we can define the sum and product of characters which is again a character.
\chapter{Orthogonality Relations}
 Let $G$ be a finite group. Let $W_1, W_2, \ldots, W_h, \ldots$ be irreducible representations of $G$ of dimension $n_1, n_2, \ldots, n_h,\ldots$
 over $\mathbb C$. Let $\chi_1, \chi_2,\ldots,\chi_h,\ldots$ are corresponding characters, called {\bf irreducible characters of $G$}.
 We will fix this notation from now onwards. We will prove that the number of irreducible characters and hence the number of
 irreducible representations are finite and equal to the number of conjugacy classes.
 
In the last chapter, we introduced an inner product $\langle,\rangle$ on $\mathbb C[G]$. We also observed that character of
any representation belongs to $\mathcal H$, the space of class functions. 
{\color{red}
\begin{theorem}\label{charorthonormal} 
The set of irreducible characters $\{\chi_1,\chi_2,\ldots\}$ form an orthonormal set of $(\mathbb C[G],\langle,\rangle)$. 
That is, 
\begin{enumerate}
\item If $\chi$ is a character of an irreducible representation then $\langle \chi,\chi\rangle = 1$.
\item If $\chi$ and $\chi'$ are two irreducible characters of non-isomorphic representations then $\langle \chi,\chi'\rangle =0$.
\end{enumerate}
\end{theorem}}
\begin{proof}
Let $\rho$ and $\rho'$ be non-isomorphic irreducible representations and $(a_{ij})$ and $(b_{ij})$ be the corresponding 
matrix elements. Then $\chi(g)=\sum_{i}a_{ii}(g)$ and $\chi'(g)=\sum_{j}b_{jj}(g)$.
Then $\langle \chi,\chi'\rangle = (\chi,\chi') = (\sum_{i}a_{ii}, \sum_{j}b_{jj}) = \sum_{i,j}(a_{ii},b_{jj}) = 0$
from Theorem~\ref{matrixelements} and Proposition~\ref{character} part 4.
Using the similar argument we get $\langle \chi,\chi\rangle = (\chi,\chi) = (\sum_{i}a_{ii}, \sum_{j}b_{jj}) = 
\sum_{i} (a_{ii},b_{ii}) = \sum_{i=1}^{n}\frac{1}{n}=1$ where $n$ is the dimension of the representation $\rho$.
\end{proof}
\begin{corollary}
\begin{enumerate}
 \item The set of irreducible characters form a linearly independent subset of $\mathbb C[G]$.
 \item The number of irreducible characters is finite.
 \end{enumerate}
\end{corollary}
\begin{proof}
Since the irreducible characters form an orthonormal set they are linearly independent. Hence their number has to be less than
the dimension of $\mathbb C[G]$ which is $|G|$, hence finite.
\end{proof}
\noindent Once we prove that two representations are isomorphic if and only if their characters are 
the same this corollary will also give that there are finitely many non-isomorphic irreducible 
representations. We prove this for the irreducible representations first. 

\begin{proposition}
Let $(\rho, V)$ and $(\rho', V')$ be two irreducible representations with characters $\chi$ and 
$\chi'$ respectively. Then, $(\rho, V)$ and $(\rho', V')$ are equivalent if and only if 
$\chi=\chi'$.
\end{proposition}
\begin{proof} Let $\chi=\chi'$ and we need to show $V$ and $V'$ are $G$-equivalent. On 
the contrary, suppose $V$ and $V'$ are not equivalent. Let $a_{ij}$ and $b_{ij}$ be the corresponding 
matrix elements of $\rho$ and $\rho'$ respectively. Then, $\chi=\sum_{i} a_{ii}$ and 
$\chi'=\sum_{j} 
b_{jj}$. Now, from Theorem~\ref{matrixelements} $\langle \chi, \chi' \rangle= \langle \sum a_{ii}, 
\sum b_{jj} \rangle = \sum_{i, j} \langle a_{ii}, b_{jj} \rangle =0 $. However, $\langle \chi, 
\chi' \rangle = \langle \chi, \chi \rangle = 1$, a contradiction. Thus, $V$ and $V'$ must be 
$G$-equivalent. 
\end{proof}

We are going to use the above results to analyse the general representation of $G$ and identify its irreducible components.
{\color{red}
\begin{theorem} \label{irrcomponents}
Let $(\rho,V)$ be a representation of $G$ with character $\chi$. 
Let $V$ decompose into a direct sum of irreducible representations:
$$
V=V_1\oplus\cdots \oplus V_m.
$$ 
Then the number of $V_i$ isomorphic to $W_j$ (a fixed irreducible representation) is equal to the scalar product 
$\langle \chi,\chi_j\rangle$.
\end{theorem}}
\begin{proof}
Let $\phi_1,\ldots,\phi_m$ be the characters of $V_1,\ldots,V_m$.
Then $\chi=\phi_1+\cdots+\phi_m$.
Also $\langle \chi,\chi_j\rangle = \sum_{i}\langle \phi_i,\chi_j\rangle = \sum_{\phi_i=\chi_j} \langle \phi_i,\chi_j\rangle =$ the number
of $V_i$ isomorphic to $W_j$.
\end{proof}
\begin{corollary}
With the notation as above,
\begin{enumerate}
\item The number of $V_i$ isomorphic to a fixed $W_j$ does not depend on the chosen decomposition.
\item Let $(\rho, V)$ and $(\rho', V')$ be two representations with characters $\chi$ and $\chi'$ respectively. 
Then $V\cong V'$ if and only if $\chi=\chi'$. 
\end{enumerate}
\end{corollary}
\begin{proof}
The proof of part 1 is clear from the theorem above.
For the proof of part 2, it is clear that if $V\cong V'$ we get $\chi=\chi'$.
Now suppose $\chi=\chi'$. Let $V\cong W_1^{n_1}\oplus\cdots\oplus W_h^{n_h}$ and $V'\cong W_1^{m_1}\oplus\cdots\oplus W_h^{m_h}$
be the decomposition as a direct sum of irreducible representations (we can do this using Maschke's Theorem) where $n_i,m_j\geq 0$.
Suppose $\chi_1,\chi_2,\ldots,\chi_h$ be the irreducible characters of $W_1,\ldots,W_h$.
Then $\chi=n_1\chi_1+\cdots+n_h\chi_h$ and $\chi'=m_1\chi_1+\cdots+m_h\chi_h$. However as $\chi_i$'s form an orthonormal set
they are linearly independent. Hence $\chi=\chi'$ implies $n_i=m_i$ for all $i$. Hence $V\cong V'$. 
\end{proof}
From this corollary, it follows that the number of irreducible representations is the same as the number of irreducible characters
which is less than or equal to $|G|$. In fact, later we will prove that this number is equal to the number of conjugacy classes.
The above analysis also helps to identify whether a representation is irreducible by use of the following:
{\color{red}
\begin{theorem}[Irreducibility Criteria]
Let $\chi$ be the character of a representation $(\rho,V)$.
Then $\langle \chi,\chi\rangle$ is a positive integer and  $\langle \chi,\chi\rangle =1$ if and only if $V$ is irreducible.
\end{theorem}}
\begin{proof}
Let $V\cong W_1^{n_1}\oplus\cdots\oplus W_h^{n_h}$. Then $\chi=n_1\chi_1+\cdots+n_h\chi_h$ and 
$$
\langle\chi,\chi\rangle = \langle n_1\chi_1+\cdots+n_h\chi_h, n_1\chi_1+\cdots+n_h\chi_h \rangle = \sum_i n_i^2.
$$
Hence $\langle \chi,\chi\rangle =1$ if and only if one of the $n_i=1$, i.e, $\chi=\chi_i$ for some $i$.
Hence the result.
\end{proof}
\begin{exercise}
Let $\chi$ be an irreducible character. Show that $\overline{\chi}$ is so. 
Hence a representation $\rho$ is irreducible if and only if $\rho^*$ is so.
\end{exercise}
\begin{exercise}
Use the above criteria to show that the ``Permutations representation'', and ``Regular Representation'' (defined in the second chapter) are not irreducible if $|G|>1$.
\end{exercise}
\begin{exercise}
Let $\rho$ be an irreducible representation and $\tau$ be a $1$-dimensional representation. Show that $\rho\otimes \tau$ is irreducible.
\end{exercise}

\chapter{Main Theorem of Character Theory}
Now we deal with the main question of how many irreducible representations are there for a given 
group $G$.

\section{Regular Representation}
Let $G$ be a finite group and $\chi_1,\ldots,\chi_h$ be the irreducible characters of dimension $n_1,\ldots,n_h$ respectively. 
Let $L$ be the left regular representation of $G$ with the corresponding character $l$.
\begin{exercise}
The character $l$ of the regular representation is given by $l(1)=|G|$ and $l(t)=0$ for all $1\neq t\in G$.
\end{exercise}
{\color{red}
\begin{theorem}
Every irreducible representation $W_i$ of $G$ is contained in the regular representation 
with multiplicity equal to the dimension of $W_i$ (which is denoted by $n_i$).
Hence, $l=n_1\chi_1+\cdots+n_h\chi_h$.
\end{theorem}}
\begin{proof}
In the view of Theorem~\ref{irrcomponents} the number of times $\chi_i$ is contained in the regular 
representation is given by $\langle l,\chi_i\rangle = 
\frac{1}{|G|}l(1)\chi_i(1)=\frac{1}{|G|}|G|n_i=n_i$. This proves the required result.
\end{proof}
If we are asked to construct irreducible representations of a finite group we don't know where to 
look for them. This theorem ensures a natural place, namely the regular representation, where we 
can find all of them. For this reason, one may also call it a ``God Given Representation''. Some 
further properties are listed here as an exercise.

\begin{exercise}
With the notation as above,
\begin{enumerate}
\item The degree $n_i$ satisfy $\sum_{i=1}^{h}n_i^2=|G|$. 
\item If $1\neq s\in G$ we have $\sum_{i=1}^{h}n_i\chi_i(s)=0$. 
\end{enumerate}
\end{exercise}
Hints: This follows from the formula for $l$ as in the theorem by evaluating $s=1$ and $s\neq 1$. 

These relations among characters will be useful to determine the character table of the group $G$.

\section{The Number of Irreducible Representations}

Now we will prove the main theorem of the character theory.
{\color{red}
\begin{theorem}[Main Theorem]
The number of irreducible representations of $G$ (up to isomorphism) is equal to the number of conjugacy classes of $G$.
\end{theorem}}

The proof of this theorem will follow from the following proposition.
We know that the irreducible characters $\chi_1,\ldots,\chi_h\in \mathcal H$ and form an orthonormal 
set (see Theorem~\ref{charorthonormal}). We prove that, in fact, they form an orthonormal basis of $\mathcal H$ and generate
as an algebra whole of $\mathbb C[G]$.
\begin{proposition}\label{hbasis}
The irreducible characters of $G$ form an orthonormal basis of $\mathcal H$, the space of class functions.
\end{proposition}
To prove this we will make use of the following:
\begin{lemma}
Let $f\in\mathcal H$ be a class function on $G$.
Let $(\rho,V)$ be an irreducible representation of $G$ of degree $n$ with character $\chi$. 
Let us define $\rho_f=\sum_{t\in G}f(t)\rho(t) \in \End(V)$. 
Then, $\rho_f=\lambda .Id$ where $\lambda=\frac{|G|}{n}\langle f,\overline{\chi}\rangle$.
\end{lemma}
\begin{proof}
We claim that $\rho_f$ is a $G$-map and use Schur's Lemma to prove the result.
For any $g\in G$ we have,
$$
\rho(g)\rho_f\rho(g^{-1})= \sum_{t\in G} f(t) \rho(g)\rho(t)\rho(g^{-1})
= \sum_{t\in G} f(t) \rho(gtg^{-1}) = \sum_{s\in G} f(g^{-1}sg)\rho(s) =\rho_f.
$$
Hence $\rho_f$ is a $G$-map. From Schur's Lemma (see~\ref{corschur}) we get that $\rho_f=\lambda.Id$ for some $\lambda\in\mathbb C$.
Now we calculate the trace of both sides:
$$
\lambda.n=tr(\rho_f)=\sum_{t\in G}f(t)tr(\rho(t)) = \sum_{t\in G}f(t)\chi(t)=
   |G|\frac{1}{|G|} \sum_{t\in G}f(t)\overline{\overline{\chi(t)}} =|G|\langle f,\overline{\chi} \rangle.
$$
Hence we get $\lambda = \frac{|G|}{n}\langle f, \overline{\chi}\rangle$.
\end{proof}

\begin{proof}[{\bf Proof of the Proposition~\ref{hbasis}}]
We need to prove that irreducible characters span $\mathcal H$ as being orthonormal they are already linearly independent.
Let $f\in \mathcal H$. Suppose $f$ is orthogonal to each irreducible $\chi_i$, i.e., $\langle f,\chi_i\rangle =0$ for all $i$.
Then we will prove $f=0$.

Since $\langle f,\chi_i\rangle =0$ it implies $\langle f,\overline{\chi_i}\rangle =0$ for all $i$.
Let $\rho_i$ be the corresponding irreducible representation. Then from the previous lemma 
${\rho_i}_f=\left(\frac{|G|}{n}\langle f, \overline{\chi_i}\rangle\right).Id=0$ for all $i$.
Now let $\rho$ be any representation of $G$. From Maschke's Theorem, it is a direct sum of $\rho_i$'s, the irreducible ones.
Hence $\rho_f=0$ for any $\rho$.

In particular we can take the regular representation $L \colon G\rightarrow \GL(\mathbb C[G])$ for 
$\rho$ and we get 
$L_f=0$. Hence $L_f(e_{1})=0$ implies $\sum_{t\in G}f(t)L(t)(e_1)=0$, i.e., $\sum_{t\in G}f(t)e_{t}=0$ hence $f(t)=0$ for
all $t\in G$. Hence $f=0$.
\end{proof}

\begin{proposition}
Let $s\in G$ and let $r_s$ be the number of elements in the conjugacy class of $s$.
Let $\{\chi_1,\ldots,\chi_h\}$ be irreducible representations of $G$.
Then,
\begin{enumerate}
\item We have $\sum_{i=1}^{h}\chi_i(s)\overline{\chi_i(s)}=\frac{|G|}{r_s}$.
\item For $t\in G$ not conjugate to $s$, we have $\sum_{i=1}^{h}\chi_i(s)\overline{\chi_i(t)}=0$.
\end{enumerate}
\end{proposition}
\begin{proof}
Let us define a class function $f_t$ by $f_t(t)=1$ and $f_t(g)=0$ if $g$ is not conjugate to $t$. Since irreducible characters
span $\mathcal H$ (see~\ref{hbasis}) we can write $f_t=\sum_{i=1}^{h}\lambda_i\chi_i$ where 
$\lambda_i=\langle f_t, \chi_i\rangle = \frac{1}{|G|}r_t\overline{\chi_i(t)}$. Hence 
$f_t(s) = \frac{r_t}{|G|} \sum_{i=1}^h\overline{\chi_i(t)} \chi_i(s)$ for any $s\in G$. This gives the required result by
taking $s$ conjugate to $t$ and not conjugate to $t$.
\end{proof}

\section{Artin-Wedderburn Decomposition}
There is a more conceptual proof of the main theorem mentioned in the last section. It follows from a general
decomposition theorem on semisimple algebras called Artin-Wedderburn decomposition. We prove this 
in the case $k=\mathbb C$. 
{\color{red}
\begin{theorem}[Artin-Wedderburn]
 Let $G$ be a finite group and $k$ a field such that $char(k)\nmid |G|$. Then 
 $$k[G]\cong M_{n_1}(D_1)\times \ldots M_{n_r}(D_r)$$
 where $D_i$ are division algebras over $k$.
 In the case $k=\mathbb C$, $\mathbb C[G]\cong M_{n_1}(\mathbb C)\times \ldots \times M_{n_r}(\mathbb C)$.
\end{theorem}}
\begin{proof}
 We prove for the case $k=\mathbb C$. We have $\mathbb C[G]\cong W_1^{n_1}\oplus\cdots\oplus W_h^{n_h}$ as
 $\mathbb C[G]$-module. We compute $\Hom_{\mathbb C[G]}(\mathbb C[G], \mathbb C[G])$ in two different ways.
 If we think of $\mathbb C[G]$ as a ring (possibly non-commutative) then $\Hom_{\mathbb C[G]}(\mathbb C[G], \mathbb C[G]) \cong
 \mathbb C[G]^{opp}$. 
On the other hand 
\begin{eqnarray*}
\Hom_{\mathbb C[G]}(\mathbb C[G], \mathbb C[G]) &\cong& \Hom_{\mathbb C[G]}(W_1^{n_1}\oplus\cdots\oplus W_h^{n_h},W_1^{n_1}\oplus\cdots\oplus W_h^{n_h})\\
&\cong& \bigoplus_{i=1}^{h} \Hom_{\mathbb C[G]}(W_i^{n_i}, W_i^{n_i}) \\
&\cong& \bigoplus M_{n_i}(\End_{\mathbb C[G]}(W_i))\\
&\cong& \bigoplus M_{n_i}(\mathbb C).
\end{eqnarray*}
This gives us $\mathbb C[G]^{opp}\cong \oplus_{i=1}^h M_{n_i}(\mathbb C)$.
 \end{proof}
 \begin{exercise}
 Let $R$ be a ring (possibly non-commutative) then $\End_R(R)\cong R^{opp}$, the opposite ring.
\end{exercise}
\begin{exercise}
 Prove the main theorem using the Artin-Wedderburn decomposition by computing the dimension of the centre on
 both sides.
\end{exercise}
\begin{exercise}
 There are no finite dimensional division algebras over $\mathbb C$.
\end{exercise}
\begin{exercise}
 Verify the steps of the proof in the above theorem using Schur's Lemma.
\end{exercise}

\begin{exercise}
 Find the Artin Wedderburn decomposition of a cyclic group $\mathbb Z/n\mathbb Z$ over fields $\mathbb C$, 
 $\mathbb R$ and $\mathbb Q$.
\end{exercise}

\begin{exercise}
 Show that $\mathbb R[Q_8]\cong \mathbb R\times \mathbb R\times \mathbb R\times \mathbb R\times \mathbb H$ 
 where $\mathbb H$ is the real quaternion algebra.
\end{exercise}

\begin{exercise}
 Show that $\mathbb R[D_4]\cong \mathbb R\times \mathbb R\times \mathbb R\times \mathbb R\times M_2(\mathbb R)$.
\end{exercise}

\begin{exercise}
 Show that $\mathbb C[Q_8]\cong \mathbb C[D_4]\cong \mathbb C^4\times M_2(\mathbb C)$.
\end{exercise}

\begin{exercise}
 What is the decomposition of $\mathbb Q[S_3]$?
\end{exercise}

\section{Exercises on Character Table}
Character table can be thought of as a matrix of size $h$, the number of conjugacy classes in $G$, 
with complex entries. We denote here as $X$. 

\begin{exercise}
Show that the character table $X$ of a finite group $G$ is an invertible matrix. 
\end{exercise}

\begin{exercise}
Is it possible to have a zero column in a character table? 
\end{exercise}

\begin{exercise}
Let $G$ be a finite group. The elements $s,t\in G$ are conjugate if and only if $\chi(s)=\chi(t)$ 
for all (irreducible) characters $\chi$ of $G$.
\end{exercise}
{\bf Hint:} Consider the class functions $f_s$ and $f_t$ which take vale $1$ on the conjugacy 
classes corresponding to $s$ and $t$ respectively and $0$ elsewhere. Now use that $\mathcal H$ is 
spanned by irreducible characters.

\begin{exercise}
An element $g\in G$ is called real if $g$ is conjugate to $g^{-1}$. A character $\chi$ is called 
real if it takes all real values.
Prove that the number of real conjugacy classes is the same as the number of real irreducible 
characters.
\end{exercise}
{\bf Hint: } When $\chi$ is irreducible, so is $\bar \chi$. Thus, $\chi\rightarrow 
\overline\chi$ permutes the rows of $X$. Thus, $\overline X = PX$ where $P$ is a permutation 
matrix. Now, if we interchange the columns under $g\rightarrow g^{-1}$ once again we get a 
permutation of $X$, which is in fact $X_{r}=\overline X$ (because $\chi(t^{-1})=\overline 
\chi(t)$). Thus, $X_r=\overline X=XQ$ for some permutation matrix $Q$. Now, $PX=XQ$ gives 
$P=XQX^{-1}$ hence $trace(P)=trace(Q)$. Note that the trace of a permutation matrix is the number of 
fixed points and we are done.

\begin{exercise}
What is the determinant of the character table?
\end{exercise}
{\bf Hint: } Let us begin with computing 
$$\tr X \overline X =  \left( \sum_m \chi_m(g_i) \bar \chi_m(g_j)\right) = 
\diag\left(\frac{|G|}{r_1}, \ldots, \frac{|G|}{r_l}, \ldots, \frac{|G|}{r_h} \right)$$
using column orthogonality relations.
Thus, $\det(X) \overline{\det(X)} = \prod_{l} \frac{|G|}{r_1} = \prod_{i=1}^h |\mathcal Z_G(g_i)|$ 
where $g_i$ are representative of conjugacy classes. Now, make use of $X=P\overline X$ where $P$ is 
a permutation matrix to get the answer.

\begin{exercise}
Prove that the sum of any row in the character table is a non-negative integer.
\end{exercise}
{\bf Hint: } Look at the action of $G$ on itself by conjugation. Take the corresponding 
representation $\rho\colon G\rightarrow \GL(k[G])$. Calculate the character $\tau$ of $\rho$, i.e, 
show that $\tau(g) = |\mathcal Z_G(g)|$. Now, write $\tau = \sum_{i=1}^h m_i\chi_i$ and prove that 
$m_i= \langle \tau, \chi_i\rangle = \sum_{j=1}^h \chi_i(g_j)$.

\begin{exercise}
Let $\chi$ be an irreducible character. Show that $\sum_{g\in G} \chi(g) =0$.  
\end{exercise}
{\bf Hint:} Consider the inner product of $\chi$ with the trivial character.

\begin{exercise}
Let $\rho\colon G \rightarrow \GL(V)$ be an irreducible representation and $\chi$ be its character. 
Then, $ker(\rho)$ is $\{g\in G \mid \chi(g) =\chi(1)\}$. Thus, $ker(\rho)$ can be determined from 
the character table as being a normal subgroup it is a union of certain conjugacy classes.  
\end{exercise}
{\bf Hint: } Suppose dimension of $V$ is $n$. We know $\chi(g)=\lambda_1+\cdots + 
\lambda_n$ where $\rho(g)$ can be thought of diagonal matrix $diag(\lambda_1, \ldots , 
\lambda_n)$ and each $\lambda_i$ are roots of unity. Then, $|\chi(g)| \leq n$. When can equality occur gives the answer.
{\color{teal}
\begin{definition}
Let $\rho\colon G \rightarrow \GL(V)$ be a representation and $\chi$ be its character. The kernel of 
character $\chi$ is defined to be  $ker(\chi)=\{g\in G \mid \chi(g) =\chi(1)\}$.
\end{definition}}
\begin{exercise}
A normal subgroup of $G$ is a disjoint union of conjugacy classes. 
\end{exercise}

\begin{exercise}
Any normal subgroup can be obtained by looking at the intersection of the kernel of some characters. 
\end{exercise}
{\bf Hint: } Let $N$ be a normal subgroup. Then, $G$ acts on $G/N$ giving rise to a representation, say $\rho$. Let $\psi$ be its character. The kernel of character $\psi$ is precisely $N$. However, we don't know if this character is irreducible. But, we can write $\psi=\sum m_i\chi_i$.  Then, $ker(\rho)$ is the intersection of kernels of $\chi_i$ appearing in this sum.

\begin{exercise}
Let $\chi$ be a character which takes all values $0$ except at $e$. Show that $\chi$ is a multiple 
of the regular character $\chi_{reg}$. 
\end{exercise}
{\bf Hint:} Let us compute $\langle \chi, {\bf 1} \rangle = \frac{1}{|G|}\chi(e)\neq 0$ 
integer (here ${\bf 1}$ is the trivial character). Thus, $\chi(e) = \langle \chi, {\bf 1} \rangle 
.|G| = \langle \chi, {\bf 1} \rangle. \chi_{reg}(e)$. Hence, $\chi = \langle \chi, {\bf 1} 
\rangle. \chi_{reg}$. 

\begin{exercise}
Let $(\rho,V)$ be a faithful representation of $G$. Then, every irreducible representation of $G$ 
is a subrepresentation of $V^{\otimes m}$, for some $m$.
\end{exercise}
{\bf Hint:} Let $\chi$ be the character of $\rho$ and dimension be $n$. Let $\{\chi(g) \mid g\neq 
e\} = \{\alpha_1, \ldots, \alpha_m\}$. Consider $\phi = \chi\prod_{i=1}^m (\chi-\alpha_i \bf 1)$ in 
$\mathbb C[G]$. Clearly, $\phi$ is a class function and $\phi(g)=0$ if $g\neq e$ whereas $\phi(e) = 
\chi(e)\prod_{i=1}^m (\chi(e)-\alpha_i) = n\prod_{i=1}^m (n-\alpha_i)\neq 0$. Hence, $\phi$ is a 
multiple of the regular character $\chi_{reg}$, i.e., $\phi=r \chi_{reg}$. Now, let $\chi_i$ be 
an irreducible character. Let us assume contrary that $\chi_i$ doesn't appear in $\chi^m$ for 
all $m$, i.e., $\langle \chi_i, \chi^m\rangle =0$ for all $m$. Then, $\langle \chi_i, \phi \rangle 
=0$ and hence $\langle \chi_i, r \chi_{reg} \rangle =0$ which is a contradiction that $\chi_{reg}$ 
contains all irreducible characters.

\begin{exercise}
Let $\chi$ be a non-linear (i.e., dimension is greater than $1$) irreducible character of $G$. 
Prove that there exists a conjugacy class where $\chi$ takes value $0$.
\end{exercise}

\begin{exercise}
\begin{enumerate}
\item Let $G$ be a finite group of size $n$ with the number of conjugacy classes $h$. Suppose $\frac{|G|}{|[G,G]|}=r$. Then, $n+3r\geq 4h$.

{\bf Hint: } We know that $G$ has $h$ irreducible characters of dimension say, $n_i$. We also know that there are $r$ irreducible characters of dimension $1$. So, $n=|G|= \sum_{i=1}^h n_i^2 =  \sum_{n_i=1} n_i^2 + \sum_{n_i\geq 2} n_i^2 = r+\sum_{n_i\geq 2} n_i^2 \geq r +4(h-r)$. 
\item {\bf Commuting Probability}
Given a finite group $G$, the commuting probability $cp(G) =\frac{|\{(g_1,g_2)\in G^2\mid g_1g_2=g_2g_1\}|}{|G|^2}$. Show that if $G$ is non-Abelian, $cp(G)\leq \frac{5}{8}$. This is due to Erd\"os and Turan.

{\bf Hint: }  We write $|\{(g_1,g_2)\in G^2\mid g_1g_2=g_2g_1\}| = \sum_{g\in G} |\mathcal Z_G(g)| =\sum_{g\in G}\frac{|G|}{|Cl(g)|} = |G|h$ (as each $Cl(g)$ appears as many times in the sum). Hence, $cp(G) = \frac{h}{|G|}$. Using the previous exercise $cp(G) = \frac{h}{|G|} \leq \frac{1}{4} +\frac{3r}{4n}$. Now, since $G$ is non-Abelian we have $|[G,G]|\geq 2$, i.e, $r=\frac{|G|}{|[G,G]|}\leq \frac{|G|}{2}$. This gives us $cp(G)\leq \frac{1}{4} +\frac{3r}{4n} \leq  \frac{1}{4} +\frac{3}{8}=\frac{5}{8}$. 
\item Show, by computing the $cp(G)$ for groups of order $8$, that equality holds in the above.
\end{enumerate}
\end{exercise}


\chapter{Examples}

Let $G$ be a finite group. Let $\rho_1,\ldots,\rho_h$ be irreducible representations of $G$ over $\mathbb C$ with corresponding characters 
$\chi_1,\ldots,\chi_h$. We know that the number $h$ is equal to the number of conjugacy classes in $G$. We can make use 
of this information about $G$ and make a {\bf character table} of $G$. The character table is a matrix of size $h\times h$
of which rows are labelled as characters and columns as conjugacy classes.
\vskip3mm
\[\begin{array}{c|c|c|c|c|}
        & r_1=1 & r_2  & \cdots   & r_h \\
        & g_1=e & g_2 & \cdots &   g_h   \\ \hline
\chi_1  & n_1=1&1 &\cdots &1    \\ \hline
\chi_2  & n_2& & &   \\ \hline
\vdots & \vdots& & &   \\ \hline
\chi_h & n_h& & &   \\ \hline
\end{array}\]
\vskip2mm
where $g_1,g_2,\ldots$ denote representative of the conjugacy class and $r_i$ denotes the number of elements in the conjugacy
class of $g_i$.
The following proposition summarizes the results proved about the characters. We also recall the inner product on $\mathbb C[G]$
defined by $\langle f_1,f_2\rangle = \frac{1}{|G|}\sum_{t\in G} f_1(t)\overline{f_2(t)}$.
\begin{proposition}
With the notation as above, we have,
\begin{enumerate}
\item The number of conjugacy classes is the same as the number of irreducible characters which is the same as the number
of non-isomorphic irreducible representations.
\item Two representations are isomorphic if and only if their characters are equal.
\item A representation $\rho$ with character $\chi$ is irreducible if and only if $\langle \chi,\chi\rangle =1$.
\item $|G|=n_1^2+n_2^2+\cdots +n_h^2$ where $n_1=1$ corresponds to the trivial character.
\item The characters form an orthonormal basis of $\mathcal H$, i.e., 
$$\sum_{t\in G}\chi_i(t)\overline{\chi_j(t)}=\sum_{g_l}r_l\chi_i(g_l)\overline{\chi_j(g_l)}= \delta_{ij}|G|.$$
That is, the rows of the character table are orthonormal.
\item The columns of the character table also form an orthogonal set, i.e.,
$$\sum_{i=1}^h \chi_i(g_l)\overline{\chi_i(g_l)} = \frac{|G|}{r_l}$$ and 
$$\sum_{i=1}^h \chi_i(g_l)\overline{\chi_i(g_m)} =0$$ 
where $g_l$ and $g_m$ are representative of different conjugacy classes.
\item The character table matrix is an invertible matrix.
\item The degree of irreducible representations divide the order of the group, i.e., $n_i\mid |G|$.
\end{enumerate}
\end{proposition}
The proof of the last statement will be done later.

{\bf Warning :} If the character table of two groups are the same that does not imply that the groups are isomorphic. 
Look at the character tables of $Q_8$ and $D_4$.
In fact, all non-Abelian groups of order $p^3$ (there are two non-isomorphic ones) have the same 
character table (find a reference for this).

\section{Groups having Large Abelian Subgroups}
First, we will give another proof of the Theorem~\ref{Abelian} using Character Theory.
{\color{red}
\begin{theorem}
Let $G$ be a finite group. Then $G$ is Abelian if and only if all irreducible representations are of dimension $1$, 
i.e., $n_i=1$ for all $i$.
\end{theorem}}
\begin{proof}
With the notation as above let $G$ be Abelian. We have 
$$|G|=n_1^2+\cdots+n_h^2$$
where $h=|G|$. Hence the only solution to the
equation is $n_i=1$ for all $i$. Now suppose $n_i=1$ for all $i$. Then the above equation implies $h=|G|$. 
Hence each conjugacy class has a size of $1$ and the group is Abelian.
\end{proof}
\begin{proposition}
Let $G$ be a group and $A$ be an Abelian subgroup. Then $n_i\leq \frac{|G|}{|A|}$ for all $i$.
\end{proposition}
\begin{proof}
Let $\rho\colon G\rightarrow \GL(V)$ be an irreducible representation. We can restrict $\rho$ to $A$ 
and denote it
by $\rho_A\colon A\rightarrow \GL(V)$ which may not be irreducible. Let $W$ an irreducible 
$\rho_A$- invariant subspace of $V$.
The above theorem implies $\dim(W)=1$. Say $W=<v>$ where $v\neq 0$. Consider $V'=<\{\rho(g)v \mid g\in G\}>\subset V$.
Clearly $V'$ is $G$-invariant and as $V$ is irreducible $V'=V$. Notice that $\rho(ga)v=\lambda\rho(g)v$, i.e., $\rho(ga)v$ and
$\rho(g)v$ are linearly dependent. Hence $V'=<\{\rho(g_1)v,\ldots,\rho(g_m)v\}>$ where $g_i$ are representatives of the 
coset $g_iA$ in $G/A$. 
This implies $\dim(V)\leq m =\frac{|G|}{|A|}$.
\end{proof}
\begin{corollary}\label{dihedral}
Let $G=D_n$ be the dihedral group with $2n$ elements. Any irreducible representation of $D_n$ has dimension $1$ or $2$. 
\end{corollary}

\section{Character Table of Some Groups}
\begin{example}[Cyclic Group]
Let $G=\mathbb Z/n\mathbb Z$. All representations are one-dimensional hence giving character. The characters are $\chi_1,\ldots,\chi_n$ given by $\chi_r(s)=e^{\frac{2\pi i}{n}rs}$ for $0\leq s\leq n$.
\end{example}

\begin{example}[$S_3$]

$$S_3=\{1,(12), (23), (13), (123), (132)\}$$
We already know the one-dimensional representations of $S_3$ which are trivial representations and the sign representation.
The characters for these representations are themselves. Now we use the formula $6=|G|=n_1^2+n_2^2+n_3^2 = 1+1+n_3^2$ gives 
$n_3=2$. Now we use $\langle \chi_3,\chi_1\rangle = \frac{1}{6} \{1.2.1 + 3.a.1 + 2.b.1\} =0$ and 
$\langle \chi_3,\chi_2\rangle = \frac{1}{6}\{1.2.1 + 3.a.(-1) + 2.b.1 =0\}$. Hence solving the above equations $3a+2b=-2$ and
$-3a+2b=-2$ gives $a=0$ and $b=-1$.

\vskip1mm
\[\begin{array}{c|c|c|c|}
        & r_1=1 & r_2=3    & r_3=2 \\
        & g_1=1 & g_2=(12) & g_3=(123)   \\ \hline
\chi_1  & n_1=1 & 1        & 1    \\ \hline
\chi_2  & n_2=1 & -1       & 1  \\ \hline
\chi_3  & n_3=2 & a=0        & b=-1   \\ \hline
\end{array}\]
\vskip1mm
\end{example}

\begin{example}[$Q_8$]
$$Q_8=\{1,-1, i,-i,j,-j,k,-k\}$$
The commutator subgroup of $Q_8=\{1,-1\}$ and $Q_8/\{\pm1\}\cong \mathbb Z/2\mathbb Z \times \mathbb Z/2\mathbb Z$.
Hence it has $4$ one-dimensional representations which are lifted from $\mathbb Z/2\mathbb Z \times \mathbb Z/2\mathbb Z$.
Notice that the first $\mathbb Z/2\mathbb Z$ component is the image of $i$ and the second one of $j$.
Now we use $8=|G|=n_1^2+n_2^2+n_3^2 +n_4^2 +n_5^2 = 1+1+1+1+n_5^2$ gives $n_5=2$. Again using orthogonality of $\chi_5$
with other known $\chi_i$'s we get the following equations:
\begin{eqnarray*}
1.2.1 + 1.a.1 + 2.b.1 + 2.c.1 + 2.d.1 &=& 0 \\
1.2.1 + 1.a.1 + 2.b.1 + 2.c.(-1) + 2.d.(-1) &=& 0 \\
1.2.1 + 1.a.1 + 2.b.(-1) + 2.c.1 + 2.d.(-1) &=& 0 \\
1.2.1 + 1.a.1 + 2.b.(-1) + 2.c.(-1) + 2.d.1 &=& 0 
\end{eqnarray*}
This gives the solution $a=-2, b=0, c=0$ and $d=0$.
\vskip1mm
\[\begin{array}{c|c|c|c|c|c|}
        & r_1=1 & r_2=1  & r_3=2   & r_4=2 & r_5=2 \\
        & g_1=1 & g_2=-1 & g_3=i &   g_4=j & g_5 =k   \\ \hline
\chi_1  & n_1=1 & 1      & 1     &  1     &  1  \\ \hline
\chi_2  & n_2=1 & 1     & 1    &   -1     &-1   \\ \hline
\chi_3  & n_3=1 & 1     &  -1   &  1   & - 1  \\ \hline
\chi_4 & n_4 = 1& 1     & -1    &  -1    & 1 \\ \hline
\chi_5 & n_5=2  & a=-2  & b= 0  & c= 0   & d=0 \\ \hline
\end{array}\]
\vskip1mm
\end{example}

\begin{example}[$D_4$]
  $$D_4=\{r,s \mid r^4=1=s^2, rs=sr^{-1}\}$$
The commutator subgroup of $D_4$ is $\{1,r^2\}=\mathcal Z(D_4)$. 
And $D_4/\mathcal Z(D_4)\cong \mathbb Z/2\mathbb Z\times \mathbb Z/2\mathbb Z$ hence there are $4$ one dimensional 
representations as in the case of $Q_8$. We can also compute the rest of it as we did in $Q_8$. We also observe
that the character table is the same as $Q_8$.
\vskip1mm
\[\begin{array}{c|c|c|c|c|c|}
        & r_1=1 & r_2=1  & r_3=2   & r_4=2 & r_5=2 \\
        & g_1=1 & g_2=r^2 & g_3=r &   g_4=s & g_5 =sr   \\ \hline
\chi_1  & n_1=1 & 1      & 1     &  1     &  1  \\ \hline
\chi_2  & n_2=1 & 1     & 1    &   -1     &-1   \\ \hline
\chi_3  & n_3=1 & 1     &  -1   &  1   & - 1  \\ \hline
\chi_4 & n_4 = 1& 1     & -1    &  -1    & 1 \\ \hline
\chi_5 & n_5=2  & a=-2  & b= 0  & c= 0   & d=0 \\ \hline
\end{array}\]
\vskip1mm
\end{example}

\begin{example}[$A_4$]
$$A_4=\{1,(12)(34),(13)(24),(14)(23), (123),(132), (124),(142),(134), (143), (234),(243)\}$$
We note that the $3$ cycles are not conjugate to their inverses. Here we have 
$$H= \{1,(12)(34),(13)(24),(14)(23)\}\cong \mathbb Z/2\mathbb Z\times \mathbb Z/2\mathbb Z$$ a normal subgroup of $A_4$ and $A_4/H\cong \mathbb Z/3\mathbb Z$. 
This way the $3$ one dimensional irreducible representations of $ \mathbb Z/3\mathbb Z$ lift to 
$A_4$ as we have maps $A_4\rightarrow A_4/H   \cong  \mathbb Z/3\mathbb Z\rightarrow \GL_1(\mathbb 
C)$ given by $\omega$ where $\omega^3=1$. 
We use $1^2+1^2+1^2+n_4^2=12=|A_4|$ to get $n_4=3$. 
Now we take the inner product of $\chi_4$ with others and get the equations:
\begin{eqnarray*}
1.3.1+3.a.1+4.b.1+4.c.1 &=& 0\\
1.3.1+3.a.1+4.b.\overline{\omega} + 4.c.\overline{\omega}^2 &=&0\\
1.3.1+3.a.1+4.b.\overline{\omega}^2 + 4.c.\overline{\omega} &=&0
\end{eqnarray*} 
This gives us the character table.
\vskip1mm
\[\begin{array}{c|c|c|c|c|}
        & r_1=1 & r_2=3  & r_3=4   & r_4=4 \\
        & g_1=1 & g_2=(12)(34) & g_3=(123) &   g_4=(132)   \\ \hline
\chi_1  & n_1=1& 1 & 1& 1    \\ \hline
\chi_2  & n_2=1& 1&\omega &\omega^2   \\ \hline
\chi_3 & n_3=1& 1&\omega^2 &\omega   \\ \hline
\chi_4 & n_4=3& a=-1&b=0 &c=0   \\ \hline
\end{array}\]
\vskip1mm 
\end{example}

\begin{example}[$S_4$]
The group $S_4$ has $2$ one-dimensional representations given by trivial and sign. We recall that the permutation representation (Example~\ref{permutation}) gives rise to the subspace $V=\{(x_1,x_2,\ldots,x_n)\in\mathbb C^n\mid x_1+x_2+\cdots +x_n=0\}$ which is an $n-1$ dimensional irreducible representation of $S_n$. We will make use of this to get a $3$ dimensional representation of $S_4$ and corresponding character $\chi_3$. A basis of the space $V$ is $\{e_1-e_2,e_2-e_3,e_3-e_4\}$ and the action is given by:
\begin{eqnarray*}
(12) : & (e_1-e_2) \mapsto e_2-e_1 = -(e_1-e_2)\\
          & (e_2-e_3)\mapsto e_1-e_3 = (e_1-e_2) + (e_2-e_3)\\
          &(e_3-e_4) \mapsto (e_3-e_4)
\end{eqnarray*}
So the matrix is $\left[\begin{matrix} -1 & 1 &0 \\ 0&1 & 0 \\ 0 &0 &1\end{matrix}\right]$ and $\chi_3((12))=1$. The action of $(12)(34)$ is $e_1-e_2\mapsto e_2-e_1=-(e_1-e_2)$, $e_2-e_3\mapsto e_1-e_4=(e_1-e_2)+(e_2-e_3)+(e_3-e_4)$ and $e_3-e_4 \mapsto e_4-e_3 =-(e_3-e_4)$ and the matrix is $\left[\begin{matrix} -1 & 1 &0 \\ 0&1 & 0 \\ 0 &1 &-1\end{matrix}\right]$. So $\chi_3((12)(34))=-1$. The action of $(123)$ is $e_1-e_2\mapsto e_2-e_3$, $e_2-e_3\mapsto e_3-e_1=-(e_1-e_2)-(e_2-e_3)$ and $e_3-e_4\mapsto e_1-e_4 = (e_1-e_2)+(e_2-e_3)+(e_3-e_4)$. So the matrix is $\left[\begin{matrix} 0 & -1 &1 \\ 1&-1 & 1 \\ 0 &0 &1\end{matrix}\right]$ and $\chi_3((123))=0$. And the action of $(1234)$ is $e_1-e_2\mapsto e_2-e_3$, $e_2-e_3\mapsto e_3-e_4$ and $e_3-e_4\mapsto e_4-e_1=-(e_1-e_2)-(e_2-e_3)-(e_3-e_4)$.
So the matrix is $\left[\begin{matrix} 0 & 0 &-1 \\ 1&0 & -1 \\ 0 &1 &-1\end{matrix}\right]$ and $\chi_3((1234))=-1$.
This gives $\chi_3$. We can get another character $\chi_4=\chi_3.\chi_2$ corresponding to the representation $Perm\otimes sgn$. We can check that this is different from others so corresponds to a new representation and is also irreducible as $\langle \chi_4,\chi_4\rangle =1$.

To find $\chi_5$ we can use orthogonality relations and get equations.
\vskip1mm 
\[\begin{array}{c|c|c|c|c|c|}
          & r_1=1 & r_2=6  &r_3=3   & r_4=8 &r_5=6 \\
          & g_1=1 & g_2=(12) & g_3=(12)(34) &   g_4=(123) &g_5=(1234)   \\ \hline
\chi_1  & n_1=1&1 &1 &1 &1   \\ \hline
\chi_2  & n_2=1&-1 &1 &1 &-1   \\ \hline
\chi_3 & n_3=3& 1 &-1 &0&-1   \\ \hline
\chi_4 & n_4=3& -1 &-1 &0 &1   \\ \hline
\chi_5 & n_5=2 &a=0&b=2&c=-1&d=0\\\hline
\end{array}\]
\vskip2mm
The representation to which $\chi_5$ corresponds is the $2$-dimensional irreducible representation of
$S_4/V_4\cong S_3$. Let us understand this isomorphism. Consider the action of $S_4$ on the set 
$X=\{x_1=\{1,2\}\sqcup \{3,4\}, x_2=\{1,3\}\sqcup \{2,4\}, x_3=\{1,4\}\sqcup \{2,3\}\}$. It's a transitive
action with kernel $V_4$.

\end{example}
\begin{example}[$D_n$, $n$ even]
 $$D_n=\{a,b \mid a^n=1=b^2, ab=ba^{-1}\}$$
and the conjugacy classes are $\{1\}$, $\{a^{\frac{n}{2}}\}$, $\{a^j,a^{-j}\}$ for all $1\leq j\leq 
\frac{n}{2}-1$, $\{a^jb\mid j\ even\}$ and $\{a^jb\mid j\ odd\}$. The center is $\{1, 
a^{\frac{n}{2}} \}$. The subgroup generated by $a^2$ is a commutator and is normal hence there 
are $4$ one-dimensional representations.
The rest of them are two-dimensional (refer to Corollary~\ref{dihedral}) representations defined in the Example~\ref{dihedralrep}. Using them we can make the character table.
\end{example}

\begin{example}[$D_n$, $n$ odd]
 $$D_n=\{a,b \mid a^n=1=b^2, ab=ba^{-1}\}$$
and the conjugacy classes are $\{1\}$, $\{a^j,a^{-j}\}$ for all $1\leq j\leq \frac{n-1}{2}$ and
$\{a^jb\mid 0\leq j\leq n-1\}$. The commutator subgroup is generated by $a$, and is normal and hence 
there are $2$ one-dimensional representations.
The rest of them are two-dimensional (refer to Corollary~\ref{dihedral}) representations defined in the Example~\ref{dihedralrep}. Using them we can make the character table.
\end{example}

\begin{exercise}
Write down a criterion such that a representation is conjugate to its own adjoint.
Find examples of representations which are not equivalent to their adjoint (analyse one-dimensional representations). 
\end{exercise}

\begin{exercise}\label{character-action}
Let $G$ be a group acting on a finite set $X$. Let us denote the corresponding representation on 
$k[X]$ by $\lambda_X$.
\begin{enumerate}
\item A function $f\in k[X]$ is fixed by $\lambda_X$ if and only if $f$ is constant on orbits of 
$G$ in $X$. In particular, $\dim_{k}k[X]^G$ is the number of orbits of $G$ in $X$.
\item Show that the character of this representation $\chi_{\lambda_X}(g)$ is the number of fixed 
points of $g$ in $X$.
\item Show that $k[X]=W \oplus W_0$, where $W$ is space of constant functions and $W_0=\{f\in 
k[X]\mid \sum_{x\in X}f(x)=0\}$, is a $G$ decomposition.
\end{enumerate}
\end{exercise}

\begin{exercise}
In continuation to Exercise~\ref{character-action}, now let us assume that $G$ is doubly 
transitive on $X$.
\begin{enumerate}
\item Consider the diagonal action of $G$ on $X\times X$. Show that $\frac{1}{|G|}\sum_{g\in 
G}\chi_{\lambda_{X\times X}}(g) =2$.
\item Show that $k[X]\otimes k[X]\cong k[X\times X]$ as $G$-modules and hence 
$\chi_{\lambda_{X\times X}}=\chi_{\lambda_X}^2$.
\item Show that $2=\langle \chi_{\lambda_X}, \chi_{\lambda_X}\rangle_G$.
\item Show that $W_0$ is irreducible by showing that the character has norm $1$. 
 \end{enumerate}

\end{exercise}

\section{Characters of Direct Product}
Let $G$ and $G'$ be two finite groups. Let $(\rho, V)$ be a representation of $G$ and $(\rho', V')$ 
be a representation of $G'$. Then, we can define $(\rho\otimes\rho', V\otimes V')$ a representation 
of $G\times G'$ as follows: $\rho\times\rho' \colon G\times G' \rightarrow \GL(V\otimes V')$ such 
that $(\rho\otimes\rho')(g,g')(\sum v\otimes v') = \sum \rho(g)(v) \otimes \rho(g')(v')$.
\begin{exercise}
Show that $\rho\otimes \rho'$ is a representation of $G\times G'$. 
\end{exercise}
\begin{exercise}
Let $\chi$ be the character of $\rho$ and $\chi'$ be that of $\rho'$. Show that the character of $\rho\otimes \rho'$ is $\chi\chi'$ where $(\chi\chi')(g,g')=\chi(g)\chi'(g')$.
\end{exercise}
\begin{proposition}
Consider $\mathbb C$-representations. If $\rho$ and $\rho'$ are irreducible representations of $G$ and $G'$ respectively then $\rho\otimes \rho'$ is an irreducible representation of $G\times G'$.
\end{proposition}
\begin{proof} 
\begin{eqnarray*}
\langle \chi\chi', \chi\chi' \rangle &=& 
\frac{1}{|G\times G'|} \displaystyle\sum_{(t,t')\in G\times G'} \chi(t)\chi'(t')\bar\chi(t)\bar\chi'(t') \\
&=& \frac{1}{|G||G'|} \left(\displaystyle\sum_{t\in G} \chi(t)\bar\chi(t)\right)\left (\displaystyle\sum_{t'\in G'} \chi'(t')\bar\chi'(t')\right)\\
&=& \langle \chi, \chi \rangle  \langle \chi', \chi' \rangle = 1. 
\end{eqnarray*}
\end{proof}

Thus, if $\rho_1 \ldots \rho_r$ are irreducible representations with irreducible characters $\chi_1, \ldots, \chi_r$ of $G$, and $\rho'_1 \ldots \rho'_s$ are irreducible representations with irreducible characters $\chi'_1, \ldots, \chi'_s$ of $G'$ then $\rho_i\otimes \rho'_j$ are irreducible representations of $G\otimes G'$ with characters $\chi\chi'$.  

From the Fundamental Theorem of Abelian groups, any Abelian group is a product of cyclic groups. For the cyclic groups, all irreducible representations can be easily written by finding out an appropriate root of unity in $\mathbb C$. This section gives a complete method to writing down all representations for an Abelian group explicitly. 

\begin{exercise}
Compute the character table of all Abelian groups up to order $25$. 
\end{exercise}
\begin{exercise}
Show that the above Proposition need not be true for a semi-direct product (try the Dihedral groups).
\end{exercise}
Once we understand the induced representations we will be able to find a method to get characters of a semi-direct product of groups.

\chapter{Character Table of $S_5$}

With the theory developed so far, we are now ready to construct the character table for the symmetric group $S_5$. 
The size of this group is $|S_5| = 5! =120$. Two elements $\sigma$ and $\sigma'$ of $S_n$ are conjugate if and only if  their cycle structure is the same when they are written as products of disjoint cycles. Thus, the conjugacy classes in $S_5$ and the size of classes are as follows:
    \begin{center}
    \begin{tabular}{c|c|c|c|c|c|c|c|} 
   $|(g_i)|$ & $1$ & $10$ & $20$ & $15$ & $30$ & $20$ & $24$ \\
   $g_i$ & $1$ & $(1 2)$ & $(1 2 3)$ & $(1 2)(3 4)$ & $(1 2 3 4)$ & $(1 2)(3 4 5)$ & $(1 2 3 4 5)$ \\ 
   \hline 
\end{tabular}
\end{center}
We already know $2$ one-dimensional (and hence irreducible) representations. One is the trivial representation and the other the sign representation which sends each transposition to $-1$.

For any $S_n$, we know an irreducible representation of dimension $n-1$ (as in the Example~\ref{permutation} obtained by the action of $S_n$ on the subspace $V$ of $\mathbb C ^n$ given by $V=\{(x_1,x_2,\ldots,x_n) \in \mathbb C^n \mid  x_1 + x_2 +\cdots + x_n =0 \}$).
We fill  this information into the character table.
We need to find $4$ more irreducible characters. 

Let us take the tensor product of the trivial representation with any other representation $(\rho,V)$ gives a representation isomorphic to $(\rho,V)$ (since $\chi_1 \chi_V = \chi_V$). Hence we only need to consider the tensor product of the second and the third representation namely $sgn\otimes V$ whose character is $\chi_2 \chi_3$. Check $  \langle \chi_2 \chi_3, \chi_2 \chi_3 \rangle = \frac{1}{120}((4)^2 + 10(-2)^2) + 20 (1)^2 +20 (1)^2 +24 (-1)^2) = 1 = \langle  \chi_3,  \chi_3 \rangle
$. Thus this representation turns out to be irreducible. Let $\chi_4 = \chi_2 \chi_3\neq \chi_3$. We include this character $\chi_4$ into the character table:

\begin{center}
    \begin{tabular}{c|c|c|c|c|c|c|c|}
  $|(g_i)|$ & $1$ & $10$ & $20$ & $15$ & $30$ & $20$ & $24$ \\
   $g_i$ & $1$ & $(1 2)$ & $(1 2 3)$ & $(1 2)(3 4)$ & $(1 2 3 4)$ & $(1 2)(3 4 5)$ & $(1 2 3 4 5)$ \\ 
   \hline
   $\chi_1$ & $1$ & $1$ & $1$ & $1$ & $1$ & $1$ & $ 1$ \\
   $\chi_2$ & $1$ & $-1$ & $1$ & $1$ & $-1$ & $-1$ & $1$ \\
   $\chi_3$ & $4$ & $2$ & $1$ & $0$ & $0$ & $-1$ & $-1$ \\ 
$\chi_4$ & $4$ & $-2$ & $1$ & $0$ & $0$ & $1$ & $-1$ \\ 
  \hline
\end{tabular}
\end{center}
\vskip1mm
We now consider $\chi = \chi_3$ and examine the representations $\chi_S$ and $\chi_A$:
\begin{center}
  \begin{tabular}{c|c|c|c|c|c|c|c|}
   $|(g_i)|$ & $1$ & $10$ & $20$ & $15$ & $30$ & $20$ & $24$ \\
   $g_i$ & $1$ & $(1 2)$ & $(1 2 3)$ & $(1 2)(3 4)$ & $(1 2 3 4)$ & $(1 2)(3 4 5)$ & $(1 2 3 4 5)$ \\
   \hline
   $\chi_A$ & $6$ & $0$ & $0$ & $-2$ & $0$ & $0$ & $1$ \\
   $\chi_S$ & $10$ & $4$ & $1$ & $2$ & $0$ & $1$ & $0$ \\
  \hline
\end{tabular}
\end{center}
\vskip2mm
We check that $\chi_A$ is irreducible as $ \langle \chi_A , \chi_A \rangle = \frac{1}{120}((6)^2 + 15(-2)^2 + 24 (1)^2) = 1$. Thus $\chi_5 = \chi_A$ is the fifth irreducible character of $S_5$. 

Suppose that $\chi_6$ and $\chi_7$ are the other $2$ irreducible characters. Since every representation of a finite group can be written as a direct sum of irreducible ones we have, $
\chi_S = m_1 \chi_1 + m_2 \chi_2 + \cdots + m_7 \chi_7 $ where $m_i = \langle \chi_S , \chi_i \rangle$. Calculations show that $\langle \chi_S , \chi_S \rangle = 3$, $\langle \chi_S , \chi_1 \rangle = 1$ and $\langle \chi_S , \chi_3 \rangle = 1$. We also have $\sum m_i ^2 = \langle \chi_S, \chi_S \rangle = 3$.
 Thus $\chi_S = \chi_1 + \chi_3 + \psi$, where $\psi$ is an irreducible character. We rewrite $ \psi = \chi_S - \chi_1 -\chi_3$ explicitly and get

\begin{center}
    \begin{tabular}{c|c|c|c|c|c|c|c|}
   $|(g_i)|$ & $1$ & $10$ & $20$ & $15$ & $30$ & $20$ & $24$ \\
   $g_i$ & $1$ & $(1 2)$ & $(1 2 3)$ & $(1 2)(3 4)$ & $(1 2 3 4)$ & $(1 2)(3 4 5)$ & $(1 2 3 4 5)$ \\ \hline
   $\psi$ & $5$ & $1$ & $-1$ & $1$ & $-1$ & $1$ & $0$ \\\hline
\end{tabular}
\end{center}
\vskip2mm
Since  $\psi$ is irreducible, $\langle \psi, \psi \rangle = \frac{1}{120}(5^2 + 10 + 20 +15 +30 +20) = 1$ as expected. 
Let $\chi_6 = \psi$ then $\chi_7 = \chi_2 \chi_6$ will be another new irreducible character. 
We have thus found the character table of $S_5$:
\vskip2mm

\begin{center}
   \textit{Character Table of $S_5$}\\\vskip2mm
    \begin{tabular}{c|c|c|c|c|c|c|c|}
   $|(g_i)|$ & $1$ & $10$ & $20$ & $15$ & $30$ & $20$ & $24$ \\
   $g_i$ & $1$ & $(1 2)$ & $(1 2 3)$ & $(1 2)(3 4)$ & $(1 2 3 4)$ & $(1 2)(3 4 5)$ & $(1 2 3 4 5)$ \\    \hline 
   $\chi_1$ & $1$ & $1$ & $1$ & $1$ & $1$ & $1$ & $ 1$ \\
   $\chi_2$ & $1$ & $-1$ & $1$ & $1$ & $-1$ & $-1$ & $1$ \\
   $\chi_3$ & $4$ & $2$ & $1$ & $0$ & $0$ & $-1$ & $-1$ \\
$\chi_4$ & $4$ & $-2$ & $1$ & $0$ & $0$ & $1$ & $-1$ \\
$\chi_5$ & $6$ & $0$ & $0$ & $-2$ & $0$ & $0$ & $1$ \\
$\chi_6$ & $5$ & $1$ & $-1$ & $1$ & $-1$ & $1$ & $0$ \\
$\chi_7$ & $5$ & $-1$ & $-1$ & $1$ & $1$ & $-1$ & $0$ \\
  \hline
\end{tabular}
\end{center}\vskip2mm
Using this and restriction we can find the character table of $A_5$ which we do later.

\chapter{Restriction of a Representation}

Let $(\rho, V)$ be a  representation of the group $G$, i.e., $\rho \colon G \rightarrow \GL(V)$ is a 
group homomorphism. Suppose $H$ is a subgroup of $G$. Then we can restrict the map $\rho$ to $H$ 
denoted as $Res_H^G\rho$ or $\rho|_{H}$ to obtain a representation for $H$. Note that even if 
$(\rho, V)$ is an irreducible representation of $G$, $(Res_H^G\rho, V)$ need not be irreducible for 
$H$. 
\begin{example}
Consider a $2$-dimensional irreducible representation of the Dihedral Group and restrict it to its cyclic subgroup. 
\end{example}
\noindent  In this chapter we explore the connection between the characters of a group $G$ and any of its subgroups $H$. Recall from Chapter~\ref{chartheory} we define an inner product on $\mathbb C[G]$ denoted as $\langle , \rangle$. We denote this inner product on $\mathbb C[H]$ by $\langle , \rangle_H$ thinking of $H$ as a group in itself.
 
\begin{proposition} Let $G$ be a group and $H$ be a subgroup. Let $\psi$ be a non-zero character of $H$. Then there exists an irreducible character $\chi$ of $G$ such that $\langle Res_H^G\chi, \psi \rangle_{H} \neq 0$.
\end{proposition}
\begin{proof}
Let $\chi_1, \chi_2, \ldots, \chi_h$ be the irreducible characters of $G$. 
We have the regular representation of $G$: $ \chi_{reg} = \sum_{i=1}^h \chi_{i} (1) \chi_i $, with values $\chi_{reg} (1) = |G| $, $ \chi_{reg} (g) = 0$ for all $g \neq 1$.
  Thus,
  $$   \langle Res_H^G\chi_{reg}, \psi \rangle_{H} = \frac{1}{|H|}\chi_{reg}(1)\psi(1)= \frac{|G|\psi(1)}{|H|} \neq 0.$$
  Therefore, $\langle Res_H^G\chi_{reg}, \psi \rangle_{H} = \sum_{i=1}^h \chi_i(1) \langle Res_H^G\chi_{i},\psi \rangle_{H} \neq 0 $.  Since all quantities involved in this sum are non-negative at least one of the $\langle \chi_i , \psi \rangle_{H} \neq 0$.
\end{proof}
{\color{red}
\begin{theorem}\label{charsubgroup}
Let $G$ be a group and $H$ a subgroup. Let $\chi$ be an irreducible character of $G$. 
Suppose $\psi_1 ,\psi_2 ,\ldots, \psi_k$ are all irreducible characters of $H$ and  
 $Res_H^G\chi = d_1 \psi_1 + d_2 \psi_2 + \cdots + d_k \psi_k$ for some $d_1, d_2, \ldots, d_k$ integers. Then, 
$$\sum_{i=1} ^k d_{i} ^2 \leq \frac{|G|}{|H|}$$
and equality occurs if and only if $\chi (g) = 0$ for all $g \notin H$. 
\end{theorem}}
\begin{proof}
Thinking of $H$ as a group we have $\langle Res_H^G\chi, Res_H^G\chi \rangle_{H}= \sum_{i=1}^k d_i^2$. Since $\chi$ is an irreducible character of $G$ we have,
\begin{eqnarray*}
 1= \langle \chi,\chi \rangle &=& \frac{1}{|G|}\sum_{g \in G} \chi(g)\overline{\chi(g)} = \frac{1}{|G|}\left(\sum_{h \in H} \chi(h)\overline{\chi(h)} + \sum_{g \notin H} \chi(g)\overline{\chi(g)} \right) \\  &=& \frac{|H|}{|G|} \langle Res_H^G\chi, Res_H^G\chi \rangle_{H} + \frac{1}{|G|} \sum_{g \notin H} \chi(g)\overline{\chi(g)}.
\end{eqnarray*}
Rewriting the above we get,
$\frac{|H|}{|G|} \langle Res_H^G\chi, Res_H^G\chi \rangle_{H}  = 1 - \frac{1}{|G|} \sum_{g \notin H} \chi(g)\overline{\chi(g)} $ which implies
$$ \sum d_i^2 =\langle Res_H^G\chi, Res_H^G\chi \rangle_{H}  = \frac{|G|}{|H|} - \frac{1}{|H|} \sum_{g \notin H} \chi(g)\overline{\chi(g)}  \leq  \frac{|G|}{|H|}.$$
Moreover equality occurs if and only if $\frac{1}{|H|} \sum_{g \notin H} \chi(g)\overline{\chi(g)} = 0$, i.e., $\sum_{g \notin H} |\chi(g)|^2 = 0$ which happens if and only if  $|\chi(g)| = 0$ and hence $\chi(g)=0$ for all $g \notin H$.
\end{proof}
\noindent We will apply the above theorem for index two subgroups and get the following,
\begin{corollary}\label{index2}
Let $G$ be a group and $H$ be a subgroup of index $2$. Let $\chi$ be an irreducible character of $G$. 
Then one of the following happens :
\begin{enumerate}
  \item $Res_H^G\chi = \psi$ is an irreducible character of $H$. This happens if and only if there exists $g \in G$ and $g\notin H$ such that $\chi(g) \neq 0$.
  \item $Res_H^G\chi = \psi_1 + \psi_2$ where $\psi_1$ and $\psi_2$ are irreducible characters of $H$. This happens  if and only if  $\chi(g) = 0$ for all $g \notin H$.
\end{enumerate}
\end{corollary}
\begin{proof}
 Let us write  $Res_H^G\chi = d_1 \psi_1 + d_2 \psi_2 + \cdots + d_k \psi_k$ then from Theorem~\ref{charsubgroup} we get $\sum d_i^2\leq 2$. Further, $\sum d_i^2 = 2$ if and only if $\chi(g)=0$ for all $g\notin H$, i.e., in this case $d_1=d_2=1$ is the solution and we get (2). Otherwise,   $\sum d_i^2 = 1$ and the only solution is $d_1=1$ which gives the first case.  
\end{proof}

\section{Character Table of $A_5$}
Now we apply the results obtained so far on $G=S_5$ and get the characters of its subgroup $H=A_5$.
We aim to write down the character table of $A_5$. For this, we will use the Corollary~\ref{index2} and the character table of $S_5$ derived earlier in the previous chapter. 
The conjugacy classes of $A_5$ and their corresponding sizes are as follows:
\begin{center}
    \begin{tabular}{c|c|c|c|c|c|}
   $|(g_i)|$ & $1$ & $20$ & $15$ & $12$ & $12$  \\
   $g_i$ & $1$ &  $(1 2 3)$ & $(1 2)(3 4)$ &  $(1 2 3 4 5)$  & $(1 3 4 5 2)$ \\    \hline
\end{tabular}
\end{center}
\vskip2mm
From the character table of $S_5$ it follows that $Res\chi_1  = Res\chi_2$, $Res\chi_3 = Res\chi_4$ and $Res\chi_6 = Res\chi_7$ are irreducible characters of $A_5$. 
Hence we get the partial character table of $A_5$ as follows.
\begin{center}
    \begin{tabular}{c|c|c|c|c|c|}
   $|(g_i)|$ & $1$ & $20$ & $15$ & $12$ & $12$  \\
   $g_i$ & $1$ &  $(1 2 3)$ & $(1 2)(3 4)$ &  $(1 2 3 4 5)$  & $(1 3 4 5 2)$ \\   \hline
   $\psi_1$ & $1$ &  $1$ & $1$ &  $1$  & $1$ \\
   $\psi_2$ & $4$ &  $1$ & $0$ &  $-1$  & $-1$ \\
   $\psi_3$ & $5$ &  $-1$ & $1$ &  $0$  & $0$ \\
   $\psi_4$ & $n_4=3$ &  $a_1$ & $a_2$ &  $a_3$  & $a_4$ \\
   $\psi_5$ & $n_5=3$ &  $b_1$ & $b_2$ &  $b_3$  & $b_4$ \\\hline
   \end{tabular}
\end{center}
\vskip1mm
We know that $1^2 + 4^2 + 5^2 +{n_4}^2 +{n_5}^2 = 60 $ and hence ${n_4}^2 +{n_5}^2 = 18$. The only possible integral solutions of this equation are $n_4 = n_5 =3$.

Now we see that the only irreducible character of $S_5$ whose restriction to $A_5$ is not irreducible is $\chi_5$. 
Since $\psi_1, \psi_2, \psi_3$  are all obtained by restriction of the characters other than $\chi_5$ it follows from Theorem~\ref{charsubgroup} that only possibly $\langle \psi_4, Res\chi_5 \rangle_{A_5} \neq 0$ and $\langle \psi_5, Res\chi_5 \rangle_{A_5} \neq 0$. Now from Corollary~\ref{index2} it follows that $Res\chi_5$ is a sum of two characters of $A_5$ and hence $Res \chi_5 = \psi_4 + \psi_5$. Thus we have $a_1 = -b_1$, $a_2 = -2-b_2$, $a_3 = 1 -b_3$ and $a_4 = 1 -b_4$.

Now we use the orthogonality relations for characters of $A_5$ and get
$\langle \psi_4, \psi_1 \rangle = 3 + 20a_1 + 15a_2 + 12a_3 + 12a_4=0$, $\langle \psi_4,\psi_2 \rangle = 12 + 20a_1 - 12a_3 - 12a_4= 0$,  and $\langle \psi_4,\psi_3 \rangle = 15 - 20a_1+ 15a_2 =0$.
Solving these equations we get $a_1 = 0 $, $a_2 = -1$ and $a_3+a_4=1$.  Hence we have the following:

\begin{center}
    \begin{tabular}{c|c|c|c|c|c|}
   $|(g_i)|$ & $1$ & $20$ & $15$ & $12$ & $12$  \\
   $g_i$ & $1$ &  $(1 2 3)$ & $(1 2)(3 4)$ &  $(1 2 3 4 5)$  & $(1 3 4 5 2)$ \\   \hline
   $\psi_1$ & $1$ &  $1$ & $1$ &  $1$  & $1$ \\
   $\psi_2$ & $4$ &  $1$ & $0$ &  $-1$  & $-1$ \\
   $\psi_3$ & $5$ &  $-1$ & $1$ &  $0$  & $0$ \\
   $\psi_4$ & $3$ &  $0$ & $-1$ &  $a_3$  & $a_4=1-a_3$ \\
   $\psi_5$ & $3$ &  $0$ & $-1$ &  $b_3=1-a_3=a_4$  & $b_4=1-a_4=a_3$ \\\hline
   \end{tabular}
\end{center}
\vskip2mm

\begin{proposition}\label{realchar}
Every element of $A_5$ is conjugate to its own inverse. Hence the entries of the character table are real numbers. 
\end{proposition}
\begin{proof}
Clearly, it is enough to prove that the representatives of the conjugacy classes are conjugate to their own inverse. It is clear for the element $1, (123)$ and $(12)(34)$. For others, we check that
\begin{eqnarray*}
  (1 2 3 4 5)^{-1} &=& (5 4 3 2 1) = (1 5) (2 4) (1 2 3 4 5) (1 5) (2 4) \\
  (1 3 4 5 2)^{-1} &=& (2 5 4 3 1) = (1 2) (3 5) (1 3 4 5 2) (1 2) (3 5) 
\end{eqnarray*}
Now we know that $\chi(g^{-1}) = \overline{\chi(g)}$ and $g$ being conjugate to $g^{-1}$ this is also equal to $\chi(g)$. Hence $\chi(g)=\overline{\chi(g)}$ gives $\chi(g)\in\mathbb R$ for all $g\in A_5$.  
\end{proof}
The above proposition implies that $a_3,a_4,b_3$ and $b_4$ are real numbers.
From $\langle \chi_4, \chi_4 \rangle = 1$ we get $a_3 ^2 + a_4 ^2 =3$. Substituting $a_4=1-a_3$ we get $a_3^2-a_3-1=0$. And the solutions are 
$$a_3 = \frac{1 + \sqrt{5}}{2} =b_4, a_4 = \frac{1 - \sqrt{5}}{2}= b_3 $$ 
or
$$a_3 = \frac{1 - \sqrt{5}}{2}=b_4, a_4= \frac{1 + \sqrt{5}}{2}=b_3 .$$

 Since the values of $\psi_4$ and $\psi_5$ on other conjugacy classes are the same, both the above solutions would give the same set of irreducible characters. Hence without loss of generality, we may take the first set of solutions. This gives the complete character table of $A_5$ as follows:

\begin{center}
   \textit{Character Table of $A_5$}\\\vskip2mm
    \begin{tabular}{c|c|c|c|c|c|}
   $|(g_i)|$ & $1$ & $20$ & $15$ & $12$ & $12$  \\
   $g_i$ & $1$ &  $(1 2 3)$ & $(1 2)(3 4)$ &  $(1 2 3 4 5)$  & $(1 3 4 5 2)$ \\    \hline 
   $\psi_1$ & $1$ &  $1$ & $1$ &  $1$  & $1$ \\
   $\psi_2$ & $4$ &  $1$ & $0$ &  $-1$  & $-1$ \\
   $\psi_3$ & $5$ &  $-1$ & $1$ &  $0$  & $0$ \\
   $\psi_4$ & $3$ &  $0$ & $-1$ &  $\frac{1 + \sqrt{5}}{2}$  & $\frac{1 - \sqrt{5}}{2}$ \\
   $\psi_5$ & $3$ &  $0$ & $-1$ &  $\frac{1 - \sqrt{5}}{2}$  & $\frac{1 + \sqrt{5}}{2}$ \\\hline
   \end{tabular}
\end{center}
\vskip5mm

\begin{exercise}
Prove that every element of $S_n$ is conjugate to its own inverse and hence character table consists of real numbers.
\end{exercise}
{\bf Remark : } In fact more is true that every element of $S_n$ is conjugate to all those powers of itself which generates the same subgroup (called rational conjugacy). Hence it is true that characters of $S_n$ always take value in integers. However, this is not true for $A_n$, for example, check $A_4$ and $A_5$. In general, for $A_n$, they may not be even real-valued.

\chapter{Central Characters}

Let $G$ be a finite group. Let $\rho_1,\ldots,\rho_h$ be all irreducible $\mathbb 
C$-representations of $G$ of dimension $n_1,\cdots, n_h$ respectively with corresponding characters 
$\chi_1,\cdots,\chi_h$. We also fix $g_1, g_2, \ldots, g_h$ to be a representative of conjugacy 
classes of $G$ with the respective conjugacy classes size $r_1, \ldots, r_h$. 

For any give representation $\rho \colon G\rightarrow \\GL(V)$ we can define an algebra 
homomorphism 
$\tilde\rho\colon \mathbb C[G] \rightarrow \End(V)$ by $\displaystyle\sum_g\alpha_g g \mapsto 
\sum_{g}\alpha_g\rho(g)$. 
We know that $\mathcal Z(\mathbb C[G])$, the center of $\mathbb C[G]$, is spanned by the elements $c_{g_1}, \ldots, c_{g_h}$
where 
$$c_{g_i}=\sum_{t\in G \atop t = sg_{i}s^{-1}} t$$
i.e., the sum of all conjugates of $g_i$. Here $g_i$ are representatives of the conjugacy classes.
\begin{exercise}
\begin{enumerate}
 \item Show that $\tilde\rho$ is an algebra homomorphism, that is, it is a vector space 
homomorphism as well as ring homomorphism.
\item Let $\sum_g\alpha_g g\in \mathcal Z(\mathbb C[G])$ then $\alpha_g=\alpha_{sgs^{-1}}$ for any 
$s\in G$.
\item Show that the center of $M_n(k)$ (and $\End(V)$) is the set of all scalar matrices (works 
over any field $k$).
\end{enumerate} 
\end{exercise}
\begin{exercise}
Take $\rho$ to be a $1$-dimensional irreducible representation of a cyclic group. Show that 
$\tilde\rho$ is not injective even if $\rho$ is so.  
\end{exercise}

\noindent We wish to understand how $\tilde\rho$ behaves on the center. 

\begin{proposition}
Let $\rho$ be an irreducible representation and $z\in \mathcal Z(\mathbb C[G])$. Then 
$\tilde\rho(z) = \lambda.I$ for some $\lambda \in \mathbb C$ ($\lambda$ depends on $z$) and $I$ is 
the identity transformation. Further, $\lambda=\frac{Tr(\tilde\rho(z))}{\dim(\rho)}$.
\end{proposition}
\begin{proof}
We claim that $\tilde\rho(z)\in\End(V)$ is a $G$-map and then use Schur's Lemma. Let $z= 
\sum_g \alpha_g g$ then $\alpha_g=\alpha_{sgs^{-1}}$ for any $s\in G$ (as $z$ is in the center). 
Then,
\begin{eqnarray*}
\tilde\rho(z)(\rho(t)v) &=& \tilde\rho\left(\sum_g\alpha_g g\right)(\rho(t)v)= \sum_g\alpha_g 
\rho(g)(\rho(t)v) = \rho(t)\sum_g\alpha_g\rho(t^{-1}gt)(v) \\
&=& \rho(t)\sum_u\alpha_{tut^{-1}}\rho(u)v=\rho(t)\tilde\rho(z)v.
\end{eqnarray*}
In the last step, we use $\alpha_u=\alpha_{tut^{-1}}$. Since $\rho$ is irreducible,  
Corollary~\ref{corschur} (Schur's Lemma) implies that $\tilde\rho(z)=\lambda.I$ for some  
$\lambda\in\mathbb C$.
\end{proof}

With the notation as above let us consider algebra homomorphisms $\tilde\rho_i$ corresponding to 
the irreducible representations $\rho_i$. In the light of above proposition let us denote 
$\tilde\rho_i(c_{g_j})=\lambda_{ij}I_{n_i}$ where $\lambda_{ij}\in \mathbb C$ and 
$I_{n_i}$ denotes the identity transformation. Further, we can take traces on both sides and get, 
$$
n_i\lambda_{ij} =tr(\tilde\rho_i(c_{g_j})) =\sum_{t\in G \atop t=sg_js^{-1}} 
tr(\rho_i(t))=r_j\chi_i(g_j)
$$ 
where $r_j$ is the number of conjugates of $g_j$. 
This gives, 
$$
\lambda_{ij}=\frac{r_j\chi_i(g_j)}{n_i} = r_j \frac{\chi_i(g_j)}{\chi_i(1)}.
$$
Thus, each irreducible representation $\rho_i$ determines $\lambda_i$ a function called {\bf 
central character} given the map $\lambda_i \colon G\rightarrow \mathbb 
C$ defined on the conjugacy classes by $\lambda_i(g_j)= \lambda_{ij}$.
{\color{teal}
\begin{definition}
To an irreducible representation $\rho$, the map $\lambda \colon G\rightarrow \mathbb 
C$ defined on the conjugacy classes by $\lambda(g_j)= r_j \frac{\chi(g_j)}{\chi(1)}$ is said to 
be the {\bf central character} associated to the irreducible representation $\rho$. Clearly, it is 
a class function.
\end{definition}} 

The next proposition says that the values of central characters are algebraic integers. 
\begin{proposition}\label{lambdaij}
Each $\lambda_{ij}$ is an algebraic integer.
\end{proposition}
\begin{proof} 
Let us consider $M=c_{g_1}\mathbb Z \oplus \cdots \oplus c_{g_h}\mathbb Z\subset \mathcal Z(\mathbb C[G]) 
= c_{g_1}\mathbb C \oplus \cdots \oplus c_{g_h}\mathbb C$.
Clearly, $M$ is a $\mathbb Z$-submodule. It is a subring also (multiply $c_{g_j},c_{g_k}$ by 
writing them as elements in $\mathbb C[G]$). Take $c_{g_j},c_{g_k}\in \mathcal Z(\mathbb C[G])$. 
Then $c_{g_j}c_{g_k}\in \mathcal Z(\mathbb C[G])$, in fact $c_{g_j}c_{g_k}\in M$. 
Hence we can write $c_{g_j}c_{g_k}=\displaystyle\sum_{l=1}^{h}a_{jkl}c_{g_l}$ where $a_{jkl}$ are 
integers. 
By applying $\tilde\rho_i$ we get, 
$$
(\lambda_{ij}I_{n_i})(\lambda_{ik}I_{n_i})=\tilde\rho_i(c_{g_j})\tilde\rho_i(c_{g_k})=\tilde\rho_i(
c_{g_j}c_{g_k})
=\sum_{l=1}^{h}a_{jkl}\lambda_{il}I_{n_i}.
$$
This gives $\lambda_{ij}\lambda_{ik}=\sum_{l=1}^{h}a_{jkl}\lambda_{il}$ where each $a_{jkl}\in \mathbb Z$. 

Now we take $N=\lambda_{i1}\mathbb Z\oplus\cdots\oplus \lambda_{ih}\mathbb Z\subset \mathbb C$ which is a finitely
generated $\mathbb Z$-module and $\lambda_{ij}N\subset N$ for all $j$.
This implies $\lambda_{ij}$ is an algebraic integer. 
\end{proof}

\begin{lemma}\label{charint}
For any $s\in G$ and $\chi$ character of a representation, $\chi(s)$ is an algebraic integer.
\end{lemma}
\begin{proof}
Let $s\in G$ be of order $d$ in $G$. Then $\rho(s)$ is of the order less than or equal to $d$. Since the
field is $\mathbb C$ we can choose a basis such that the matrix of $\rho(s)$, say $A$, becomes 
diagonal (see~\ref{diagonal} and also proof in~\ref{character}).
Clearly, $A^d=1$ implies diagonal elements are the root of the polynomial $X^d-1$ and hence are algebraic 
integer (being roots of unity). 
As a sum of algebraic integers is again an algebraic integer we get a sum of diagonals of $A$ which is $\chi(s)$ is an
algebraic integer.
\end{proof}

Now we can prove the main theorem of this chapter,
{\color{red}
\begin{theorem}
The order of an irreducible representation divides the order of the group, i.e., $n_i$ divides $|G|$ for all $i$.
\end{theorem}}
\begin{proof} 
Let $\rho_i$ be an irreducible representation of degree $n_i$ with character $\chi_i$.
From the orthogonality relations we have, $1= \langle \chi_i, \chi_i \rangle = 
\frac{1}{|G|}\sum_{t\in G}\chi_i(t)\chi_i(t^{-1})$. Thus, 
\begin{eqnarray*}
|G| &=&\sum_{j=1}^{h} r_j\chi_i(g_j)\chi_i(g_j^{-1}) 
= \sum_{j=1}^{h}n_i\lambda_{ij}\chi_i(g_j^{-1}).
\end{eqnarray*}
After re-writing we get,  $\displaystyle\sum_{j=1}^{h}\lambda_{ij}\chi_i(g_j^{-1}) 
=\frac{|G|}{n_i}$. The left side of this equation is an algebraic integer (using 
Proposition~\ref{lambdaij} and Lemma~\ref{charint}) and the right side is a rational number. Hence 
$\frac{|G|}{n_i}$ is an algebraic integer as well as an algebraic number, which implies must be an integer. 
This gives the required result, $n_i$ divides $|G|$. 
\end{proof}

\section{Algebraic Integers}
Read about this topic.

\chapter{Burnside's $pq$ Theorem}
As an application of the character theory, we prove Burnside's $pq$ Theorem (see the original 
article~\cite{bu}). Let $p\neq q$ be 
primes. We know that any group of order $p$ is cyclic, of order $p^2$ is Abelian and $p^a$ is 
solvable. In this chapter, we prove that any group of order $p^aq^b$ is also solvable. In general 
groups of order divisible by three different primes need not be solvable (think of $A_5$).

\begin{exercise}\label{pgroups}
A group of order $p^r$, for some $r$, is called a $p$-group.
Let $G$ be a $p$ group of order $p^r$.
Prove the following:
\begin{enumerate}
\item The center of $G$ is non-trivial.
\item $G$ has a subgroup of order $p^s$ for all $0\leq s\leq r$.
\item Prove that for $r=2$, $G$ is always Abelian. 
\item Let $G$ be acting on a finite set $X$. Show that 
$$ |X| \equiv |X^G| \imod p$$
where $X^G=\{x\in X\mid gx=x\forall g\in G\}$.
\item Use the above to show that $ |\mathcal Z(G)| \equiv 0 \imod p$.
\item Show that $G$ is solvable.
\end{enumerate}
\end{exercise}

\noindent We continue with the notation in the previous chapter and recall,
\begin{itemize}
\item the values of central character $\lambda_{ij}=r_j\frac{\chi_i(g_j)}{n_i}$ are algebraic 
integers.
\item the values of characters $\chi_i(t)$ are algebraic integers.
\end{itemize}

\noindent Let us begin with a basic result. 
\begin{lemma}
With the notation as above suppose $r_j$ and $n_i$ are relatively prime. Then, either $\rho_i(g_j)$ is 
in the center of $\rho_i(G)$ or $tr(\rho_i(g_j)) = \chi_i(g_j)=0$.
\end{lemma}
\begin{proof} As in the proof of Proposition~\ref{character} and Lemma~\ref{charint} we can choose a 
basis such that the matrix of $\rho_i(g_j) = \diag\{\omega_1,\cdots,\omega_{n_i}\}$ and hence 
$\chi_i(g_j)= \omega_1+\cdots+\omega_{n_i}$ where $\omega_k$'s are $d$-th root of unity (here $d$ 
is the order of $g_j$). 
Now, $|\omega_1+\cdots+\omega_{n_i}|\leq 1+\cdots+1=n_i$, hence $|\frac{\chi_i(g_j)}{n_i}|\leq 1$.

In the case $|\frac{\chi_i(g_j)}{n_i}|= 1$ we must have $\omega_1=\omega_2=\ldots=\omega_{n_i}$. 
This implies that the matrix of $\rho_i(g_j)= \diag\{\omega_1,\cdots,\omega_{1}\}$ is central, and 
hence in this case $\rho_i(g_j)$ belongs in the center of $\rho_i(G)$.

Now, suppose that $|\frac{\chi_i(g_j)}{n_i}|< 1$. Let us denote $\alpha :=\frac{\chi_i(g_j)}{n_i}$. 
We will show that $\alpha=0$. Since $r_j$ and $n_i$ are relatively prime we can find integers 
$l,m\in \mathbb Z$ such that $r_jl+n_im = 1$. Then, $\lambda_{ij}=r_j\alpha$ gives 
$l\lambda_{ij} = (1-n_im)\alpha=\alpha - m\chi_i(g_j)$.
Since $\lambda_{ij}$ and $\chi_{i}(g_j)$ both are algebraic integers (see 
Proposition~\ref{lambdaij} and Lemma~\ref{charint}) we get $\alpha$ is an algebraic integer. Now we 
use a little bit of Galois Theory (see Section 14.5~\cite{df}) of Cyclotomic extension.  
Let $\zeta$ be a primitive $d$-th root of unity and let us consider the Galois extension 
$K=\Q(\zeta)$ of $\mathbb Q$. 
Let $\omega_k=\zeta^{a_k}$ then we can write $\alpha= \frac{1}{n_i}(\omega_1 +\cdots + 
\omega_{n_i})= \frac{1}{n_i}(\zeta^{a_1}+\cdots + 
\zeta^{a_{n_i}})$.
For $\sigma\in \gal(K/\mathbb Q)$ the element $\sigma(\alpha)$ is also of the same kind, and hence 
$|\sigma(\alpha)| \leq 1$. 
This implies that the norm of $\alpha$ defined by $N(\alpha) = \displaystyle \prod_{\sigma\in 
\gal(K/\mathbb Q)}\sigma(\alpha)$ has $|N(\alpha)|<1$. 
Since $\alpha$ is an algebraic integer so are $\sigma(\alpha)$, and hence the product $N(\alpha)$ is 
an algebraic integer. 
Since $N(\alpha)$ is also invariant under all $\sigma\in\gal(K/\mathbb Q)$ it is a rational number 
hence it must be an integer. However as $|N(\alpha)|< 1$ this gives $N(\alpha)=0$. This eventually 
gives the required result $\chi_i(g_j)=n_i \alpha=0$.
\end{proof}

In what follows we may assume $G$ is non-trivial and of order involving at least two different 
primes.
\begin{proposition}
Let $G$ be a finite group and $C$ be a conjugacy class of $g\in G$.
If $|C|= p^r$, where $p$ is a prime and $r\geq 1$, then there exists a non-trivial irreducible 
representation $\rho$ of $G$ such that $\rho(C)$ is contained in the centre of $\rho(G)$. 
In particular, $G$ is not a simple group.
\end{proposition}
\begin{proof} 
On contrary let us assume that $\rho_i(C)$ is not contained in the center of $\rho_i(G)$ for all 
irreducible representations $\rho_i$ of $G$. Then, from previous Lemma if $(p^r, n_i)=1$, i.e, 
$p\nmid n_i$, then we must have $\chi_i(g)=0$ for $g\in C$.

Consider the character of the regular representation $\chi_{reg} = \sum_{i=1}^h n_i\chi_i$. 
Then for any $1\neq s\in G$ we have 
$$
0=\chi_{reg}(s)=\sum_{i=1}^{h}n_i\chi_i(s)=1+\sum_{i=2}^{h}n_i\chi_i(s).
$$ 
Let us take $g\in C$ (note that $g\neq 1$ since $|C|> 1$). Then  $ \sum_{i=2}^{h}n_i\chi_i(g)=-1$ 
which we may re-write after dividing by $p$  as $\sum_{i=2}^{h}\frac{n_i}{p}\chi_i(g) = 
-\frac{1}{p}$.
On the left hand side the term is either $0$ (when $p\nmid n_i$ then $\chi_i(g)=0$) or an 
algebraic integer (when $p\mid n_i$, $\frac{n_i}{p}$ is an integer). Thus, the overall left side is an 
algebraic integer while the right-hand side is a rational number hence it should be an integer which is 
a contradiction. Hence, there exists a non-trivial irreducible representation $\rho_i$ such that 
$\rho_i(C)$ is contained in the center of $\rho_i(G)$.

Now, let us take the above $\rho\colon G\rightarrow \GL(V)$ and we have $\rho(C)\subset \mathcal 
Z(\rho(G))$ where $C$ is a conjugacy class of non-identity element. We want to show that $G$ can 
not be simple. If $ker(\rho)\neq 1$ it will be a normal subgroup of $G$ and thus $G$ is not simple.
In case $ker(\rho)=1$, $\rho$ is an injective map and $\mathcal Z(\rho(G))\neq 1$. But the centre is 
always a normal subgroup which again implies $G$ is not simple.
\end{proof}

\noindent With these results in hand we are ready to prove our main Theorem. 
{\color{red}
\begin{theorem}[Burnside's Theorem]
Every group of order $p^aq^b$, where $p,q$ are distinct primes, is solvable.
\end{theorem}}
\begin{proof} We use induction on $a+b$. 
If $a+b=1$ then $G$ is a $p$-group and hence $G$ is solvable. 
Now assume $a+b\geq 2$, and any group of order $p^rq^s$ with $r+s < a+b$ is solvable. 
Let $Q$ be a Sylow $q$-subgroup of $G$. 
If $Q=\{e\}$ then $b=0$ and $G$ is a $p$-group and hence solvable (see the Exercise~\ref{pgroups}). 
So let us assume $Q$ is non-trivial. 
Since $Q$ is a $q$-group (prime power order) it has a non-trivial centre. 
Let $1\neq t\in \mathcal Z(Q)$. 
Then 
$$t\in Q\subset C_G(t)\subset G$$
and hence $|C_G(t)|=p^lq^b$ for some $0\leq l\leq a$ 
which gives $[G:C_G(t)]=p^{a-l}$.

First, we claim that $G$ is not simple. 
If $G=C_G(t)$ then $t\in\mathcal Z(C_G(t))=\mathcal Z(G)$, i.e., $\mathcal Z(G)$ is a non-trivial 
normal subgroup and $G$ is not simple. 
Hence we may assume $|C_G(t)|=p^lq^b$ with $l < a$.
Then using the formula we get the size of the conjugacy class of $t$: 
$|C(t)|=\frac{|G|}{|C_G(t)|}=p^{a-l}$. Here $C(t)$ denotes the conjugacy class of $t$. 
Now, using the previous proposition we get $G$ is not simple.

Since we know that $G$ is not simple, it has a proper normal subgroup, say $N$. 
The order of $N$ and $G/N$ both satisfy the hypothesis hence they are solvable. Thus $G$ is solvable 
(see Proposition 10, Section 6.1~\cite{df}). 
\end{proof}

The most amazing and celebrated result in this direction is Feit-Thompson Theorem (see~\cite{ft}): 
Every finite group of odd order is solvable. Thus, non-Abelian finite simple groups must be of even 
order. 

\section{Solvable Groups}
Read about this topic.


\chapter*{}
\vskip10cm
\begin{center}
{\bf \Huge PART -- III}
\end{center}
\vskip5cm

One of the important techniques to produce representations is the method of induction.

\chapter{Induced Representation}\label{ind-rep}

Let $G$ be a finite group and $H$ its subgroup. Let $k$ be a field. Note that the various definitions and constructions make sense over any field $k$. However, whenever we talk about characters we consider only $\mathbb C$-representations. Given a representation 
$\rho\colon G \rightarrow \GL(V)$ we can restrict this representation to $H$ and get $\rho|_H 
\colon H \rightarrow \GL(V)$ where $\rho_H(h)=\rho(h)$. This is also written as $Res_H^G(\rho)$ or 
$\rho\!\downarrow$. The idea is to do the other way round, i.e., to begin with, a representation of 
a subgroup $H$ and get a ``method'' which gives a representation of $G$. There are several ways to 
understand this. We will see some of them here.

\section{Induced Representations from Subgroups}
Let $H$ be a subgroup of $G$, and let $V$ be a $G$-representation. We can restrict the actin of $G$ 
to $H$ and consider $V$ as a $H$-representation. Let $W$ be a $H$-invariant subspace of $V$. Note that even 
though $V$ is irreducible as a $G$ representation it need not be irreducible as a $ H$ representation. We consider $gW\subset 
V$ for any $g\in G$. Notice that for any $h\in H$ the subspaces $gW$ and $ghW$ are the same. Hence, if 
at all, there are possible as many distinct subspaces $gW$ as coset representatives $g\in G/H$. If we consider $W'=\displaystyle \sum_{g\in 
G/H} gW$ then $W'\subset V$ is a $G$-subspace. Now, we are ready to define induced representation.
{\color{teal}
\begin{definition}\label{ind-def1}
The $G$-representation $V$ is called {\bf induced} from the $H$-subrepresentation $W$ if 
$V=\displaystyle\bigoplus_{g\in G/H} gW$ where sum varies over representatives of the left cosets 
of $H$. The induced module $V$ is denoted by $Ind_H^GW$ or $Ind(W)$ or $W\!\uparrow$.
\end{definition}}
\noindent Let us try to understand what $V$ looks like if it is an induced module. Let 
$s_1=e,s_2,\ldots, s_r$ be a set of representatives of distinct cosets of $G/H$, i.e, $G=H 
\cup s_2H \cup \ldots \cup s_rH$. Then, 
\begin{proposition}\label{ind-dimension}
Let $G$ be a group with $H$ a subgroup. Let $V$ be a $G$-representation and $W$ be an $H$-subrepresentation of $V$. 
Suppose $V=Ind_H^GW$. Then, 
\begin{enumerate}
\item $s_iW=s_ihW$ for any $h\in H$. 
 \item $s_iW\bigcap s_jW =\{0\}$ when $i\neq j$.
 \item the $G$ action on the set $\{s_iW \mid s_i \in G/H \}$ permutes $s_iW$'s (that is acts transitively with stabiliser $H$).
 \item $\dim V = |G/H| \dim W$.
\end{enumerate}
Here $s_1=e,s_2,\ldots, s_r$ is a set of representatives of distinct cosets of $G/H$.
\end{proposition}
\begin{proof}
Since $W$ is a $H$-representation, it is clear that $s_iW=s_ihW$ for any $h\in H$. 

Since $V=Ind_H^GW = \displaystyle\bigoplus_{g\in G/H} gW$ it follows that $\{x\in G \mid xW=W\} = 
H$. We can also define $G$ action on $\{s_iW \mid 1\leq i \leq r\}$ by $g.s_iW = (gs_i)W$. This is 
a transitive action and the stabiliser of $W$ is $H$. 

Fix a basis of $W$, say, $w_1, \ldots w_l$. Then, a basis of $V$ would be 
$$w_1, \ldots, w_l, s_2w_1, \ldots, s_2w_l, \ldots, s_rw_1, \ldots s_rw_l$$
where $s_iw_1, \ldots, s_iw_l$ forms a basis of $s_iW$. Thus, $\dim V = |G/H| \dim W$.
\end{proof}
\begin{exercise}
With notation as in the proof above, show that 
$$\{w_1, \ldots, w_l, s_2w_1, \ldots, s_2w_l, \ldots, s_rw_1, \ldots s_rw_l \}$$
is a basis of $V$. In fact, $s_iw_1, \ldots, s_iw_l$ forms a basis of $s_iW$. 
\end{exercise}

Let us try to understand how $G$-representation action would be on $V=Ind(W)$ in terms of the given $H$-representation $W$. What really happens is that the induced module action is fully determined by two things: 
\begin{itemize}
\item $G$ action on $G/H$, and 
\item $H$ action on $W$. 
\end{itemize}
Since $G \times G/H \rightarrow G/H$ given by $g.xH = 
(gx)H$ is a permutation action, let us write this as follows: $g.s_iH = s_{\sigma(i)}H$ where $\sigma\in 
S_r$ depends on $g$. That is, $gs_i = s_{\sigma(i)} h_i$ for some $h_i\in H$. Now we can 
write the $G$-module action on the basis vectors as follows: 
$$g.(s_iw_j) = s_{\sigma(i)}(h_i w_j)$$
where $s_iw_j$ are basis vectors of $V$ (as in the proof of the Proposition above). Let us work out some examples.
\begin{example}
Consider the regular representation of $G$ on $\mathbb C[G]=\{\sum \alpha_ge_g\mid \alpha_g\in \mathbb C\}$ given by $x.e_g = e_{xg}$. Now consider the trivial subgroup $H=\{1\}$ and $W=\{\alpha e_1 \mid \alpha\in \mathbb C\}$. We claim that $\mathbb C[G]=Ind_H^GW$. This can be seen following the explanation above. The coset representatives of $G/H$ would be the elements of $G$ with group multiplication as an action. So, the basis of $Ind(W)$ would be $\{se_1=e_s \mid s\in G\}$ and the representation would be $g.se_1 = e_{gs}$ which is the regular representation. Thus, regular representation is an induced representation, induced from a $1$-dimensional representation of the trivial subgroup. 
\end{example}

\begin{example}\label{indz4}
Let $G=\mathbb Z/4\mathbb Z $ and $V=\mathbb Ce_1 + \mathbb Ce_2$ be the representation $\mathbb 
Z/4\mathbb Z \rightarrow \GL_2(\mathbb C)$ given by $\bar 1\mapsto \begin{pmatrix} i& \\ & -i 
\end{pmatrix}$. Consider the subgroup $H=\langle \bar 2\rangle \cong \mathbb Z/2\mathbb Z$ and 
$W=<e_1>$. Since $W$ is invariant under whole $G$, we see that $\sum_{g\in G/H} gW = W$. Thus, $V$ 
is not induced by $W$ in this case. 
\end{example}

\begin{example}
Once again let us consider the regular representation of $G$ on $\mathbb C[G]=\{\sum \alpha_ge_g\mid \alpha_g\in \mathbb C\}$ given by $x.e_g = e_{xg}$. Now consider any subgroup $H$ and $W=\{\sum_{h\in H}\alpha_{h} e_h \mid \alpha_h \in \mathbb C\}$. Clearly, $W$ is a $H$ subrepresentation. We claim that $\mathbb C[G]=Ind_H^GW$. This can be seen as follows. The coset representatives of $G/H$, say are $s_1, \ldots, s_r$ and the basis vectors of $W$ are $\{e_h \mid h\in H\}$. Thus, basis vectors of $Ind_H^G W$ would be $\{s_ie_h \mid h\in H, i=1, \ldots, r\} =\{e_g \mid g\in G\}$. Further, the representation would be $g.s_ie_h = ge_{s_ih} =e_{gs_ih}$ which is the regular representation. Thus, regular representation is induced representation, induced from $H$-representation $W=\{\sum_{h\in H}\alpha_{h} e_h \mid \alpha_h \in \mathbb C\}$.
\end{example}

\begin{example}
Let us consider the representation of the dihedral group $G=D_m=\langle a,b \mid a^m=1=b^2, ab=ba^{m-1} \rangle$ on $V=\mathbb Ce_1\oplus \mathbb Ce_2$ defined by:
$$a \mapsto \left[\begin{matrix} e^{\frac{2\pi i }{m}} &  \\ & e^{-\frac{2\pi i}{m}}\end{matrix}\right], 
b \mapsto \left[\begin{matrix} 0 & 1 \\ 
1 & 0\end{matrix}\right].$$
Consider the representation of $H= \langle a\rangle \cong \mathbb Z/m\mathbb Z$ on the subspace $W=<e_1>$. We claim that $Ind_H^GW = V$. Let us verify this. We can take coset representatives of $G/H$ to be $\{1, b\}$. Then, the basis vectors for $Ind_H^GW$ would be $\{e_1, be_1=e_2\}$.
Now let us determine the action of $G$. Check that $a.e_1= e^{\frac{2\pi i }{m}} e_1$ (this is given) and  $a.be_1= ba^{-1} e_1 = e^{-\frac{2\pi i }{m}}b  e_1 = e^{-\frac{2\pi i }{m}} e_2$. Similarly, $be_1=e_2$ and $b.be_1 = 1.e_1=e_1$. Note that in this case an irreducible representation, after induction, gives an irreducible representation.
\end{example}

\begin{exercise}
Given $H$ and $W$ there is a unique $G$-representation $Ind_H^GW$.
\end{exercise}

\section{Construction of Induced Representation}
This formula suggests that we can always construct an induced $G$-module from a given $ H$ module 
$W$. We don't need to have a $G$-representation $V$ a priori. Let us do this. Let $s_1, \ldots, s_r$ 
be representatives of $G/H$. Consider, $V=W\oplus W\oplus \cdots \oplus W$, $r$-copies. We have 
$G\times G/H \rightarrow G/H$ given by $g.xH=(gx)H$. Now, define the $G$-representation as follows:
$$g.we_i = h_i(w)e_j$$
where $gs_iH = s_jH$ gives $gs_i = s_jh_i$ for some $h_i\in H$. That is, $(0, \ldots, 0, w, 0, 
\ldots,0)$ with $w$ at $i$-th place gets mapped to  $(0, \ldots, 0, h_i(w), 0, \ldots,0)$ with 
$h_i(w)$ at $j$-th place.
\begin{exercise}
Check that we have a $G$-representation with its character given by the formula of induced 
character.
\end{exercise}

\begin{example}
Consider the group $Q_8$ and $H=\langle i \rangle \cong \mathbb Z/4\mathbb Z$. Take the 
$1$-dimensional representation $\theta\colon H \rightarrow \mathbb C^*$ given by $ i \mapsto \iota$ (here the first $i$ is an element of $Q_8$, and the complex one is denoted by $\iota$ to avoid confusion). What's 
the induced representation $Ind_H^{Q_8}\theta$? Let the representation of $H$ be $W=<w>$, i.e, $i.w=\iota w$. Let us fix a representatives 
of $Q_8/H$ to be $\{1,j\}$ and a basis of $Ind(W)$ to be $\{1w, jw\}$. Now, for $i\in Q_8$, $i.1w = 
i.w = \iota w$ and $i.jw = -jiw = -j \iota w=-\iota jw$, hence $i\mapsto \begin{pmatrix} \iota & \\ 
& -\iota\end{pmatrix}$. Similarly, for $j \in Q_8$, $j.1w = jw$ and $j.jw = -w$, hence $j \mapsto 
\begin{pmatrix}  & -1 \\ 1& \end{pmatrix}$. Note that here we get an irreducible representation 
after induction.
\end{example}

\begin{example}
Let us look at Example~\ref{indz4} again. Fix the coset representatives of $\mathbb Z/4 \mathbb 
Z / \langle 2\rangle$ to be $\{\bar 0,\bar 1\}$. Take the representation of $\langle \bar 2\rangle 
\cong \mathbb Z/2 \mathbb Z /$ to be $\langle \bar 2\rangle \rightarrow \mathbb C^*$ given by $\bar 
2 \mapsto -1$. Suppose the representation space is $W=<w>$. Take a basis of $Ind(W)$ to be $\{w, 
\bar1 w\}$ which would be a $2$-dimensional representation given by $\bar 1. w = \bar1w$ and $\bar 
1. \bar 1w = \bar 2 w = -w$. Thus, the induced representation is given by $\bar 1 \mapsto 
\begin{pmatrix}  & -1 \\ 1& \end{pmatrix}$. 
\end{example}

\chapter{Character of the Induced Representation}
As we have seen earlier that the induced character is determined by the $G$ action on $G/H$ together with the $H$ representation $W$. So, if we know the character of $H$-representation $W$, say $\theta$, we should be able to determine the character of the 
induced representation $V=Ind(W)$, namely $Ind(\theta)$ (or $Ind_H^G\theta$).
{\color{red}
\begin{theorem}\label{indchar}
 Let $\theta$ be the character of the $H$-representation $W$. Then, the character 
$\chi=Ind_H^G\theta$ is given as follows:
$$Ind_H^G\theta(g) = \chi(g) = \sum_{s_i \atop s_i^{-1}gs_i \in H} \theta(s_i^{-1}gs_i) = 
\frac{1}{|H|}\sum_{x\in G \atop xgx^{-1}\in H}\theta(xgx^{-1})$$
where $s_1=e,s_2,\ldots, s_r$ is a set of representatives of distinct cosets of $G/H$.
\end{theorem}}
\begin{proof}
For $g\in G$, we need to compute $\chi(g)=tr(g)$ on $V$. For this, we make use of $V=s_1W\oplus 
s_2W\oplus\cdots \oplus s_rW$ on which $G$ acts by permutation. Let us consider the basis $\{w_1, \ldots, w_l, s_2w_1, \ldots, s_2w_l, \ldots, 
s_rw_1, \ldots s_rw_l\}$. More precisely, the action is given by $g.(s_iw_j) = s_{\sigma(i)}(h_i w_j)$ where $gs_i = 
s_{\sigma(i)} h_i$ for some $h_i\in H$. Imagine that we have to write the matrix of $g$ with respect to this basis. When $g.s_iW\neq s_iW$ it contributes $0$ to the trace. So, the only contribution for the trace of $g$ will come when  $g.s_iW=s_iW$. We do this more precisely below.

When $g.s_i W = s_jW$ with $i\neq j$ it won't contribute to the trace of $g$. Thus, we need to look 
at the components where $g.s_i W = s_iW$. In this case, $gs_i = s_i h_i$ for some $h_i\in H$, that 
is, $s_i^{-1}gs_i \in H$. Thus,
\begin{eqnarray*}
 \chi(g) &=& tr(g) = \sum_{s_i \atop g.s_iW=s_iW} tr(g|_{s_iW})\\
 &=& \sum_{s_i\atop s_i^{-1}gs_i\in H} tr(s_i^{-1}gs_i) = \sum_{s_i\atop s_i^{-1}gs_i\in H} 
\theta(s_i^{-1}gs_i).
\end{eqnarray*}
The second equality follows as the size of each coset is $|H|$.
\end{proof}

\begin{example}
Consider $H=\{1\}$ the trivial subgroup of $G$. The, character of the trivial representation on $H$ is $\theta\colon H \rightarrow \mathbb C$ given by $\theta(1)=1$. We can use the above formula to compute the character $Ind_H^G\theta$ as follows. First of all, $Ind_H^G\theta(1) = \displaystyle \frac{1}{|H|}\sum_{x\in G \atop x1x^{-1}\in H}\theta(x1x^{-1}) = \sum_{x\in G }\theta(1) =|G|$, and for $g\neq 1$, $Ind_H^G\theta(g) = \displaystyle \frac{1}{|H|}\sum_{x\in G \atop xgx^{-1}\in H}\theta(xgx^{-1}) = 0$ as $xgx^{-1}\in H$ if and only if $g=1$. Thus, $Ind_H^G\theta$ is the regular character on $G$. 
\end{example}

\begin{example}
Let us consider the subgroup $H= \langle a\rangle \cong \mathbb Z/m\mathbb Z$ of the dihedral group $G=D_m=\langle a,b \mid a^m=1=b^2, ab=ba^{m-1} \rangle$. Take the character $\theta\colon H \rightarrow \mathbb C$ given by $\theta(a)= e^{\frac{2\pi i }{m}}$. We can take coset representatives of $G/H$ to be $\{1, b\}$.  Then, $Ind_H^G\theta(1) = \displaystyle \frac{1}{|H|}\sum_{x\in G \atop x1x^{-1}\in H}\theta(x1x^{-1}) = \frac{1}{m}\sum_{x\in G }\theta(1) =\frac{2m}{m}=2$. Similarly, $Ind_H^G\theta(a) = \displaystyle \sum_{s_i \atop s_i^{-1}as_i \in H} \theta(s_i^{-1}as_i) =   \theta(a) +\theta (b^{-1}ab) =     \theta(a) + \theta(a^{-1}) = e^{\frac{2\pi i }{m}}+    e^{-\frac{2\pi i}{m}}$, and $Ind_H^G\theta(b) = \displaystyle \sum_{s_i \atop s_i^{-1}bs_i \in H} \theta(s_i^{-1}bs_i) =  0$ as $1b1\notin H$ as well as $b^{-1}bb\notin H$. Note that in this case an irreducible character, after induction, gives an irreducible character.
\end{example}

Given $f\in \mathbb C[H]$ we can extend $f$ to $\dot f\in \mathbb C[G]$ as follows: $\dot f(g)=f(g)$ if $g\in H$ and $\dot f(g)=0$ otherwise. With this notation, we can introduce the following map 
\begin{eqnarray*}
Ind_H^G &\colon& \mathbb C[H] \rightarrow \mathbb C[G] \\ 
Ind_H^G f(g) &=& \sum_{s_i\in G/H} \dot f(s_i^{-1}gs_i) = 
\frac{1}{|H|}\sum_{x\in G}\dot f(xgx^{-1}).
\end{eqnarray*}
Thus, in the view of the above theorem, a character gets mapped to a character.

\section{Induced Representation via Tensor Product}

Let $G$ be a group and $H$ be a subgroup. Let $k$ be a field. Suppose we are given $W$ a $k[H]$-module which amounts to having a representation of $H$. We would like to construct a $k[G]$-module. Since $H$ is a subgroup of $G$, we already have $k[G]$ a $k[H]$-module (think of ring/algebra-extension).
{\color{teal}
\begin{definition} 
Now we define induced representation using the tensor product as follows:
$$Ind_H^G W := k[G]\otimes_{k[H]} W$$
which is naturally a $k[H]$-module. We make it a $k[G]$-module as follows: for $\phi \in k[G]$ define $\phi. (\sum_i f_i \otimes w_i) := (\sum_i \phi f_i)\otimes w_i$. 
\end{definition}}

\begin{exercise}
Show that with the above definition $Ind_H^G W$ is a $k[G]$-module.
\end{exercise}
\begin{exercise}
Show that $k[G]$ is a free $k[H]$-module of rank $|G/H|$, and use this to show that the two definitions of induced modules are the same. We can do this by choosing coset representatives $\{s_1, \ldots, s_r\}$ of $G/H$ which ensures that every element of $k[G]$ is of the form $\beta_1s_1+\cdots \beta_r s_r$ where $\beta_j\in k[H]$. Then, 
$$\sum (\beta_1s_1+\cdots +\beta_r s_r)\otimes w =  \sum \beta_1s_1 \otimes w +\cdots +\beta_r s_r\otimes w= \sum s_1\otimes\beta_1 w + \cdots +s_r \otimes \beta_r w.$$
Again, by fixing a basis of $W$, say, $w_1, \ldots, w_l$, we get a basis 
$$\{s_1\otimes w_1, \ldots, s_r\otimes w_1, \ldots, s_1\otimes w_l, \ldots, s_r\otimes w_l\}$$
of $Ind_H^GW$.
\end{exercise}
Thus, we see that this definition is the same as the Definition~\ref{ind-def1}. This definition is quite conceptual and makes it clear that the construction of an induced module can be naturally done. Thus,
$Ind_H^G$ is a functor which sends any $k[H]$-module $W$ to the $k[G]$-module $k[G]\otimes_{k[H]}W$. This is quite a in sync with the $Ind_H^G$ sending any class function on $H$ to a class function on $G$.

\section{Induced Representations via Invariant-functions}

There is yet another way to define induced representations via invariant maps. Each of these definitions has the advantage of being generalised to different setups. The following definitions are more often used when we are dealing with infinite groups, such as compact groups, Lie groups etc. Once again we begin with a subgroup $H$ of $G$ and a representation $W$ of $H$.  We take 
$$V=\{f\colon G\rightarrow W \mid f(xh^{-1})=hf(x)\}$$ 
the set of $H$ invariant maps from $G$ to $W$. We define $G$ action on $V$
as follows: $(gf)(x)=f(g^{-1}x)$. We first check that $gf\in V$. Notice that if we fix $\{s_1,\ldots s_r\}$ as left coset representatives for $G/H$ then the maps in $V$ are defined once they are specified on $s_i$'s.
\begin{exercise}
 \begin{enumerate}
  \item Prove that $gf\in V$. Note that $(gf)(xh^{-1}) = f(g^{-1}xh^{-1})=hf(g^{-1}x)=h(gf)(x)$.
  \item $f\in V$ is defined by its values on $\{s_1, \ldots, s_r\}$. Because, $f(g)=f(s_ih)=h^{-1}f(s_i)$ where $g\in s_iH$.
  \item Show that $V$ is a vector space of dimension $r.\dim(W)$.
\item Prove that $V\cong Ind_H^GW$. For this, consider the functions $f_i\in V$ which map $s_iH$ to $W$ and other elements of $G$ to $0$.
 \end{enumerate}
\end{exercise}
{\bf Warning:} If we take $V$ as left-$H$-invariant functions then we need to make $G$ act on the right. Be aware of which definition you are using. 

Now we are going to see that the induced module satisfies a certain kind of universal property:
\begin{proposition}\label{ind-universal}
Let $H$ be a subgroup of $G$ and $W$ be a $H$-module. Then, for any $G$ module $M$ and any $H$-homomorphism $\phi \colon W\rightarrow M$ there exists a unique 
$G$-homomorphism $\tilde\phi \colon Ind_H^G W\rightarrow M$ such that the following diagram commutes:
$$
\xymatrix{
W\ar[rd]_{\phi} \ar[rr]& & Ind_H^G W \ar@{-->}[ld]^{\tilde\phi} \\
&M&
}
$$
where $W$ embeds in $Ind_H^G W$ as $1.W$ component. 
Thus, $Hom_H(W,Res_H^G M) \cong Hom_G(Ind_H^G, M)$ as a $k$-vector space.
\end{proposition}
\begin{proof}
Given $\phi$ we define $\tilde\phi\colon k[G]\otimes W \rightarrow M$ as follows: $\tilde\phi(\sum \alpha\otimes w) = \sum \alpha\phi(w)$. Check that $\tilde\phi$ is a $k[G]$-module homomorphism. Now we know that any map on the induced module can be determined by its value on $W$.
\end{proof}
\begin{exercise}
 Prove that $Ind_H^GW$ is a relatively free $kH$-module, i.e., prove the above proposition.
\end{exercise}
\begin{exercise}
Show that induction is transitive. That is, if $H$ is a subgroup of $G$ and $K$ is a subgroup of 
$H$, i.e, $K\subset H \subset G$, then for any $K$-module $W$ we have
$$Ind_H^G Ind_K^H W = Ind_K^G W.$$
\end{exercise}

\begin{exercise}
We have the permutation representation $\rho$ of $S_n$ on the set $X=\{1, 2, \ldots, n\}$. Consider $S_n$ as a subgroup of $S_{n+1}$ by fixing the last symbol $n+1$. Determine the representation and character of $Ind_{S_n}^{S_{n+1}}\rho$. 
\end{exercise}

\chapter{Fr\"{o}benius Reciprocity}

Let $G$ be a group and $H$ be a subgroup. We can define the restriction map $Res_H^G \colon \mathbb C[G] \rightarrow \mathbb C[H]$ by simply restricting the function $f\colon G \rightarrow \mathbb C$ to $f|_H=: Res_H^G(f)$.  

We can also define the induction map $Ind_H^G \colon \mathbb C[H] \rightarrow \mathbb C[G]$ as follows (see Theorem~\ref{indchar}). For any function $f$ on $H$ we define induced function $Ind_H^Gf \colon G\rightarrow \mathbb C$ by
$$Ind_H^Gf(g) = \sum_{i=1}^r \dot f(s_igs_i^{-1}) = \frac{1}{|H|}\sum_{x\in G} \dot f(xgx^{-1})$$
where $\dot f$ is an extension of $f$ on $G$ defined to be $0$ outside $H$.
\begin{exercise}
\begin{enumerate}
\item Prove that if $f$ is a class function on $H$ then $Ind_H^Gf$ is a class function on $G$. This is clear if we look at the second formula.
\item Prove that if $f$ is a character on $H$ then $Ind_H^Gf$ is so. We have proved this in the 
last chapter that $Ind_H^Gf$ is a character of the induced representation.
\end{enumerate}
\end{exercise}

Recall that for a group $\mathcal G$ we define the following inner product on the class functions:
$$\langle \phi_1,\phi_2 \rangle_{\mathcal G} = \frac{1}{|\mathcal G|}\sum_{t\in \mathcal G}\phi_1(t)\overline{ \phi_2(t)}.$$
The Fr\"{o}benius Reciprocity theorem tells us that $Res_H^G$ and $Ind_H^G$ are adjoints of each 
other. More precisely we have the following,
{\color{red}
\begin{theorem}[Fr\"{o}benius Reciprocity]
Let $G$ be a group and $H$ be a subgroup. Let $f$ be a class function on $H$ and $\psi$ be a class function on $G$. Then, 
$$\langle Ind_H^G f,\psi\rangle_G = \langle f, Res_H^G \psi \rangle_H.$$
\end{theorem}}
\begin{proof}
The proof is computational using the formula.
\begin{eqnarray*}
\langle Ind_H^G f,\psi\rangle_G &=&  \frac{1}{|G|}\sum_{t\in G} Ind_H^Gf(t)\overline{ \psi(t)} = \frac{1}{|G||H|}\sum_{t\in G}  \sum_{x\in G} \dot f(xtx^{-1}) \overline{ \psi(t)}\\
&=& \frac{1}{|G||H|}\sum_{x\in G}  \sum_{t\in G} \dot f(xtx^{-1}) \overline{ \psi(t)}.
\end{eqnarray*}
Now, note that $\dot f$ is zero outside $H$, so, the non-zero terms in the inside sum would appear only when $xtx^{-1}\in H$. That is, when $xtx^{-1}=h$ or $t=x^{-1}hx$ for some $h\in H$. Thus, 
\begin{eqnarray*}
\langle Ind_H^G f,\psi\rangle_G &=& \frac{1}{|G||H|}\sum_{x\in G}  \sum_{t\in G} \dot f(xtx^{-1}) \overline{ \psi(t)} = \frac{1}{|G||H|}\sum_{x\in G}  \sum_{h \in H} \dot f(h) \overline{ \psi(x^{-1}hx)}\\
&=& \frac{1}{|G||H|}\sum_{x\in G}  \sum_{h\in H} f(h) \overline{ \psi(h)} = \frac{1}{|G|}\sum_{x\in G}  \langle f, Res_H^G \psi \rangle_H = \langle f, Res_H^G \psi \rangle_H.
\end{eqnarray*}
This completes the proof.
\end{proof}

Let us also define $\langle V_1,V_2\rangle_{\mathcal G} := \dim_{\mathbb C}\Hom_{\mathcal G}(V_1,V_2)$.

\begin{exercise}
 \begin{enumerate}
  \item Let $\chi_1$ and $\chi_2$ be characters of $V_1$ and $V_2$ respectively. 
Then $\langle \chi_1,\chi_2\rangle=\langle V_1,V_2\rangle.$
\item Prove that $\langle W, Res_H^G V\rangle_H = \langle Ind_H^G W,V\rangle_G$ where $W$ is a $H$-representation and $V$ is a $G$-representation.
 \end{enumerate}
\end{exercise}

An interesting question to study is if we can get all possible characters of a given group $G$ by induction of representations of proper subgroups. Let us define $R(G)$ to be the free Abelian subgroup generated by all irreducible characters of $G$, i.e., $R(G)=\mathbb Z\chi_1 + \cdots + \mathbb Z \chi_h$. Notice that $R(G)\otimes \mathbb C = \mathcal H$, the space of class functions. Thus, $Res_H^G \colon R(G) \rightarrow R(H)$ and $Ind_H^G \colon R(H) \rightarrow R(G)$.
\[\xymatrix{ R(H) \ar@/^2pc/[rr]|{Ind_H^G}
&& R(G) \ar@/^2pc/[ll]|{Res_H^G} }\]
\vskip1mm 
Now, we can take $\displaystyle\prod_{H< G} R(H) \rightarrow R(G)$ and ask the question if it is a surjective map. There are a couple of theorems due to Brauer and Artin (see~\cite{se} Part 2) which answer these questions. We illustrate this with an example.

\begin{example}
 Let us consider the Kleins four group $G=V_4=\{e,a,b,c\}$. Consider the subgroups $H_1=<a>, 
H_2=<b>, H_3=<c>$. Let $\chi_1, \chi_2, \chi_3$ be the non-trivial characters of $H_1, H_2$ and $H_3$ 
respectively. Let ${\bf 1} =\phi_1, \phi_2, \phi_3$ and $\phi_4$
are the characters of $G$. 
\vskip1mm
\[\begin{array}{c|c|c|c|c|}
        & r_1=1 & r_2=1  & r_3=1   & r_4=1  \\
        & e & a & b &   c    \\ \hline
{\bf 1} = \phi_1  & n_1=1 & 1      & 1     &  1       \\ \hline
\phi_2  & n_2=1 & 1     & -1    &   -1       \\ \hline
\phi_3  & n_3=1 & -1     &  1   &  -1     \\ \hline
\phi_4 & n_4 = 1& -1     & -1    &  1    \\ \hline
\end{array}\]
\vskip1mm
Then $Ind_{H_1}^G {\bf 1} = {\bf 1} +\phi_2$ and $Ind_{H_1}^G \chi_1=\phi_3+\phi_4$ and similarly 
others. Let us check this: 
$Ind_{H_1}^G {\bf 1}(g) =  \dot {\bf 1} (g) + \dot {\bf 1}(b^{-1}gb) = 2\dot {\bf 1}(g) = 
({\bf 1} +\phi_2)(g)$ and $Ind_{H_1}^G \chi_1(g) =  \dot \chi_1(g) + \dot \chi_1(b^{-1}gb) = 2\dot 
\chi_1(g)  = (\phi_3+\phi_4)(g)$. 
 This defines a maps
$$\prod_{i=1}^3 R(H_i)\rightarrow R(G)$$
of which image contains ${\bf 1}+\phi_2, {\bf 1} +\phi_3, {\bf 1} +\phi_4, \phi_3+\phi_4, 
\phi_2+\phi_4, \phi_2+\phi_3$. We claim that the image is of index
$2$. We check that $2\phi_j$ belongs to the image and $1$ doesn't. However, the image is the whole of it if we work over $\mathbb Q$.
\end{example}

\begin{example}
 In the case of the dihedral group $D_m$ consider $H=\langle a \rangle$. Let $\theta_r$, for $0\leq r 
\leq m-1$, be the character of $H$ given by $a\mapsto e^{-\frac{2\pi ir}{m}}$. Then, for $r\neq 0$, 
$Ind_H^G\theta_r $ is an irreducible character of dimension $2$ when $\theta_r(a)\neq \pm 1$. Note 
that $r$ and $m-r$ give the same character.  Determine the image $R(H) \rightarrow R(G)$.
\end{example}

\begin{exercise}
Determine the induced representations from the Abelian subgroups for $G=S_3,S_4, Q_8$ and 
determine the subgroup in $R(G)$. 
\end{exercise}

\begin{exercise}
Consider the following character $\phi$ of $S_3$. 
\[\begin{array}{c|c|c|c|}
        & r_1=1 & r_2=3    & r_3=2 \\
        & g_1=1 & g_2=(12) & g_3=(123)   \\ \hline
\phi  & 2 & 0        & -1   \\ \hline
\end{array}\]
Consider $S_3$ as a subgroup of $S_4$ by fixing the last symbol $4$. Recall the character table of 
$S_4$ as follows: 
\[\begin{array}{c|c|c|c|c|c|}
          & r_1=1 & r_2=6  &r_3=3   & r_4=8 &r_5=6 \\
          & g_1=1 & g_2=(12) & g_3=(12)(34) &   g_4=(123) &g_5=(1234)   \\ \hline
\chi_1  & 1&1 &1 &1 &1   \\ \hline
\chi_2  & 1&-1 &1 &1 &-1   \\ \hline
\chi_3 & 3& 1 &-1 &0&-1   \\ \hline
\chi_4 & 3& -1 &-1 &0 &1   \\ \hline
\chi_5 & 2 &0&2&-1&0\\\hline
\end{array}\]\vskip2mm  
\begin{enumerate}
\item Use the Fr\"obenius reciprocity to show that $\langle Ind_{S_3}^{S_4}\phi, \chi_i \rangle= 
\langle\phi, Res_{S_3}^{S_4} \chi_i \rangle$ is $0$ for $i=1,2$ and is $1$ for $i=3,4,5$. Thus, 
conclude that $Ind_{S_3}^{S_4}\phi =\chi_3+\chi_4+\chi_5$. 
 \item Compute the character $Ind_{S_3}^{S_4}\phi$ using the formula and show that 
\[\begin{array}{c|c|c|c|c|c|}
          & r_1=1 & r_2=6  &r_3=3   & r_4=8 &r_5=6 \\
          & g_1=1 & g_2=(12) & g_3=(12)(34) &   g_4=(123) &g_5=(1234)   \\ \hline
Ind_{S_3}^{S_4}\phi  & 8&0 &0 &-1 &0   \\ \hline
\end{array}\]
Now, we can also directly write $Ind_{S_3}^{S_4}\phi =\chi_3+\chi_4+\chi_5$ decomposed as a sum of 
irreducible characters of $S_4$. 
\end{enumerate}
\end{exercise}

\chapter{Mackey's Irreducibility Criteria}

One of the purposes of defining induced representation is to be able to construct irreducible 
representations of the given group $G$. One wonders if knowing irreducible representations of 
various proper subgroups is of some help. Let $H$ be a subgroup of $G$ with a representation 
$(\phi, W)$ and character $\theta$. Let $Ind_H^GW$ be the induced representation of $G$ with the 
corresponding character $Ind_H^G\theta$. To check if $Ind_H^G\theta$ is irreducible we need to 
compute 
$$\langle Ind_H^G\theta, Ind_H^G\theta \rangle_G = \langle \theta, Res_H^G Ind_H^G\theta \rangle_H$$
and see if this is $1$. Thus, in this process we need to understand $Res_H^G Ind_H^G\theta$, more 
generally, the decomposition of $Res_H^G Ind_H^G W$. For this, we take up a slightly more general 
situation as follows.

\section{Restriction of an Induced Representation}
Let $H$ and $K$ be two subgroups of $G$. We are given a representation $(\phi, W)$ of $H$. We would 
like to understand the restriction of the $G$-representation $Ind_H^G W$ to the subgroup $K$, that 
is, $Res_K^G Ind_H^G W$. To understand this we would require the concept of {\bf double coset 
decomposition} of $G$ with respect to $(K,H)$. That is, we define a relation $\sim$ on $G$ as 
follows: for $g,g'\in G$ we say $g\sim g'$ if there exists $k\in K$ and $h\in H$ such that $g'=kgh$. 
The equivalence classes are called $(K,H)$-double cosets of $G$. Let $S=\{s_1, \ldots, s_l\}$ be a 
set of representatives of $(K,H)$-double cosets of $G$. 
\begin{exercise}
Show that $\sim$ is an equivalence relation on $G$. Further, $G$ is a disjoint union of $KsH$ 
where $s\in S$. 
\end{exercise}
\noindent Equivalently, we can write $s\in K\backslash G/H$. Now, for any $s\in S$ we denote the 
subgroups $H_s:= sHs^{-1}\cap K \leq K$.  

\begin{exercise}
Consider the group $\GL_2(q)$ and its Borel subgroup $\mathfrak B$ consisting of all upper 
triangular matrices, i.e., $\mathfrak B = \left\{ \begin{pmatrix} a&b\\ &c 
\end{pmatrix} \mid a, c\in \mathbb F_q^*, b\in \mathbb F_q \right\}$ . Compute $\B\backslash
\GL_2(q)/\B$ by showing that a set of representatives can be chosen to be $\left\{I, 
\begin{pmatrix} 
&1\\ 1& \end{pmatrix}\right\}$.
\end{exercise}
\begin{exercise}
Consider the group $\SL_2(q)$ and its Borel subgroup $\B$ consisting of all upper 
triangular matrices, i.e., $\B = \left\{ \begin{pmatrix} a&b\\ &a^{-1} 
\end{pmatrix} \mid a\in \mathbb F_q^*, b\in \mathbb F_q \right\}$ . Compute $\B\backslash
\SL_2(q)/\B$ by showing that a set of representatives can be chosen to be $\left\{I, 
\begin{pmatrix} 
&-1\\ 1& \end{pmatrix}\right\}$.
\end{exercise}

\begin{exercise}
Consider the Dihedral group $D_n=\langle a, b \mid a^n=1=b^2, ab=ba^{-1}\rangle$ and the subgroup 
$N=\langle a \rangle$. Compute $N\backslash D_n/N$.
\end{exercise}
\begin{exercise}
Consider the group $\GL_n(q)$ and its Borel subgroup $B_n$ consisting of all upper triangular 
matrices. Compute $B_n\backslash \GL_n(q)/B_n\cong S_n$ by showing that the representatives can be 
chosen to be the permutation matrices.
\end{exercise}

Now, since we are given a representation $(\phi, W)$ of $H$ we can consider representations $s\phi 
s^{-1}$ on $sHs^{-1}$. Let us set the following notation for this: $(\phi_s,W_s)$ a representation 
of $H_s:= sHs^{-1}\cap K$. Here we have $W_s=W$ and $\phi_s(x)=\phi(s^{-1}xs)$ for all 
$x\in H_s$. Just to distinguish the action of $sHs^{-1}$ from that of $H$ we denote the 
representation space by $W_s$ here (though it's the same space $W$). 
{\color{red}
\begin{theorem}\label{ind-res}
The representation $Res_K^G Ind_H^G W$ is decomposed as follows:
$$Res_K^G Ind_H^G W \cong \bigoplus_{s\in K\backslash G/H} Ind_{H_s}^K W_s$$
where $s$ varies over a set of representatives of the double cosets. 
\end{theorem}}
\begin{proof}
Denote $V= Ind_H^GW =\displaystyle\bigoplus_{t \in G/H} tW$ and recall $G$ acts on $\{tW \mid 
t\in G/H\}$ with stabiliser $H$. 

{\bf Step 1:} For $s\in S$, let us denote by $V(s)=\displaystyle \sum_{x\in KsH} xW = \sum_{k\in K} 
ksW $. We claim that 
\begin{enumerate}
\item each $V(s)$ is a $K$-representation (this is clear by second formula) and, 
\item $V=\displaystyle\bigoplus_{s\in S} V(s)$ (because $G$ is a disjoint union of the $(K,H)$ 
double cosets). 
\end{enumerate}

{\bf Step 2:} We claim that $V(s) \cong Ind_{H_s}^K W_s$. First, we note that 
$V(s)=\displaystyle \bigoplus_{k\in K/H_s} k(sW)$. The sum part follows from the definition and for 
the direct sum, we note that $\{x\in K \mid x(sW)=sW\} = H_s$. More precisely, if $k_1sW=k_2sW$ then 
$s^{-1}k_2^{-1}k_1s W=W$, that is, $s^{-1}k_2^{-1}k_1s \in H$, thus, $k_1\equiv k_2 \imod {H_s}$. 
This proves that $V(s)=\displaystyle \bigoplus_{k\in K/H_s} k(sW) = Ind_{H_s}^K(sW)$. 

{\bf Step 3:} Now for the final step we show that $W_s\cong sW$ as $H_s$-modules. This follows from 
writing out the definition in both cases explicitly. For $g\in H_s$ write $g=shs^{-1}$ for some 
$h\in H$, then $g.W_s = s^{-1}gs(W) = h(W)$ and $g.sW = shs^{-1}sW=sh(W)$. 

This completes the proof.
\end{proof}

{\bf Summary (steps to write the decomposition):}

{\bf Given: } 
\begin{itemize}
          \item $H$ and $K$ subgroups of $G$, and 
          \item $(\phi, W)$ a representation of $H$.
\end{itemize}

{\bf Want: } to understand $Res_K^G(Ind_H^G W)$ in terms of given data?

{\bf Steps: } 
\begin{itemize}
\item Choose $S$ a set of representatives of the double cosets $K\backslash G/H$.
\item Construct the subgroups $H_s$ of $K$ and its representation $\phi_s$ on $W$ (denoted as 
$W_s$ to indicate a different action).
\item Now consider the induced representations $Ind_{H_s}^K W_s$ which gives the final answer.
\end{itemize}

\section{Mackey-Wigner Irreducibility Criteria}

Now we can get back to our main problem where we can apply the Theorem~\ref{ind-res} by taking $K=H$. Let 
us recall the main problem. Let $G$ be a group and $H$ be a subgroup with a representation 
$(\phi, W)$ with associated character $\theta$. We would like to know when $Ind_H^GW$ 
is an irreducible representation of $G$, equivalently when the corresponding character 
$Ind_H^G\theta$ is an irreducible character.  
{\color{teal}
\begin{definition}
Two representations $(\rho_1, V_1)$ and $(\rho_2, V_2)$ of $G$ are said to be disjoint if $V_1$ and 
$V_2$ do not have a common irreducible representation when they are written as a sum of irreducible 
ones. 
\end{definition}}
\begin{proposition}
The representations $(\rho_1, V_1)$ and $(\rho_2, V_2)$ of $G$ are disjoint if their corresponding characters are orthogonal.
\end{proposition}
\begin{proof}
The characters of $\rho_1$ and $\rho_2$ will not have any common irreducible character when they are written as a sum of irreducible ones. Now the result follows using the orthogonality relations.
\end{proof}
{\color{red}
\begin{theorem}[Mackey]\label{mackey-irr}
Let $H$ be a subgroup of $G$ and $(\phi, W)$ be a representation of $H$. For $s\in G$, define the 
subgroups $H_s = sHs^{-1}\cap H \leq H$, and a representation $(\phi_s, W_s)$ of $H_s$ where $W_s=W$ 
with the action $\phi_s \colon H_s\rightarrow \GL(W)$ given by $\phi_s(x)=\phi(s^{-1}xs)$. Then, the 
representation $Ind_H^G W$ is irreducible if and only if
\begin{enumerate}
\item $W$ is an irreducible representation of $H$, and
\item for all $s\in G - H$ the representations $(\phi_s, W_s)$ and $Res_{H_s}^H(\phi, W)$ of 
$H_s$ are disjoint.
\end{enumerate}
\end{theorem}}
\begin{proof}
To prove this consider the corresponding characters. Let $\theta$ be the character of $\phi$. Then, $Ind_H^G W$ is irreducible if and only if $\langle Ind_H^G\theta, Ind_H^G\theta \rangle_G = \langle \theta, Res_H^G Ind_H^G\theta \rangle_H = 1$. Now using Theorem~\ref{ind-res} $Res_H^G Ind_H^G\theta = \displaystyle\sum_{s\in H\backslash G/H} Ind_{H_s}^H \theta_s$ where $\theta_s$ is the character of $\phi_s$. Thus, $Ind_H^G W$ is irreducible if and only if 
\begin{eqnarray*} 1 &=&  \langle Ind_H^G\theta, Ind_H^G\theta \rangle_G = \langle \theta, \sum_{s\in H\backslash G/H} Ind_{H_s}^H \theta_s \rangle_H \\
&=& \displaystyle \sum_{s\in H\backslash G/H} \langle \theta, Ind_{H_s}^H \theta_s \rangle_{H} = \displaystyle \sum_{s\in H\backslash G/H} \langle Res_{H_s}^H\theta, \theta_s \rangle_{H_s} 
\end{eqnarray*} 
(we use reciprocity for $H$ and $H_s$ here). Now each term in the last sum is a non-negative integer and for $s=e$ we have the corresponding term $\langle Res_{H_s}^H\theta, \theta_s \rangle_{H_s} = \langle \theta, \theta \rangle_{H} \geq 1$. Thus, $Ind_H^G W$ is irreducible if and only if  $\langle Ind_H^G\theta, Ind_H^G\theta \rangle_G = 1$ if and only if $\langle \theta, \theta \rangle_{H} = 1$ and $\langle Res_{H_s}^H\theta, \theta_s \rangle_{H_s} =0$. These are precisely the required conditions.
\end{proof}
A particular case of the above situation arises when we have a normal subgroup of $G$. In this case, 
we can rephrase the above condition and eventually get what is called {\bf Wigner's little group 
method}. This is applied to get irreducible representations of a semidirect product.

\begin{corollary}
Suppose $N$ is a normal subgroup of $G$  and $(\phi, W)$ is a representation of $N$. For any $s\not\in N$ define representations of $N$ by $\phi_s(x)=\phi(s^{-1}xs)$. Then, $Ind_N^G W$ is irreducible if and only if 
\begin{enumerate}
 \item $\phi$ is an irreducible representation of $N$, and, 
 \item $\phi$ and $\phi_s$ are not isomorphic for any $s\in G- N$.   
\end{enumerate}
\end{corollary}
\noindent Now we summarise the above so that we can apply it in a practical situation.

\subsection{{\bf Summary (steps to check irreducibility): }}\label{check-irr}

{\bf What's Given: } 
\begin{itemize}
          \item $H$ a subgroup of $G$, and 
          \item $(\phi, W)$ a representation of $H$.
\end{itemize}

{\bf Want: } to understand when $Ind_H^G W$ is an irreducible representation of $G$.

{\bf Steps: } 
\begin{itemize}
\item Check if $\phi$ is an irreducible representation of $H$. If yes proceed further.
\item Choose $S$ a set of representatives of the double cosets $H\backslash G/H$.
\item For all $s\in S$, construct the subgroups $H_s$ of $H$ and its representation $\phi_s$ on $W$ 
(denoted as $W_s$ to indicate a different action). Note if $H$ is normal all $H_s=H$ but $\phi_s$ 
could still be different. Also, for $s=e$ the subgroup $H_s=H$ and $\phi_s=\phi$.
\item Check if $(\phi_s, W_s)$ and $Res_{H_s}^H(\phi, W)$ of $H_s$ are disjoint for all $s\neq e$. 
If yes, the induced representation is irreducible.
\end{itemize}

\subsection{{\bf Summary (steps to check irreducibility when we have a normal subgroup): }}

{\bf Given: } 
\begin{itemize}
          \item $N$ a normal subgroup of $G$, and 
          \item $(\phi, W)$ a representation of $N$.
\end{itemize}

{\bf Want: } to understand when $Ind_N^G W$ is an irreducible representation of $G$.

{\bf Steps: } 
\begin{itemize}
\item Check if $\phi$ is an irreducible representation of $N$. If yes, proceed further.
\item Choose $S$ a set of representatives of the double cosets $N\backslash G/N$.
\item For all $s\in S$, consider the conjugate representations $\phi_s$ on $W$ of $N$ (denoted as 
$W_s$ to indicate a different action). That is, $\phi_s\colon N \rightarrow \GL(W)$ is given by 
$\phi_s(x)=\phi(s^{-1}xs)$. 
\item Check if the representations $\phi_s$ and $\phi$ of $N$ are not isomorphic for all $s\neq e$. 
This can be done by computing their character. If yes, the induced representation $Ind_N^G W$ is 
irreducible.
\end{itemize}

\begin{exercise}\label{dihedral-ind}
Consider the Dihedral group $D_n=\langle a, b \mid a^n=1=b^2, ab=ba^{-1}\rangle$ with $n>2$ and the 
subgroup $N=\langle a \rangle$. 
\begin{enumerate}
\item $N$ is a normal subgroup of $G$. 
\item There are $n$ $1$-dimensional representations of $N$ as $N\cong \mathbb Z/n\mathbb Z$.
\item We can choose $N\backslash D_n/N =\{e, b\}$. 
\item Given $\phi$ a $1$-dimensional non-trivial representation (that is $\phi(a)\neq 1$) of $N$ we 
define the representation $\phi_b$ of $N$ by $\phi_b(x)=\phi(b^{-1}xb)$. Check 
$\phi_b(a)=\phi(bab)=\phi(a^{-1})$. Hence, $\phi$ and $\phi_b$ are not isomorphic if and only if 
$\phi\neq \bar\phi$ (which is if and only if $\phi(a)\neq \pm 1$).
\item Similarly, check the same for any $\phi_{ba^i}$ and note that $\phi$ and $\phi_{ba^i}$ 
are not isomorphic. Although this is not really needed at the end. 
\item  Show, using Mackey's criteria, that the induced representation $Ind_N^G\phi$ is an 
irreducible representation where $\phi(a)\neq \pm 1$ is a primitive $n$-th root of unity.
\end{enumerate}
\end{exercise}

\begin{exercise}
Consider the group $\SL_2(q)$ and the Borel subgroup 
$$\B=\left\{\begin{pmatrix}a&b\\&a^{-1} \end{pmatrix}\mid a\in \mathbb F_q^*, b\in \mathbb 
F_q\right\}.$$
\begin{enumerate}
 \item We can choose $\B\backslash \SL_2(q)/\B =\left\{I, s=\begin{pmatrix}&-1\\1& 
\end{pmatrix}\right\}$.
\item Fix a group homomorphism $\omega\colon \mathbb F_q^* \rightarrow \mathbb C^*$. Consider the 
$1$-dimensional representation $\phi \colon \B \rightarrow \mathbb C$ given by 
$$\phi\begin{pmatrix}a&b\\&a^{-1} \end{pmatrix} = \omega(a).$$ 
Verify that it's a representation.
\item Compute $\B_s = \B\cap s\B s^{-1} = T= \left\{\begin{pmatrix}a&\\&a^{-1} \end{pmatrix}\mid a\in 
\mathbb F_q^*\right\}$.  
\item Then, $Res_{\B_s}^{\B}\phi$ is given by $Res_{\B_s}^{\B}\phi \begin{pmatrix}a&\\&a^{-1} 
\end{pmatrix} = \omega(a)$.
\item Now, $\phi_s \colon \B_s \rightarrow \mathbb C$ is given by $\phi_s(x) = \phi(s^{-1}xs)$. 
Check, $\phi_s \begin{pmatrix}a&\\&a^{-1} \end{pmatrix} = \phi \begin{pmatrix}a^{-1}&\\&a 
\end{pmatrix} = \omega(a^{-1})$. 
\item Thus, $\phi$ is irreducible if and only if $Res_{\B_s}^{\B}\phi\neq \phi_s$ if and only if 
$\omega^2\neq 1$. 
 \end{enumerate}
\end{exercise}

\chapter{Representations of a Semidirect Product - Wigner's Little Group Method}

Let $G$ be a group with $A$ a normal Abelian subgroup. Suppose $G=A\rtimes H$ for some subgroup $H$. Hence we have $A\cap H=1$ and $G=AH$. Now, all irreducible representations of $A$ are $1$-dimensional. Let $\hat A$ be the set of all characters (equivalently $1$-dim irreducible representations) of $A$. Define an action of $G$ on $\hat A$ as follows:
Let $\chi\in \hat A$. Now, for a given $g \in G$ we define $\chi^g \colon A \rightarrow \mathbb C^*$ by $\chi^g(x) = \chi(g^{-1}xg)$ which is again a character. 
\begin{exercise}
Verify that $G\times \hat A \rightarrow \hat A$ given by $(g, \chi) \mapsto \chi^g$ is an action.
\end{exercise}
This action can be restricted to $H$ and thus $H$ acts on the set of characters of $A$ as follows: $H\times \hat A \rightarrow \hat A$ given by $h.\chi\mapsto \chi^h$. Let us denote the stabiliser of $\chi$ by $H_{\chi}= \{h\in H \mid \chi^h=\chi\}$. This allows us to construct a subgroup of $G$, 
$$G_{\chi}= A.H_{\chi} \cong A\rtimes H_{\chi}$$
called the {\bf little groups}. Now, the character $\chi\in \hat A$ can be trivially extended to $G_{\chi}$ as follows: $\chi\colon G_{\chi} \rightarrow \mathbb C^*$ given by $\chi(a.h)=\chi(a)$ for all $a\in A$ and $h\in H_{\chi}$. Thus, $\chi$ is a character of $G_{\chi}$ of dimension $1$.
\begin{exercise}
Understand all of the above for the dihedral group $\mathbb Z/n\mathbb Z \rtimes \mathbb Z/2\mathbb Z$. 
\end{exercise}

Now, let $\rho$ be an irreducible representation of $H_{\chi}$. Again, using the projection map $G_{\chi} \rightarrow H_{\chi}$ given by $a.h\mapsto h$ we can get an irreducible representation $\rho$ of $G_{\chi}$. Let us denote the representation $\theta_{\chi,\rho}:= Ind_{G_{\chi}}^G \chi\otimes \rho$ of $G$.
{\color{red}
\begin{theorem} Let $G\cong A\rtimes H$ where $A$ is a normal Abelian subgroup of $G$. For $\chi\in \hat A$ and $\rho$ an irreducible representation of $H$, extend them to the ``little subgroups" $G_{\chi}$. Now, define $\theta_{\chi,\rho} =Ind_{G_{\chi}}^G \chi\otimes \rho$ representations of $G$. Then, 
\begin{enumerate}
\item $\theta_{\chi,\rho}$ is irreducible.
\item If $\theta_{\chi,\rho}$ and $\theta_{\chi',\rho'}$ are isomorphic then $\chi, \chi'$ are in the same orbit of $\hat A/H$ and $\rho\cong \rho'$.
\item Every irreducible representation of $G$ is isomorphic to one of the $\theta_{\chi,\rho}$.
\end{enumerate}
\end{theorem}}
We leave proof of this to the reader.
\begin{exercise}
Consider the affine group $A(q)=\left\{\begin{pmatrix} a & b \\ 0&1 \end{pmatrix}\right\}\subset  
\GL_2(q)$. Write down all irreducible representations of $A(q)$. 
\end{exercise}
\begin{exercise}
Consider the Borel subgroup $\B=\left\{\begin{pmatrix} a & b \\ 0&d \end{pmatrix}\right\}\subset  
\GL_2(q)$. Write down all irreducible representations of $\B$. 
\end{exercise}
\begin{exercise}
Consider the group $ZU=\left\{\begin{pmatrix} a & b \\ 0&a \end{pmatrix}\right\}\subset  \GL_2(q)$. 
Write down all irreducible representations of $ZU$. 
\end{exercise}
\section{Representations of groups of order $pq$ and $p^3$}
In this section, $p, q$ are primes and $p\neq q$. We know any group of order $p$ is cyclic and any group of order $p^2$ is Abelian. Thus, all irreducible representations of such groups are $1$-dimensional. The representation theory of such groups can be done easily following the theory 
in Chapter~\ref{rep-abelian}. Thus, we will deal with non-Abelian groups here.

\subsection{Groups of order $pq$}

Let $G$ be a group of order $pq$. We may assume $q< p$. When $q\nmid (p-1)$ it can be shown that the group is Abelian, in fact, cyclic (See Example on page 135, 143 and 181~\cite{df}). 

Thus, we deal with the case when $q\mid (p-1)$ and $G$ are non-Abelian. Denote by $r=\frac{p-1}{q}$. In this case, there is a unique non-Abelian group, up to isomorphism, of order $pq$. Such groups can be also realised as matrix subgroup in the following way.

\begin{example}
Consider the finite field $\mathbb F_p$ and the group 
$$\mathcal G=\left \{\begin{pmatrix} 1 & x \\ & y \end{pmatrix} \mid x\in \mathbb F_p, y\in \mathbb F_p^*\right\}.$$ 
\begin{enumerate}
\item Show that $\mathcal G$ is a group of order $p(p-1)$. 
\item Fix $\zeta$ to be a generator of the cyclic group $\mathbb F_p^*$. Take $X =\begin{pmatrix} 1 &  \\ & \zeta \end{pmatrix}$ and $Y =\begin{pmatrix} 1 & 1 \\ & 1 \end{pmatrix}$. What are the orders of $X$ and $Y$? Show that $X$ and $Y$ generate $\mathcal G$?
\item Take $A=X^{r}$ and $B=Y$. Then show that $G=<A,B>$ is a nonabelian subgroup of $\mathcal G$ of order $pq$. 
\item Show that the number of conjugacy classes of $G$ is $q+r$. 
\item Show that the commutator subgroup $[G,G]=<B>$ is of order $p$. Hence, the number of $1$-dimensional representations of $G$ is $\frac{pq}{p}=q$. 
\item Use the $\det\colon G \rightarrow \mathbb F_p^*$ and show that image size is $q$. Thus, we can lift $1$-dimensional representations to get $q$ $1$-dimensional representations of $G$. 
\end{enumerate}
\end{example}

\begin{example}
Consider the group of order $21$, $G=\mathbb Z/3\mathbb Z \ltimes \mathbb Z/7\mathbb Z$.  
\end{example}

We note that $G=\mathbb Z/q\mathbb Z \ltimes \mathbb Z/p\mathbb Z$ which can be explicitly written as follows: 
$$G = \langle a,b \mid a^q=1=b^p, ab=ba^l \rangle$$
where $l$ is an element of order $q$ in $\mathbb F_p^*$. 
In this case, we will see that the remaining $r$ irreducible representations are each of dimension $q$. These are obtained as induced representations of $1$-dimensional representations of the subgroup $<B>$. 
\begin{exercise}
Complete the last step.
\end{exercise}
{\color{red}
\begin{theorem} Let $G$ be a non-Abelian group of order $pq$ where $p,q$ are prime and $q\mid (p-1)$. Then, $G$ has $q$ irreducible representations of dimension $1$ and $\frac{p-1}{q}$ irreducible representations of dimension $q$. 
\end{theorem}
}

\subsection{Groups of order $p^3$}
We deal with the non-Abelian groups of order $p^3$. When $p=2$ its either $D_4$ or $Q_8$ and we have already seen the character table of these groups. So, we may assume $p$ is odd. Once again, up to isomorphism, there are $2$ such groups. We recommend well-written online notes by Conrad~\cite{Co} or~\cite{df} for a concrete description of these groups.

We only require the following structural description:
\begin{proposition}
Let $G$ be a non-Abelian group of order $p^3$ with center $Z$. Then, $|Z|=p$, $[G,G]=Z$, and $G/Z\cong \mathbb Z/p\mathbb Z \times \mathbb Z/p\mathbb Z$.
\end{proposition}

Thus, such groups have $p^2$ representations of dimension $1$ obtained as a lift of $G/Z$. The remaining ones are $p-1$ of them and each of degree $p$.   
\begin{exercise}
Complete the last step. These representations are obtained as induced representations of an Abelian normal subgroup $N=<a,Z>$ of order $p^2$ where $a\in G-Z$.
\end{exercise}
Note that the character table of both non-isomorphic groups would be the same as we have noted similar phenomena for groups of order $8$.

\begin{exercise}
Collect the examples of groups where only two distinct character degrees are possible. In fact, the classification of such groups is known which you may explore further. 
\end{exercise}
\chapter{Characters of the group $\GL_2(q)$}
All representations considered here are over $\mathbb C$.
We consider the finite field $\mathbb F_q$, with $q \geq 5$ odd (this is simply for the convenience of some calculations involved). The group 
$$\GL_2(q)=\left\{X\in M_2(\mathbb F_q)\mid \det(X)\neq 0 \right\}$$ and 
$$\SL_2(q)=\left\{X\in M_2(\mathbb F_q)\mid \det(X)=1 \right\}.$$

\begin{exercise} Let us warm up with some basic calculations.
\begin{enumerate}
\item Show that $|\GL_2(q)| = (q^2-1)(q^2-q)$ and $|\SL_2(q)| = (q^2-1)(q-1)$.
\item Show that $\det \colon \GL_2(q) \rightarrow \mathbb F_q^*$ is a surjective group homomorphism 
with kernel the group $\SL_2(q)$.
\item Show that the center $\mathcal Z(\GL_2(q))=\{\lambda.I\mid \lambda \in \mathbb F_q^*\}$, the 
scalar matrices.
\end{enumerate}
\end{exercise}

 Now define $x_{12}(t):= \begin{pmatrix} 1&t\\ 0 &1\end{pmatrix}$ and $x_{21}(s):= \begin{pmatrix} 1& 0 \\s &1\end{pmatrix}$ for $t, s\in k$. The matrices $x_{12}(t)$ and $x_{21}(s)$ are called the {\bf elementary matrices}.
\begin{exercise} Let us show that the elementary matrices generate the group $\SL_2(q)$.
\begin{enumerate}
\item Compute $w(t)=x_{12}(t)x_{21}(-t^{-1})x_{12}(t)$. What is $w(-1)$? Compute 
$h(t)=w(t)w(-1)$.
\item For $c\neq 0$, verify the following for an element in $\SL_2(q)$, 
$$\begin{pmatrix}a&b\\c&d\end{pmatrix}= \begin{pmatrix}1&(a-1)c^{-1}\\&1\end{pmatrix} 
\begin{pmatrix}1&\\c&1\end{pmatrix} \begin{pmatrix}1&(d-1)c^{-1}\\ &1\end{pmatrix}.$$
\item Show that the group $\SL_2(q)$ is generated by the set of all elementary matrices 
$\{x_{12}(t), x_{21}(s) \mid t,s\in \mathbb F_q\}$.
\end{enumerate}
\end{exercise}

\begin{exercise} Let us compute the $1$-dimensional representations.
\begin{enumerate}
\item Show that $[\SL_2(q), \SL_2(q)]=1$ and hence the only $1$ dimensional representation of 
$\SL_2(q)$ is the trivial one.
\item Show that $[\GL_2(q), \GL_2(q)]=\SL_2(q)$ and hence $\GL_2(q)$ has $q-1$ representations of 
dimension $1$. 
\item Since $\mathbb F_q^* \cong \mathbb Z/(q-1)\mathbb Z$ (why?) we have $q-1$ representations of 
dimension $1$. This would give $q-1$ representations of $\GL_2(q)$ of dimension $1$.
\end{enumerate}
\end{exercise}

\section{Conjugacy classes in $\GL_2(q)$}
 Let us try to determine the number of conjugacy classes in this group.
The conjugacy classes in $\GL_2(q)$ are as follows. This can be determined by rational canonical 
form theory. Let $\mathfrak M(x)$ and $\mathfrak C(x)$ be the minimal and characteristic polynomial 
respectively. Clearly, $\mathfrak C(x)$ could be any degree $2$ polynomial with non-zero constant 
term and  $\mathfrak M(x)$ could be any polynomial of degree $\leq 2$. 

{\bf Central type:} When $\mathfrak C(x)=(x-\lambda)^2$ and $\mathfrak M(x)=(x-\lambda)$, this gives the conjugacy classes of central matrices $\{\lambda I\}$ for $\lambda\in \mathbb F_q^*$. There are total $q-1$ such conjugacy classes each of size $1$.

{\bf Split semisimple type:} When $\mathfrak C(x)=(x-a)(x-b) = \mathfrak M(x)$. This corresponds to the conjugacy classes of diagonal matrices $\begin{pmatrix} a & \\ &b\end{pmatrix}$ where $a, b\in \mathbb F_q^*$ and $a\neq b$ (so as they are excluded from the central type). There are total $\frac{(q-1)(q-2)}{2}$ such classes each of size $q(q+1)$. Note that the centralizer of such elements is all of the diagonal matrices hence of size $(q-1)^2$.

{\bf Anisotropic type:} When $\mathfrak C(x)=x^2+ ax+ b= \mathfrak M(x)$ is irreducible. These elements are those of which minimal polynomial and characteristic polynomials are equal and irreducible of degree $2$ and a representative is the companion matrix $\begin{pmatrix} 0 & -b \\ 1&-a\end{pmatrix}$. These correspond to picking an element in $\mathbb F_{q^2}$ which is not in $\mathbb F_q$ and two of these will give the same class. Thus, there are $\frac{q^2-q}{2}$ such classes. The centralizer of such an element is polynomials in the same matrix and hence $\cong \mathbb F_{q^2}^*$ of size $q^2-1$. Thus, the conjugacy class size is $q^2-q$.

{\bf Unipotent type:} When $\mathfrak C(x)=(x-\lambda)^2$ and $\mathfrak M(x)=(x-\lambda)^2$.  This corresponds to the conjugacy classes of Jordan matrices $\begin{pmatrix} \lambda & 1\\ &\lambda \end{pmatrix}$ where $\lambda \in \mathbb F_q^*$ and the centraliser is of size $q(q-1)$ and hence the conjugacy class size is $q^2-1$. There are $q-1$ such types.

Thus we get,
{\color{red}
\begin{theorem}
The total number of conjugacy classes in $\GL_2(q)$ is 
$$(q-1) + \frac{(q-1)(q-2)}{2} + \frac{q^2-q}{2} + q-1 = (q-1) (2+ \frac{q-2}{2}+\frac{q}{2})=(q-1)(q+1).$$
\end{theorem}}

So, far we have found $q-1$ characters each of dimension $1$. We need to find more irreducible 
characters. Let us set some notation here. We have $\det \colon \GL_2(q) \rightarrow \mathbb F_q^*$ 
given by $A\mapsto \det(A)$. Now, if $\psi\colon \mathbb F_q^* \rightarrow \mathbb C^*$ is a 
representation then lifting this to $\GL_2(q)$ we get, $\psi\circ\det$. 
\begin{proposition}
The group $\GL_2(q)$ has $q-1$ representations of dimension $1$. These are given by $\psi\circ\det$ 
where $\psi\colon \mathbb F_q^* \rightarrow \mathbb C^*$ is a group homomorphism.
\end{proposition}

\section{Parabolic Induction - preparation}
Now we consider certain representations of a parabolic subgroup and induce these representations in the whole group. In our case, the parabolic subgroup we consider is simply the Borel subgroup of upper triangular matrices. In this process, we get several irreducible representations.

Let us consider the subgroup $\B=\left\{\begin{pmatrix} a & b \\ 0&d \end{pmatrix}\right\}\subset 
\GL_2(q)$, namely the Borel subgroup. Fix group homomorphisms $\psi, \phi \colon \mathbb F_q^* 
\rightarrow \mathbb C^*$. These are $1$-dimensional representations of the cyclic group $\mathbb F_q^*$. This can be done by fixing a primitive $(q-1)$th root of unity, say 
$\zeta$. Then all characters of $\mathbb F_q^*$ (being a cyclic group of order $q-1$) will be 
given by mapping a generator of $\mathbb F_q^*$ to $\zeta, \zeta^2, \ldots, \zeta^{q-1}=1$. Now, 
using $\psi, \phi$ we define a representation 
$\rho_{\psi, \phi} \colon \B \rightarrow \mathbb C^*$ given by 
$$\rho_{\psi,\phi} \begin{pmatrix} a & b \\ 0&d \end{pmatrix} = \psi(a)\phi(d) .$$

\begin{exercise}
Show that $\rho_{\psi,\phi}$ is a representations of $\B$. That is, show that these are group 
homomorphisms. 
\end{exercise}
\begin{exercise}
Further, show that these are all possible $1$-dimensional representations of $\B$. 
\end{exercise}
\noindent Now we deal with the two cases $\psi\neq \phi$ and $\psi = \phi$ separately. We induce the 
representation $\rho_{\psi,\phi}$ of $\B$ to $\GL_2(q)$ and analyse them further.
In the first case when $\psi\neq \phi$ the induced representations turn out to be irreducible and are called the principal series representations. In the second case when $\psi=\phi$, the induced representations are not irreducible and these turn out to be a direct sum of a $1$-dimensional representation what is called the Steinberg representations.

\section{The principal series representations}

We continue the notation from the previous section. Now define the induced representations $\hat\rho_{\psi,\phi} = 
Ind_{\B}^{\GL_2(q)}\rho_{\psi,\phi}$. We want to determine when these representations are 
irreducible.
\begin{exercise}
The dimension of the representation $\hat\rho_{\psi,\phi}$ is $q+1$. This follows by a straightforward computation (see Proposition~\ref{ind-dimension}). 
\end{exercise}

\noindent Now we use the algorithm~\ref{check-irr} to check the irreducibility of these induced 
representations. Before that, we need to determine $\B \backslash \GL_2(q)/ \B$.
\begin{exercise}
Show that we can choose $\left\{I, s=\begin{pmatrix}&1\\1& \end{pmatrix}\right\}$ as a set of 
representatives for $\B\backslash \GL_2(q)/ \B$. This is equivalent to applying row-column 
operations on a matrix where we are not allowed row-column flipping. Thus, in the end, we will be 
left with one of these kinds of matrices.
\end{exercise}

\noindent Now compute $\B_s$ as follows:
\begin{exercise}
 $\B_s = \B \cap s\B s^{-1} = \mathbb T= \left\{\begin{pmatrix}a&\\&d \end{pmatrix}\mid a,d \in 
\mathbb F_q^*\right\}$. This subgroup $\mathbb T$ is called a {\bf split maximal torus}. 
\end{exercise}
\begin{exercise}
Since $\mathbb T=\mathbb F_q^* \times \mathbb F_q^*$ the irreducible representations are obtained by taking a product of the individual components which themselves are cyclic. We see that these representations are given by $\{\psi.\phi\}$ where $\psi,\phi$ are characters of $\mathbb F_q^*$ and the map $(\psi.\phi)(a,b)=\psi(a)\phi(b)$.
\end{exercise}
\begin{exercise}
Then, $Res_{\B_s}^{\B} \rho_{\psi,\phi}$ is a map on $\B_s=  \mathbb T$ given by 
$$(Res_{\B_s}^{\B}\rho_{\psi,\phi}) \begin{pmatrix}a&\\& d \end{pmatrix} = \psi(a)\phi(d).$$ 
\end{exercise}
\noindent Let us compute the conjugate maps and note the subtle difference.

\begin{exercise}
Now, $(\rho_{\psi,\phi})_s \colon \B_s \rightarrow \mathbb C$ is a map on $\B_s=\mathbb T$ given by 
$(\rho_{\psi,\phi})_s(x) = \rho_{\psi,\phi}(s^{-1}xs)$. 
Check, $$(\rho_{\psi,\phi})_s \begin{pmatrix}a&\\& d \end{pmatrix} = \rho_{\psi,\phi} 
\begin{pmatrix} d&\\&a  \end{pmatrix} = \psi(d)\phi(a) = Res_{\B_s}^{\B}\rho_{\phi,\psi}.$$  
\end{exercise}

\noindent Now we are ready to prove the following:
\begin{proposition}\label{principal-series}
When $\psi\neq \phi$ the representation $\hat\rho_{\psi,\phi} = 
Ind_{\B}^{\GL_2(q)}\rho_{\psi,\phi}$ 
is irreducible of dimension $q+1$.  
\end{proposition}
\begin{proof}
From Mackey's criteria Theorem~\ref{mackey-irr}, the representation $\hat\rho_{\psi,\phi} = 
Ind_{\B}^{\GL_2(q)}\rho_{\psi,\phi}$ is irreducible if and only if 
\begin{enumerate}
 \item $\rho_{\psi,\phi}$ is irreducible (which it is by being $1$-dimensional) representation of 
$\B$. 
 \item $Res_{\B_s}^{\B} \rho_{\psi,\phi}$ and $(\rho_{\psi,\phi})_s$ are disjoint (for this we need 
to check if these are equal or not as they are $1$-dimensional) representations of $\B_s=\mathbb T$.  
\end{enumerate}
Hence, we get $\hat\rho_{\psi,\phi} = Ind_{\B}^{\GL_2(q)}\rho_{\psi,\phi}$ is irreducible if and 
only 
if $Res_{\B_s}^{\B} \rho_{\psi,\phi} \neq (\rho_{\psi,\phi})_s$ as maps. Since we are given $\phi\neq 
\psi$ we get the result. 
\end{proof}

However, we need to also understand if we get distinct irreducible representations. These are certainly different from the earlier ones obtained so far which were $1$-dimensional as these are $(q+1)$-dimensional. 
\begin{proposition}
When $\psi_1\neq \phi_1$ and $\psi_2\neq  \phi_2$, the representations $\hat\rho_{\psi_1,\phi_1} 
\cong \hat\rho_{\psi_2,\phi_2}$ if and only if $(\psi_1, \phi_1) = (\psi_2, \phi_2)$ or $(\psi_1, 
\phi_1) = (\phi_2, \psi_2)$. 
Thus, there are total $\frac{(q-1)(q-2)}{2}$ distinct such irreducible representations. 
\end{proposition}
\begin{proof}
In view of Proposition~\ref{ind-universal} we need to determine, 
\begin{eqnarray*}
Hom_{\GL_2}(\hat\rho_{\psi_1,\phi_1}, \hat\rho_{\psi_2,\phi_2}) &=& 
Hom_{\GL_2}(Ind_{\B}^{\GL_2(q)}\rho_{\psi_1,\phi_1}, Ind_{\B}^{\GL_2(q)}\rho_{\psi_2,\phi_2}) \\
&=& Hom_\B(\rho_{\psi_1,\phi_1}, Res_\B^{\GL_2(q)}Ind_{\B}^{\GL_2(q)}\rho_{\psi_2,\phi_2}) {\rm\ 
Using\  reciprocity}\\
&=& Hom_\B(\rho_{\psi_1,\phi_1}, \rho_{\psi_2,\phi_2} \oplus Ind_{\B_s}^{\B}(\rho_{\psi_2,\phi_2})_s)\\
&=& Hom_\B(\rho_{\psi_1,\phi_1}, \rho_{\psi_2,\phi_2}) \oplus Hom_\B(\rho_{\psi_1,\phi_1}, 
Ind_{\B_s}^{\B}(\rho_{\psi_2,\phi_2})_s) \\ 
&=& Hom_\B(\rho_{\psi_1,\phi_1}, \rho_{\psi_2,\phi_2}) \oplus Hom_{\B_s}( 
Res_{\B_s}^{\B}\rho_{\psi_1,\phi_1}, (\rho_{\psi_2,\phi_2})_s) 
\end{eqnarray*}
where we use reciprocity again at the last step. 
The first term on the right here is either $1$ or $0$ dimension depending on $\rho_{\psi_1,\phi_1} 
$ is isomorphic to $\rho_{\psi_2,\phi_2}$ or not as both are irreducible representations. Thus, 
$\dim Hom_\B(\rho_{\psi_1,\phi_1}, \rho_{\psi_2,\phi_2}) = 1$ if $\psi_1=\psi_2$ and $\phi_1=\phi_2$ 
and $0$ otherwise. 

The second term has representations of $\B_s=\mathbb T$ given by $\diag(a,d) \mapsto \psi_1(a)\phi_1(d)$ and 
$\diag(a,d) \mapsto \psi_2(d)\phi_2(a)$. These two representations are equal if and only if 
$\psi_1=\phi_2$ and $\psi_2=\phi_1$. Thus, $\dim Hom_{\B_s}(Res_{\B_s}^{\B}\rho_{\psi_1,\phi_1}, 
(\rho_{\psi_2,\phi_2})_s  = 1$ if $\psi_1=\phi_2$ and $\phi_1=\psi_2$,  and $0$ otherwise.
\end{proof}

Thus so far we get $(q-1)$ irreducible representations of dimension $1$ and $\frac{(q-1)(q-2)}{2}$ irreducible representations, namely the principal series, of dimension $q+1$.

\section{The Steinberg Representations}

We continue with the setup in the previous section and discuss what happens when $\psi=\phi$. We have,
$\rho_{\psi} :=\rho_{\psi,\psi}\colon \B \rightarrow \mathbb C^*$ given by 
$$\rho_{\psi} \begin{pmatrix} a & b \\ 0&d \end{pmatrix} = \psi(ad) .$$
We induce this representation and get $\hat\rho_{\psi} = Ind_{\B}^{\GL_2(q)}\rho_{\psi}$ which is a 
representations of dimension $(1+q)$. 
We want to determine if these representations give rise to irreducible ones. It turns out that these 
are not irreducible but can be easily broken to get some new irreducible representations, called the 
Steinberg representations. We can use Mackey's criteria as in the proof or 
Proposition~\ref{principal-series} and show that $\hat\rho_{\psi} = Ind_{\B}^{\GL_2(q)}\rho_{\psi}$ 
are not irreducible as $Res_{\B_s}^\B 
\rho_{\psi}$,  $(\rho_{\psi})_s$ both are same and given by 
$$\begin{pmatrix}a&\\& d \end{pmatrix} \mapsto \psi(ad).$$

\begin{exercise}
Show that, $(\rho_{\psi})_s \colon \B_s=\mathbb T \rightarrow \mathbb C$ is given by 
$(\rho_{\psi})_s(x) = \rho_{\psi}(s^{-1}xs)$. 
Check, $$(\rho_{\psi})_s \begin{pmatrix}a&\\& d \end{pmatrix} = \rho_{\psi} 
\begin{pmatrix} d&\\&a  \end{pmatrix} = \psi(ad) = Res_{\B_s}^{\B}\rho_{\psi}.$$  
\end{exercise}

\begin{exercise}
We may also note that if $\psi$ is a trivial character then so is $\rho_{\psi}$, i.e, $\rho_\psi(A)=1$ for all $A\in\B$. Does the induced representation $\hat\rho_{\psi}$ sound familiar in this case? 
\end{exercise}
\noindent Now, let us try to understand the decomposition of $\hat\rho_{\psi} = 
Ind_{\B}^{\GL_2(q)}\rho_{\psi}$. 
First, we determine how many irreducible representations are involved in $\hat\rho_{\psi}$ with the help of Proposition~\ref{ind-universal}.
\begin{proposition}
$\hat\rho_{\psi}$ is a sum of exactly two irreducible representations.
\end{proposition}
\begin{proof} 
Suppose, $\hat\rho_{\psi} = Ind_{\B}^{\GL_2(q)}\rho_{\psi}=m_1\rho_1+\cdots +m_h\rho_h$ written as a 
sum of irreducible representations. We need to determine $m_i$s. From the Maschke's Theorem if a 
module $V$ is written as a direct sum of irreducibles $V=\bigoplus W_i^{m_i}$ then $Hom_{\mathbb 
C[G]}(W,W) = \bigoplus M_{m_i}(End_{\mathbb C[G]}(W_i))= \bigoplus M_{m_i}(\mathbb C)$. Thus, $\dim 
Hom_{\mathbb C[G]}(W,W)=\sum m_i^2$ can help us determine $m_i$s.

In view of this, we need to compute the $\dim Hom_{\GL_2}(\hat\rho_{\psi}, \hat\rho_{\psi})$.
\begin{eqnarray*}
Hom_{\GL_2}(\hat\rho_{\psi}, \hat\rho_{\psi}) &=& 
Hom_{\GL_2}(Ind_{\B}^{\GL_2(q)}\rho_{\psi}, Ind_{\B}^{\GL_2(q)}\rho_{\psi}) \\
&=& Hom_\B(\rho_{\psi}, Res_\B^{\GL_2(q)}Ind_{\B}^{\GL_2(q)}\rho_{\psi}) {\rm\ Using\  
reciprocity}\\
&=& Hom_\B(\rho_{\psi}, \rho_{\psi} \oplus Ind_{\B_s}^{\B}(\rho_{\psi})_s)\\
&=& Hom_\B(\rho_{\psi}, \rho_{\psi}) \oplus Hom_\B(\rho_{\psi}, Ind_{\B_s}^{\B}(\rho_{\psi})_s) \\ 
&=& Hom_\B(\rho_{\psi}, \rho_{\psi}) \oplus Hom_{\B_s}( Res_{\B_s}^{\B}\rho_{\psi}, (\rho_{\psi})_s) {\rm\ Using\  reciprocity}\\
&=& Hom_\B(\rho_{\psi}, \rho_{\psi}) \oplus Hom_{\B_s}( Res_{\B_s}^{\B}\rho_{\psi}, Res_{\B_s}^{\B}\rho_{\psi}) {\rm\ Use\ the\ exercise\ above}\\
&=& \mathbb C \oplus \mathbb C.
\end{eqnarray*}
Note that at the last step we have all $1$-dimensional representations (hence are irreducible) of $\B$ and $\B_s=\mathbb T$ respectively.

Thus, $\dim Hom_{\GL_2}(\hat\rho_{\psi}, \hat\rho_{\psi}) =2 $ and using the equation $m_1^2+\cdots 
+ m_h^2=2$, combined with the fact it is not irreducible, we conclude that $1^2+1^2=2$ is the only 
solution. 
\end{proof}

Now, since $\hat\rho_{\psi} = Ind_{\B}^{\GL_2(q)}\rho_{\psi}$ is a sum of two irreducible 
representations we need to determine them. Fortunately, in the present case, we can easily show that 
it has $1$-dimensional representations contained in it and thus the complement would be the required 
irreducible one.

\begin{proposition}
The representation $\hat\rho_{\psi} = Ind_{\B}^{\GL_2(q)}\rho_{\psi}$ contains a copy of the 
$1$-dimensional representation $\psi\circ\det$.
\end{proposition}
\begin{proof}
For this, we need to analyse the representation space. Recall, we begin with the one dimensional 
$\B$-space $W$ and get $V=Ind_\B^{\GL_2(q)} W$.  If $W=<e>$ then the action of $\B$ is given by 
$\rho_{\psi}$. That is, $\begin{pmatrix} a & b \\ 0&d \end{pmatrix}.e = \psi(ad) e$, in another 
words $A.e=\psi(\det(A))e$ for an $A\in \B$. Now, the induced module
$V=\{f\colon \GL_2(q) \rightarrow W \mid f(xA^{-1})=A.f(x), A\in \B\}$ and $\GL_2(q)$-action is 
given by $(gf)(x)=f(g^{-1}x)$. We also remark (which will be immediately useful) that while writing 
an element $g\in \GL_2(q)$ as $g=xA^{-1}$ with $A^{-1}\in \B$ we can take $x$ to be of $\det(x)=1$.

We claim that the function $f_0 \colon \GL_2(q) \rightarrow W$ defined by 
$f_0(g)=\psi(\det(g^{-1}))e$ has the following property:
\begin{enumerate}
\item $f_0$ is in $V$.
\item $<f_0>\subset V$ is a $\GL_2(q)$-invariant subspace of $V$.
\item The action of $\GL_2(q)$ on $f_0$ is multiplication by $\psi\circ\det$.
\end{enumerate}
To check the first one, 
$$f_0(xA^{-1}) = \psi(\det(Ax^{-1}))e = \psi(\det(A)\det(x^{-1}))e  = \psi(\det(A))\psi(\det(x^{-1}))e = A.f_0(x).$$
Now for the remaining ones we note,
$$(g'f_0)(x)=f_0(g'^{-1}x)=\psi(\det(x^{-1}g'))e=\psi(\det(g'))\psi(\det(x^{-1}))e = \psi(\det(g'))f_0(x).$$
\end{proof}

Thus, $\hat\rho_{\psi} = Ind_{\B}^{\GL_2(q)}\rho_{\psi}  \cong (\psi\circ\det) \bigoplus St_{\psi}$ 
where $St_{\psi}$ is an irreducible representation of $\GL_2(q)$ of dimension $(q+1)-1=q$, called 
the {\bf Steinberg representations}.
\begin{proposition}
For every $\psi$, a character of $\mathbb F_q^*$, we have an irreducible representation $St_{\psi}$ 
of dimension $q$. There are $(q-1)$ such irreducible characters of $\GL_2(q)$.
\end{proposition}

\section{Cuspidal representations}

So far, we have found $(q-1)$ irreducible representations of dimension $1$, $\frac{(q-1)(q-2)}{2}$ irreducible representations of dimension $(q+1)$ and $(q-1)$ irreducible representations of dimension $q$. Thus, so far we have found
$$(q-1)+\frac{(q-1)(q-2)}{2}+(q-1)$$
irreducible characters of $\GL_2(q)$. Since the number of irreducible characters is the same as the 
number of conjugacy classes, it remains to find  $\frac{q^2-q}{2}$ irreducible characters. We can 
also try to determine their degrees:
$$\sum_{i=1}^{\frac{q^2-q}{2}} n_i^2 = |G|- \left((q-1).1^2 + \frac{(q-1)(q-2)}{2} (q+1)^2 + (q-1).q^2\right) = \frac{q(q-1)^3}{2}.$$
A possible solution of this equation with $n_i\mid |\GL_2(q)|$ is that each $n_i = q-1$.
This is indeed the case. The remaining representations are said to be {\bf cuspidal representations}. We will discuss that here briefly without going into too much in detail. Once again we look at certain subgroups of which representations we can induce.

Consider the subgroup $ZU=\left\{ \begin{pmatrix} a&b\\ &a \end{pmatrix}\mid a\in \mathbb F_q^*, b\in \mathbb F_q\right\}$.
\begin{exercise}
Compute the centralizer of $u = \begin{pmatrix} 1&1\\ &1 \end{pmatrix}$ and show that it is $ZU$. 
Note that this group is actually the product (not a direct product) of the centre $Z=\mathcal 
Z(\GL_2(q))$ and $U=\left\{ \begin{pmatrix} 1&b\\ &1 \end{pmatrix}\mid b\in \mathbb F_q\right\}$.
\end{exercise}

Now given a character $\psi$ of $\mathbb F_q^*$ (there are $q-1$ such) and $\phi$ of $\mathbb F_q$ (there are $q$ such) we define a representation $\mho_{\psi, \phi}\colon ZU\rightarrow \mathbb C$ as follows:
$$\mho_{\psi, \phi}  \begin{pmatrix} a&b\\ &a \end{pmatrix} = \psi(a)\phi(b).$$
\begin{exercise}
These are all possible $1$-dimensional representations of $ZU$.
\end{exercise}
Now, we define $\tilde\mho_{\psi, \phi} = Ind_{ZU}^{\GL_2(q)}\mho_{\psi, \phi} $ which are of 
dimension $q^2-1$. As it turns out that these are not irreducible. We work with certain $\psi$ and 
show that it can be decomposed as a direct sum of two irreducible representations
$$Ind_{ZU}^{\GL_2(q)}\mho_{\psi, \phi} = Ind_{\mathcal T}^{\GL_2(q)}\psi \oplus \pi_{\psi}$$
  
Let us begin with understanding the anisotropic maximal torus $\mathcal T$. The field $\mathbb 
F_{q^2}$ can be embedded in $M_2(q)$ giving rise to $\mathcal T \subset \GL_2(q)$.
\begin{exercise}
Consider $\mathbb F_{q^2}$ over $\mathbb F_{q}$ as $2$-dimensional vector space. For any, $\theta \in \mathbb F_{q^2}$ consider the left multiplication map on $\mathbb F_{q^2}$ giving rise to a linear map. This embeds $\mathbb F_{q^2}$ in $M_2(q)$.
\end{exercise}
\noindent Now $\mathbb F_{q^2}^*$ is cyclic group of order $q^2-1$ thus has as many $1$-dimensional representations. Consider the non-trivial Galois automorphism $\sigma\colon \mathbb F_{q^2} \rightarrow \mathbb F_{q^2}$ given by $x\mapsto x^q$. This fixed field is $\mathbb F_{q}$. Thus, the maps $\psi \colon \mathbb F_{q^2}^* \rightarrow \mathbb C^*$ also include the ones for $\mathbb F_{q}^*$. We say $\psi$ is a {\bf regular character} of $\mathcal T = \mathbb F_{q^2}^*$ if $\psi\sigma \neq \psi$. Equivalently, $\psi^q\neq \psi$ or $\psi^{q-1}\neq 1$. Thus, these are $\psi$ which are not of order $q-1$, hence total such number would be $(q^2-1)-(q-1) = q^2-q$. Also, not that such $\psi$ would come in pairs $\{\psi, \psi\circ\sigma\}$.

\begin{proposition}
For a regular character $\psi$ of $\mathcal T$, we have $Ind_{ZU}^{\GL_2(q)}\mho_{\psi, \phi} = 
Ind_{\mathcal T}^{\GL_2(q)}\psi \oplus \pi_{\psi}$ where $\pi_{psi}$ is an irreducible 
representation of dimension $q-1$. The $\pi_{\psi}\cong \pi_{\psi'}$ if and only if 
$\psi'=\psi\circ\sigma$.  Thus, these are the remaining $\frac{(q^2-q)}{2}$ irreducible cuspidal 
representations.
\end{proposition}

This section is not complete and a bit of effort is required to complete all the steps. On other hand knowing the irreducible representations, as we do by now, one can simply compute the characters to verify that we get all of them.

\section{Characters of $\GL_2(q)$}

Given the description above write the character table.


\vskip1cm
\begin{thebibliography}{99}
\normalsize
\bibitem[AB]{ab} Alperin, J. L.; Bell R. B., {\it``Groups and Representations''}, Graduate Texts in Mathematics, 
162, Springer-Verlag, New York, 1995.
\bibitem[Bu]{bu} Burnside, W., {\it``On Groups of Order $p^{\alpha}q^{\beta}$''}, Proc. London Math. 
Soc. (2) 1 (1904), 388-392.
\bibitem[Co]{Co} Conrad, Keith, {\it``Groups of order $p^3$''}, online notes available at https://kconrad.math.uconn.edu/blurbs/grouptheory/groupsp3.pdf
\bibitem[CR1]{cr1} Curtis C. W.; Reiner I., {\it ``Methods of representation theory vol I, 
With applications to finite groups and orders''}, Pure and Applied Mathematics, A Wiley-Interscience Publication. 
John Wiley \& Sons, Inc., New York, 1981.
\bibitem[CR2]{cr2} Curtis C. W.; Reiner I., {\it ``Methods of representation theory vol II, 
With applications to finite groups and orders''}, Pure and Applied Mathematics, A Wiley-Interscience Publication. 
John Wiley \& Sons, Inc., New York, 1987.
\bibitem[CR3]{cr3} Curtis C. W.; Reiner I., {\it ``Representation theory of finite groups and associative algebras''},
Pure and Applied Mathematics, Vol. XI Interscience Publishers, a division of John Wiley \& Sons, New York-London 1962. 
\bibitem[D1]{d1}  Dornhoff L., {\it``Group representation theory. Part A: Ordinary representation theory''}, 
Pure and Applied Mathematics, 7. Marcel Dekker, Inc., New York, 1971.
\bibitem[D2]{d2} Dornhoff, L., {\it ``Group representation theory. Part B: Modular representation 
theory''}. 
Pure and Applied Mathematics, 7. Marcel Dekker, Inc., New York, 1972.
\bibitem[DF]{df} Dummit D. S.; Foote R. M. {\it ``Abstract algebra''}, Third edition, John Wiley \& 
Sons, Inc., 
Hoboken, NJ, 2004.
\bibitem[FH]{fh} Fulton; Harris, {\it``Representation theory: A first course''} 
Graduate Texts in Mathematics, 129, Readings in Mathematics, Springer-Verlag, New York, 1991.
\bibitem[Gr]{Gr} Grove, Larry C., {\it``Groups and characters''}, Pure and Applied Mathematics (New 
York). A Wiley-Interscience Publication. John Wiley \& Sons, Inc., New York, 1997. x+212. 
\bibitem[M]{m} Musili C. S., {\it ``Representations of finite groups''} Texts and 
Readings in Mathematics, Hindustan Book Agency, Delhi, 1993.
\bibitem[Sc]{sc} P. Schneider, {\it``Modular representation theory of finite groups''}. Springer, 
Dordrecht, 2013.
\bibitem[FT]{ft} Feit, Walter; Thompson, John G., {\it ``Solvability of groups of odd 
order''} Pacific J. Math. 13 (1963), 775-1029. 
\bibitem[L]{l} Lang, {\it ``Algebra''}, Second edition, 
Addison-Wesley Publishing Company, Advanced Book Program, Reading, MA, 1984.
\bibitem[JL]{jl} James, Gordon; Liebeck, Martin, {\it ``Representations and characters of groups''}, 
Second 
edition, Cambridge University Press, New York, 2001.
\bibitem[Se]{se} Serre J.P., {\it ``Linear representations of finite groups"} Translated from the 
second French edition by Leonard L. Scott, Graduate Texts in Mathematics, Vol. 42, Springer-Verlag, New York-Heidelberg, 1977..
\bibitem[Si]{si} Simon, B., {\it ``Representations of finite and compact groups''} Graduate Studies in 
Mathematics, 10, American Mathematical Society, Providence, RI, 1996.
\bibitem[St]{st} B. Steinberg, {\it``Representation theory of finite groups. An introductory approach''}. Universitext. Springer, New York, 2012.

\end{thebibliography}
\end{document}